\documentclass[12pt,a4paper]{amsart}
\usepackage[utf8]{inputenc}
\usepackage[english]{babel}
\usepackage[a4paper,inner=2.8cm,outer=2.5cm,top=2.8cm,bottom=2.5cm]{geometry}               
\usepackage{amsmath,amssymb,amscd,amsthm,amsfonts,anyfontsize,color,enumerate,fix-cm,layout,lipsum,lpic,mathrsfs,mdwlist,stmaryrd,tensor,tikz,thmtools,xspace,comment}
\usepackage{pdfpages}
\usepackage{graphicx}
\usepackage{hyperref}
\usepackage[all]{xy}
\usepackage{mathtools}
\usepackage[bottom]{footmisc}
\usepackage{ifthen}
\usepackage{xfrac}

\theoremstyle{plain}                 
\newtheorem{theorem}{Theorem}[section]     
\newtheorem{proposition}[theorem]{Proposition}
\newtheorem{corollary}[theorem]{Corollary}

\newtheorem{lemma}[theorem]{Lemma}    
\theoremstyle{definition}           
\newtheorem{definition}[theorem]{Definition}   
\newtheorem{example}[theorem]{Example}
\newtheorem{algorithm}[theorem]{Algorithm}
\theoremstyle{remark}       
\newtheorem{remark}[theorem]{Remark} 

\hypersetup{colorlinks=true,linkcolor=blue,citecolor=purple,filecolor=magenta,urlcolor=cyan}

\def\bea{\begin{eqnarray}}
\def\eea{\end{eqnarray}}

\begin{document}
\title[Symplectic geometry of electrical networks]{Symplectic geometry of electrical networks}

\author[B.~Bychkov]{B.~Bychkov}
\address{B.~B.: Department of Mathematics, University of Haifa, Mount Carmel, 3498838, Haifa, Israel}
\email{bbychkov@hse.ru}

\author[V. Gorbounov]{V.~Gorbounov}
\address{V.~G.: Faculty of Mathematics, National Research University Higher School of Economics, Usacheva 6, 119048 Moscow, Russia}
\email{vgorb10@gmail.com }

\author[L. Guterman]{L.~Guterman}
\address{L.~G.: Einstein Institute of Mathematics, Hebrew University, Jerusalem 91904, Israel.}
\email{lazar.guterman@mail.huji.ac.il}

\author[A. Kazakov]{A.~Kazakov}
\address{A.~K.:  Lomonosov Moscow State University, Faculty of Mechanics and Mathematics, Russia, 119991, Moscow, GSP-1, 1 Leninskiye Gory, Main Building; Centre of Integrable Systems, P. G. Demidov Yaroslavl State University, Sovetskaya 14, 150003, Yaroslavl, Russia;
Center of Pure Mathematics, Moscow Institute of Physics and Technology, 9 Institutskiy per., Dolgoprudny, Moscow Region, 141701, Russian Federation;
Kazan Federal University, N.I. Lobachevsky Institute of Mathematics and Mechanics,  Kazan, 420008, Russia}
\email{anton.kazakov.4@mail.ru}

\begin{abstract}
In this paper we relate a well-known in symplectic geometry compactification of the space of symmetric bilinear forms considered as a chart of the Lagrangian Grassmannian to the specific compactifications of the space of electrical networks in the disc obtained in \cite{L}, \cite{CGS} and \cite{BGKT}. In particular, we state an explicit connection between these works and describe some of the combinatorics developed there in the language of symplectic geometry.
We also show that the combinatorics of the concordance vectors forces the uniqueness of the symplectic form, such that corresponding points of the Grassmannian are isotropic. We define a notion of Lagrangian concordance which provides a construction of the compactification of the space of electrical networks in the positive part of the Lagrangian Grassmannian bypassing the construction from \cite{L}.
\end{abstract}
\maketitle

\tableofcontents

\hspace{0.1in}

\section{Introduction}
Thomas Lam constructed an embedding of the space $E_n$ of electrical networks with $n$ boundary points into a subspace of the intersection of the positive Grassmannian $Gr(n-1,2n)$ and a projectivization of a certain subspace $H$ \cite{L}. Examining the closure of the image of this embedding led him to the definition of the compactification of $E_n$ which extends the above embedding to the closure of the image.
Later it was shown in \cite{CGS} and \cite{BGKT} that the intersection $Gr(n-1,n)\cap\mathbb PH$ is isomorphic to the Lagrangian Grassmannian $LG(n-1)$. From now on we will use the term embedding keeping in mind that it can be continued to the compactification via Lam's construction.

We begin this paper by explaining that using symplectic geometry one obtains a family of compactifications of the space $E_n$ in $LG(n-1)$ and one can tailor the choice of an element in this family to a specific task. In particular, one can wish to parametrize the positive part of $LG(n-1)$ by electrical networks. We will call these compactifications {\it positive}. The compactifications constructed in \cite{CGS} and \cite{BGKT} are meant to be positive but it is yet to be proved. 
These compactifications look different and our next goal in this paper is to state the precise connection between them.

In \cite{CGS} it was shown that there is a degenerate 2-form $\widetilde{\Lambda}_{2n}$ in the ambient space $\mathbb R^{2n}$ such that the subspaces of dimension $n+1$ that form the image of the space of electrical networks under the Lam map are isotropic for $\widetilde\Lambda_{2n}$. The variety $LG(n-1)$ appears as the quotient of these isotropic subspaces by the kernel of $\widetilde\Lambda_{2n}$.

In the approach developed in \cite{BGKT} it was shown that there is a symplectic form $\overline{\Lambda}_{2n}$ in the ambient space such that the collection of subspaces of dimension $n-1$ that form the image of the space of electrical networks under the Lam map are isotropic for $\overline{\Lambda}_{2n}$. In fact everything is happening in a certain subspace $V\subset \mathbb R^{2n}$ of codimension two. Namely it turns out that $H\subset \bigwedge^{n-1} V$, $\Lambda$ restricts to a symplectic form $\Lambda_{2n-2}$ on $V$ and $Gr(n-1,2n)\cap \mathbb PH=LG(n-1,V)$ is the closed orbit of a certain vector in $\bigwedge^{n-1} V$ under the action of the symplectic group $Sp(V)$ for the 2-form $\Lambda_{2n-2}$.
The subspace $H$ becomes the space of the fundamental representation of $Sp(V)\cong Sp(2n-2)$ which corresponds to the last vertex of the Dynkin diagram. 

The explicit connection between these compactifiactions can be stated in two ways.

On the one hand the construction of a compactification from \cite{L} uses a bipartite graph obtained by the Temperley's trick from the graph representing an electrical network. One needs to make a non canonical choice of colors for the vertices of this bipartite graph and there are exactly two ways to do this. One choice leads to the construction from \cite{L} and \cite{BGKT}. We prove in Corollary \ref{cor:maincor} that swapping the colors of the vertices gives a compactification from \cite{CGS}.

On the other hand each isotropic for $\overline{\Lambda}_{2n}$ subspace $E$ of dimension $n-1$ has the symplectic orthogonal $E^{\bot}$ of dimension $n+1$ such that $E\subset E^{\bot}$. The subspaces of complementatry dimensions from \cite{CGS} and \cite{BGKT} are symplectic orthogonals to each other up to a fixed linear isomorphism of the ambient space. We discuss this in Section \ref{sec:linearalgebra}.

The subspace $H$ comes equipped with a remarkable basis. The basis vectors are so called {\it concordance vectors} and are defined via combinatorial properties of non-crossing matchings on the boundary vertices. 
Our next result, Theorem \ref{Classification}, says that this combinatorics forces the existence of a unique up to a scalar factor symplectic form such that the kernel of the convolution with it on $\bigwedge^{n-1} V$ is spanned by the concordance vectors. This form up to a scalar is the form $\Lambda_{2n-2}$ described above. Moreover in the course of proving this theorem we found a combinatorial description of the concordance vectors in terms of basis of the space $V$. We call this combinatorial description {\itshape Lagrangian concordance}. 
It provides a direct construction (Algorithm \ref{Algoritm_embedding}) of the compactification of $E_n$ on $LG(n-1)$ bypassing the construction from \cite{L}.

The electrical Lie algebra which is isomorphic to the symplectic Lie algebra $\mathfrak{sp}_{2n-2}$ is given in terms of the Lam generators $\mathfrak{u}_i$ and the relations between them.
Although $\mathfrak{u}_i$ are not the Chevalley generators they posses some of their properties. In particular the concordance basis is very special with respect to the action of the electrical Lie algebra: the generators $\mathfrak{u}_i$ map a concordance vector to a concordance vector or to zero, see Theorem \ref{cristal V}.

In a recent paper \cite{GLX} the authors studied the {\it grove algebra}, a quadratic algebra generated by the grove partition functions of electrical networks with the properties similar to the properties of the projective algebra of the Grassmannian. Our Theorem \ref{invariance of g_i} implies a nice action of $\mathfrak{sp}_{2n-2}$ on the set of grove coordinates. It might be useful for studying the grove algebra.

\subsection{Organization of the paper}

In Section \ref{sec:background} we present the theory of electrical networks in the form we will need it in this paper and describe the connection of the space of electrical networks to the Lagrangian Grassmannian.

In Section \ref{sec:main1} we compare the embeddings from \cite{CGS} and \cite{L}, \cite{BGKT} via symplectic geometry and via combinatorics. In Proposition \ref{Ortogonality of embeddings} we prove that the subspaces from \cite{CGS} and \cite{BGKT} corresponding to an electrical network are orthogonal to each other w.r.t. the standard bilinear form up to a fixed scalar matrix. In Corollary \ref{cor: inclusion} we prove that these subspaces are $\overline{\Lambda}_{2n}$-symplectic ortogonals to each other. 

Section \ref{th-comb} is devoted to a comparison via combinatorics. We first define a dual generalized Temperley's trick. We then find in Theorem \ref{thm:explicitcoord} a point of the Grassmannian $Gr_{\ge0}(n+1,2n)$ which represents the Postnikov's network associated to an electrical network via dual generalized Temperley's trick. In Corollary \ref{cor:maincor} we conclude that this point is coincide with one defined in \cite{CGS}.

In Section \ref{sec:main2} we relate the combinatorics of concordance with symplectic geometry. In Section \ref{sec: concordance property and symplectic forms} we proof Theorem \ref{Classification} which explains how combinatorics of the concordance forces the existence of unique, up to a scalar, form such that the kernel of a convolution with it on $\bigwedge^{n-1}V$ is spanned by the concordance vectors. The intermediate step in proving this theorem is Algorithm \ref{Algoritm_embedding} which gives an explicit procedure to express a concordance vector as an element of $\bigwedge^{n-1}V$. We finish this subsection relating concordance vectors to the description of Plucker coordinates of hollow cactus networks in Proposition \ref{empty cactus maps to p}. 

In Section \ref{main proofs} we prove auxiliary statements in the course of proving Theorem \ref{Classification}, which are Algorithm \ref{Algoritm_embedding} and Lemma \ref{diagonal}. 

In Section \ref{subsec: lagrangian concordance} we prove Theorem \ref{theorem:cactus-plu} describing the coefficients of a concordant vector as elements of $\bigwedge^{n-1}V$ and introduce the notion of a Lagrangian concordance in Definition \ref{def: Lagrangian concordance}, which provides a combinatorial way to describe a nonzero Plucker coordinates of a concordant vector as elements of $\bigwedge^{n-1}V$.

In Section \ref{sec: special property} we prove that the action of Lam's generators on the basis of $H$, consisting of concordance vectors, possess a remarkable property that is stated in Theorem \ref{cristal V}.

In Section \ref{sec:cactus} we discuss parametrization of electrical networks by effective resistances matrices and provide a different than Algorithm \ref{Algoritm_embedding} way to write an expression of concordance vectors as elements of $\bigwedge^{n-1}V$. 

In Section \ref{sec: appendix} we present a needed in Section \ref{th-comb} background on generalized Temperley's trick and Postnikov's networks.

\subsection{Acknowledgments}
Research of B.~B. was partially supported by the ISF grant 876/20.
For V.~G. this article is an output of a research project implemented as part of the Basic Research Program at the National Research University Higher School of Economics (HSE University). Working on this project V.~G. also visited the Max Plank Institute of Mathematics in Bonn, Germany in the summer 2023 and BIMSA in Beijing, China in the summer 2024. 
L.~G. was partially supported by the Israel Science Foundation grant ISF-2480/20. 
Research of A.~K. and B.~B. on Sections \ref{sec:main1} and \ref{sec:main2} was supported by the Russian Science Foundation project No. 20-71-10110 (https://rscf.ru/en/project/23-71-50012) which finances the work of A.K. at P. G. Demidov Yaroslavl State University. Research of A.~K. on Section \ref{sec:cactus} was supported by the state assignment of MIPT (project FSMG-2023-0013).
We are greateful to Boris Feigin for useful discussions. We also thank the anonymous referee for valuable comments.

\section{The space of electrical networks and its natural compactifications}\label{sec:background}
\subsection{Electrical networks} 
We start with a brief introduction into the theory of electrical networks following \cite{CIM} and \cite{CGV}.
An {\itshape electrical network} is a finite weighted undirected graph $\Gamma$ embedded into a disk, where the vertex set is divided into the set of clockwise enumerated boundary vertices and the set of interior vertices. The weight $w(e)$ of an edge is to be thought of a conductance of the corresponding edge. We will think of electrical current as a flow along edges, and each vertex has a voltage. These physical objects should satisfy two laws. Kirchhoff's law says that for any interior vertex the total current through all adjacent edges equals zero. Ohm's law says that if $e=(u,v)$ is an edge then we have 
$$I(e)=w(e)(U(u)-U(v))$$
where $w(e)$ is the conductance of the edge, $U(u),\ U(v)$ are the voltages of the vertices $u,\ v$ respectively and $I$ is the current from $u$ to $v$. We think of boundary vertices as of nodes to which a voltage is applied. It follows from the maximum principle for the discrete harmonic functions \cite[Chapter 3.2]{CIW} that fixed voltages on the boundary vertices give rise to the unique extension on all vertices. Voltages on the vertices lead to the electrical current flowing along the edges. From Kirchhoff's law follows that to determine electrical current on any edge of the graph $\Gamma$, it is enough to know electrical currents outgoing from the boundary vertices. Therefore to completely determine physical properties of an electrical network it is enough to know the map sending a vector $v\in \mathbb R^{n}$ of boundary voltages to a vector of boundary currents. This map is linear and its matrix is called the {\itshape response matrix} $M_R$ which can be determined in terms of the network as it is proved in \cite[Theorem 3.2]{CIM}. Notice that two electrical networks share the same  response matrix if and only if they can be transformed to each other with electrical transformations , see \cite[Section 2.4]{L}.

We denote an electrical network by $e(\Gamma,\omega)$ or simply $e$, where $\Gamma$ is an underlying graph and $\omega:\text{vertex set of }\Gamma\to\mathbb R_{>0}$ is a weight (conductance) function. We denote by $E_n$ the set of electrical networks on $n$ boundary vertices up to the  electrical transformations.

\subsection{Compactification of the spaces of bilinear forms and electrical networks}
This section contains some standard facts which can be found for example in \cite{AG}.

A {\itshape symplectic} vector space is a vector space equipped with a symplectic bilinear form. An {\itshape isotropic} subspace of a symplectic vector space is a vector subspace on which the symplectic form vanishes. A maximal isotropic subspace is called a {\itshape Lagrangian} subspace. For a symplectic vector space $V$, the {\itshape Lagrangian Grassmannian} $LG(n, V )$ is the space of Lagrangian subspaces. If it is clear which symplectic space we are working with we will denote the Lagrangian Grassmannian $LG(n)$. For a degenerate skew-symmetric form $\Lambda$ we denote by $IG^\Lambda(k,n)\subset Gr(k,n)$ the set of isotropic subspaces w.r.t. $\Lambda$.

Below we work with a specific symplectic space 
$$W_n = \mathbb R^n \oplus (\mathbb R^n )^*,$$ 
where $(\mathbb R^n )^*$ is the dual space of $\mathbb R^n$, equipped with the symplectic form:

\[\Lambda((x,\psi),(x',\psi')) = \psi'(x) - \psi(x'),\]
for any $(x,\psi)$ and $(x',\psi')$ in $W_n$. 

The set of all Lagrangian subspaces of $W_n$ is the set of points of an algebraic variety $LG(n)$, called the Lagrangian Grassmannian. 
Let $H$ be a hyperspace in $W_n$ and $H^\perp$ be its orthogonal with respect to the symplectic form $\Lambda$. It is a standard exercise in the linear symplectic geometry to show that $\mathrm{dim}\, H^\perp=1$, $H^\perp\subset H$.
 If $L$ is a Lagrangian subspace in $W_n$ then $H^\perp$ is either a subspace in $L$ or $\mathrm{dim}\, H^\perp\cap L = 0.$ In the first case $L$ must be a subspace of $H$ since the form $\Lambda$ is non degenerate, in the second case $\mathrm{dim}\, L\cap H=n-1$ since for non degenerate the maximal dimension of an isotropic subspace is $n$. Therefore in both these cases
\[\mathrm{dim} \,(L\cap H)/H^\perp=n-1.\]
Moreover $H/H^\perp$ inherits a symplectic form and the subspace $(L\cap H)/H^\perp$ is Lagrangian in $H/H^\perp$. Summarizing, any hyperspace $H\subset W_n$ defines a surjection 

\[f_H\colon LG(n)\rightarrow LG(n-1).\]

Denote by $Sym_n$ the space of $n\times n$ symmetric matrices with the coefficients in $\mathbb R$. 
If $Q$ is in $Sym_n$ , then we can define the subspace

\[LQ = Span\{e_i + (\sum_{j=1}^n Q_{i,j} e^*_j )\}_{i=1,\ldots ,n}\]
where $(e_i)$ is the canonical basis of $\mathbb R^n$, and $(e^*_i )$ the dual basis. $LQ$ is a Lagrangian subspace of $W_n$, it is just the graph of the linear map defined by the matrix $Q$. This correspondence defines a map from
$Sym_n$ to the algebraic variety $LG(n)$. This map is a standard description of a chart in $LG(n)$, the dimension of $LG(n)$ therefore is $(n^2+n)/2$.

In particular an electrical network in a disk with $n$ boundary points can be considered as a Lagrangian subspace. Indeed an electrical network $e(\Gamma,\omega)\in E_n$ has the response matrix $M_R$ which is symmetric $n\times n$ matrix hence the above construction applies. However there is a subtlety here, the matrix $M_R$ is always degenerate and therefore the image of the space $E_n$ is a subvariety of positive codimension in $LG(n)$.
Namely if we fixed a Lagrangian subspace in $W_n$, say $\mathbb R^n\oplus {\bf 0}$, then the image of the elements of $E_n$
lands in the submanifold of $LG(n)$ of the Lagrangian subspaces $L$ such that
\[L\cap \mathbb R^n\oplus {\bf 0}\not={\bf 0}\]
For a critical electrical network \cite{CIM} the rank of the response matrix $M_R$ is $n-1$ and the kernel $K$ of the linear map defined by $M_R$ is generated by the vector ${\bf v}$ whose all coordinates are equal to each other. Splitting $W_n$ into a direct sum
\[W_n=K\oplus H\]
for any such a hyperspace $H$ we naturally obtain a map from the space $E_n$ into $LG(n-1)$ according to the above construction. Since the dimensions of these spaces coincide we obtain a compactification of $E_n$. 
Choosing a particular complement subspace $H$ to the subspace $K$ and a particular basis of $H$ one can tailor the compactification of $E_n$ to specific tasks, for example, to use $E_n$ to parametrize the positive part of $LG(n-1).$

\subsection{Positive compactifications of electrical networks} In this section we will define the compactification of $E_n$ obtained in \cite{L}, \cite{CGS} and \cite{BGKT}.

\begin{definition}
A {\itshape non-crossing partition} $\sigma$ is a set partition of $[\bar n]:=\{\bar1,\bar2,\ldots,\bar n\}$ such that if $[\bar n]$ is arranged on a circle 
the convex hulls of its blocks are pairwise disjoint.

Equivalently, $\sigma$ is non-crossing if there is no $i<j<k<l$ with $ik$ and $jl$ belonging to different components of $\sigma$.
\end{definition}
Non-crossing partitions can also be defined given a graph embedded into a disk.

\begin{definition} \label{groves, ncp}
A {\itshape grove} $F$ on $\Gamma$ is a spanning subforest, that is an acyclic subgraph that uses all vertices, such that each connected component $F_i\subset F$ contains boundary vertices. The {\itshape boundary partition} $\sigma(F)$ is the set partition of the set $\{\bar{1},\bar{2},\ldots,\bar{n}\}$ specifies which boundary vertices lie in the same connected component of $F$. Note that since $\Gamma$ is planar, $\sigma(F)$ must be a {\itshape non-crossing partition}, also called a {\itshape planar set partition}.

We will often write set partitions in the form $\sigma=(\bar{a},\bar{b},\bar{c}|\bar{d},\bar{e}|\bar{f},\bar{g}|\bar{h})$. Denote by $\mathcal{NC}_n$ the
set of non-crossing partitions on the set $\{\bar{1},\bar{2},\ldots,\bar{n}\}$.
Each non-crossing partition $\sigma$ on the set $\{\bar{1},\bar{2},\ldots,\bar{n}\}$ has a dual non-crossing partition on $\{\widetilde{1},\widetilde{2},\ldots,\widetilde{n}\}$
 where by convention  $\widetilde{i}$ lies between  $\bar{i}$ and $\overline{i+1}$ and $\widetilde{n}$ lies between $\overline{n}$ and $\overline{1}$. For example, $(\widetilde{1},\widetilde{3}|\widetilde{2}|\widetilde{4},\widetilde{5}|\widetilde{6})$ is dual to $(\bar{1},\bar{4},\bar{6}|\bar{2},\bar{3}|\bar{5})$.

 We will often consider the non-crossing partition of $[2n]$ defined by merging the initial partition $\sigma$ and it's dual $\widetilde\sigma$, where we will use an identification $\{1,2,\ldots,2n-1,2n\} =\{\overline{1},\widetilde{1},\overline{2},\widetilde{2},\ldots,\overline{n},\widetilde{n}\}$. We will denote this partition $(\sigma|\widetilde\sigma)$ (see Fig. \ref{non crossing partition and its dual}).
\end{definition}

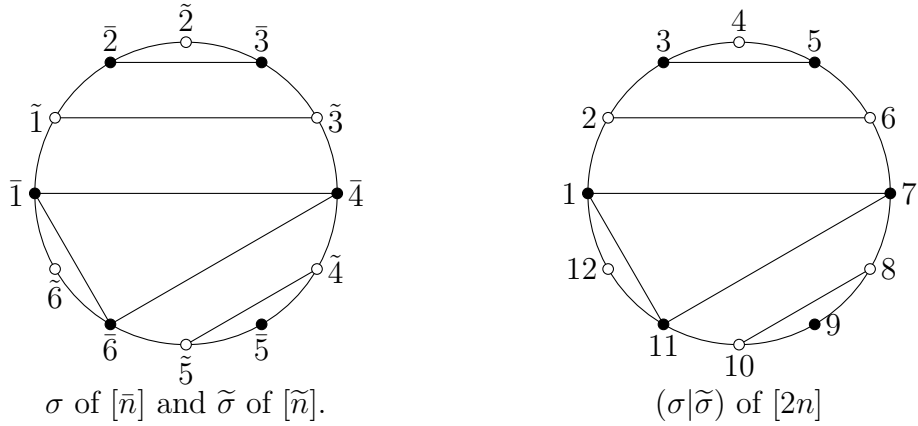
\begin{figure}[ht]
\centering
\begin{tikzpicture}
    \draw (0,0) circle (2);

    \draw (0:2) -- node [left] {}(180:2);
    \draw (30:2) -- node [left] {}(150:2);
    \draw (60:2) -- node [left] {}(120:2);
    \draw (0:2) -- node [left] {}(-120:2);
    \draw (-120:2) -- node [left] {}(180:2);
    \draw (-30:2) -- node [left] {}(-90:2);

    \filldraw[fill=black] (0:2) node [right] {$\bar4$} circle (2pt);
    \filldraw[fill=white] (30:2) node [right] {$\tilde3$} circle (2pt);
    \filldraw[fill=black] (60:2) node [above] {$\bar3$} circle (2pt);
    \filldraw[fill=white] (90:2) node [above] {$\tilde2$} circle (2pt);
    \filldraw[fill=black] (120:2) node [above] {$\bar2$} circle (2pt);
    \filldraw[fill=white] (150:2) node [left] {$\tilde1$} circle (2pt);
    \filldraw[fill=black] (180:2) node [left] {$\bar1$} circle (2pt);
    \filldraw[fill=white] (-30:2) node [right] {$\tilde4$} circle (2pt);
    \filldraw[fill=black] (-60:2) node [below] {$\bar5$} circle (2pt);
    \filldraw[fill=white] (-90:2) node [below] {$\tilde5$} circle (2pt);
    \filldraw[fill=black] (-120:2) node [below] {$\bar6$} circle (2pt);
    \filldraw[fill=white] (-150:2) node [below] {$\tilde6$} circle (2pt);

    \draw (0,-2.8) node {$\sigma$ of $[\bar n]$ and $\widetilde{\sigma}$ of $[\widetilde{n}]$.};
\end{tikzpicture} \phantom{aaaaaaaaaa}
\begin{tikzpicture}
    \draw (0,0) circle (2);

    \draw (0:2) -- node [left] {}(180:2);
    \draw (30:2) -- node [left] {}(150:2);
    \draw (60:2) -- node [left] {}(120:2);
    \draw (0:2) -- node [left] {}(-120:2);
    \draw (-120:2) -- node [left] {}(180:2);
    \draw (-30:2) -- node [left] {}(-90:2);

    \filldraw[fill=black] (0:2) node [right] {$7$} circle (2pt);
    \filldraw[fill=white] (30:2) node [right] {$6$} circle (2pt);
    \filldraw[fill=black] (60:2) node [above] {$5$} circle (2pt);
    \filldraw[fill=white] (90:2) node [above] {$4$} circle (2pt);
    \filldraw[fill=black] (120:2) node [above] {$3$} circle (2pt);
    \filldraw[fill=white] (150:2) node [left] {$2$} circle (2pt);
    \filldraw[fill=black] (180:2) node [left] {$1$} circle (2pt);
    \filldraw[fill=white] (-30:2) node [right] {$8$} circle (2pt);
    \filldraw[fill=black] (-60:2) node [right] {$9$} circle (2pt);
    \filldraw[fill=white] (-90:2) node [below] {$10$} circle (2pt);
    \filldraw[fill=black] (-120:2) node [below] {$11$} circle (2pt);
    \filldraw[fill=white] (-150:2) node [left] {$12$} circle (2pt);

    \draw (0,-2.8) node {$(\sigma|\widetilde{\sigma})$ of $[2n]$};
\end{tikzpicture}
    \caption{Non-crossing partition $\sigma=(\bar1,\bar4,\bar6|\bar2,\bar3|\bar5)$ with its dual non-crossing partition  $\widetilde{\sigma}=(\tilde1,\tilde3|\tilde2|\tilde4,\tilde5|\tilde6)$ and merged partition $(\sigma|\widetilde{\sigma})$.}
    \label{non crossing partition and its dual}
\end{figure}

\begin{definition}\label{def:hollow}
Let $S$ be a circle with $n$ boundary points labeled by $1,\ldots, n$. Let $\sigma\in\mathcal{NC}_n$, identifying the boundary points according to the parts of $\sigma$ gives a hollow cactus $S_{\sigma}$. It is a union of circles glued together at the identified points. The interior of $S_{\sigma}$, together with $S_{\sigma}$ itself is called a cactus. A cactus network is a weighted graph $\Gamma$ embedded into a cactus. In other words, we might think about cactus networks as the networks which are obtained by contracting the set of edges between the boundary vertices defined by $\sigma\in\mathcal{NC}_n$ which have infinite conductivity.
\end{definition}

The motivation to introduce cactus networks was the fact that the image of the set of cactus networks coincides with the closure of the set of electrical networks under the embedding into a certain projective space, see \cite[Section 4]{L} for more details.

Recall that we denote by $E_n$ the set of electrical networks with $n$ boundary vertices up to the  electrical transformations, see \cite[Section 2.4]{L}. We denote by $\overline{E}_n$ the corresponding set of cactus networks up to the  electrical transformations.

\begin{definition}
The {\itshape totally nonnegative Grassmannian} is the locus in the Grassmannian with the Plucker coordinates of the same sign.
\end{definition}

 We will now describe the embedding of the set of electrical networks into the totally nonnegative Grassmannian $Gr(n-1,2n)_{\ge 0}$ from \cite{L}.

\begin{definition}
For a non-crossing partition $\sigma$ we define the {\itshape grove measurement} related to $\sigma$ as follows
$$L_{\sigma(\Gamma)}:=\sum\limits_{\{F|\sigma(F)=\sigma\}}wt(F),$$
where the summation is over all groves with boundary partition $\sigma$, and $wt(F)$ is the product of the weights of the edges in $F$. By $L_{unc}$ we will denote the grove measurement of the type $L_{*|*|\ldots|*}$.
\end{definition}

\begin{definition} \label{def-ordcon}
We call an $(n-1)$-element subset $I\subset\{1,\ldots,2n\}$ {\itshape concordant} with a non-crossing partition $\sigma$ if each part of $\sigma$ and each part of the dual partition $\widetilde{\sigma}$ contains exactly one element not in $I$. In this situation we also say that $\sigma$ or $(\sigma,\widetilde{\sigma})$ is concordant with $I$.
\end{definition}

\begin{example}
Let $\sigma=(\bar{1},\bar{4},\bar{6}|\bar{2},\bar{3}|\bar{5})$ be such that $\widetilde{\sigma}=(\widetilde{1},\widetilde{3}|\widetilde{2}|\widetilde{4},\widetilde{5}|\widetilde{6})$. Then $\sigma$ is concordant with $\{2,5,7,8,11\}$ but not concordant with $\{2,5,7,8,12\}$, see Fig. \ref{concordance example}. 
Recall that we use an identification $\{1,2,\ldots,12\} =\{\overline{1},\widetilde{1},\overline{2},\widetilde{2},\ldots,\overline{6},\widetilde{6}\}$.
\end{example}

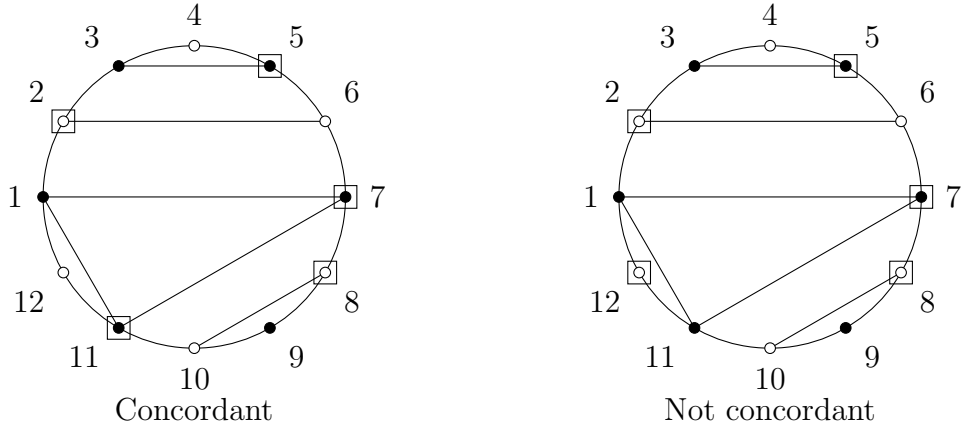
\begin{figure}[ht]
\centering
\begin{tikzpicture}
    \draw (0,0) circle (2);

    \draw (0:2) -- node [left] {}(180:2);
    \draw (30:2) -- node [left] {}(150:2);
    \draw (60:2) -- node [left] {}(120:2);
    \draw (0:2) -- node [left] {}(-120:2);
    \draw (-120:2) -- node [left] {}(180:2);
    \draw (-30:2) -- node [left] {}(-90:2);

    \filldraw[fill=black] (0:2) node [right=5pt] {$7$} circle (2pt);
    \filldraw[fill=white] (30:2) node [above right=3pt] {$6$} circle (2pt);
    \filldraw[fill=black] (60:2) node [above right=3pt] {$5$} circle (2pt);
    \filldraw[fill=white] (90:2) node [above=3.5pt] {$4$} circle (2pt);
    \filldraw[fill=black] (120:2) node [above left=3pt] {$3$} circle (2pt);
    \filldraw[fill=white] (150:2) node [above left=3pt] {$2$} circle (2pt);
    \filldraw[fill=black] (180:2) node [left=3.5pt] {$1$} circle (2pt);
    \filldraw[fill=white] (-30:2) node [below right=3pt] {$8$} circle (2pt);
    \filldraw[fill=black] (-60:2) node [below right=3pt] {$9$} circle (2pt);
    \filldraw[fill=white] (-90:2) node [below=3.5pt] {$10$} circle (2pt);
    \filldraw[fill=black] (-120:2) node [below left=3pt] {$11$} circle (2pt);
    \filldraw[fill=white] (-150:2) node [below left=3pt] {$12$} circle (2pt);

    \draw (0:2) +(-0.15,-0.15) rectangle +(0.15,0.15);
    \draw (-30:2) +(-0.15,-0.15) rectangle +(0.15,0.15);
    \draw (-120:2) +(-0.15,-0.15) rectangle +(0.15,0.15);
    \draw (150:2) +(-0.15,-0.15) rectangle +(0.15,0.15);
    \draw (60:2) +(-0.15,-0.15) rectangle +(0.15,0.15);

    \draw (0,-2.8) node {Concordant};
\end{tikzpicture} \phantom{aaaaaaaaaa}
\begin{tikzpicture}
    \draw (0,0) circle (2);

    \draw (0:2) -- node [left] {}(180:2);
    \draw (30:2) -- node [left] {}(150:2);
    \draw (60:2) -- node [left] {}(120:2);
    \draw (0:2) -- node [left] {}(-120:2);
    \draw (-120:2) -- node [left] {}(180:2);
    \draw (-30:2) -- node [left] {}(-90:2);

    \filldraw[fill=black] (0:2) node [right=5pt] {$7$} circle (2pt);
    \filldraw[fill=white] (30:2) node [above right=3pt] {$6$} circle (2pt);
    \filldraw[fill=black] (60:2) node [above right=3pt] {$5$} circle (2pt);
    \filldraw[fill=white] (90:2) node [above=3.5pt] {$4$} circle (2pt);
    \filldraw[fill=black] (120:2) node [above left=3pt] {$3$} circle (2pt);
    \filldraw[fill=white] (150:2) node [above left=3pt] {$2$} circle (2pt);
    \filldraw[fill=black] (180:2) node [left=3.5pt] {$1$} circle (2pt);
    \filldraw[fill=white] (-30:2) node [below right=3pt] {$8$} circle (2pt);
    \filldraw[fill=black] (-60:2) node [below right=3pt] {$9$} circle (2pt);
    \filldraw[fill=white] (-90:2) node [below=3.5pt] {$10$} circle (2pt);
    \filldraw[fill=black] (-120:2) node [below left=3pt] {$11$} circle (2pt);
    \filldraw[fill=white] (-150:2) node [below left=3pt] {$12$} circle (2pt);

    \draw (0:2) +(-0.15,-0.15) rectangle +(0.15,0.15);
    \draw (60:2) +(-0.15,-0.15) rectangle +(0.15,0.15);
    \draw (150:2) +(-0.15,-0.15) rectangle +(0.15,0.15);
    \draw (-150:2) +(-0.15,-0.15) rectangle +(0.15,0.15);
    \draw (-30:2) +(-0.15,-0.15) rectangle +(0.15,0.15);

    \draw (0,-2.8) node {Not concordant};
\end{tikzpicture}
    \caption{The non-crossing partition $\sigma=(\bar1,\bar4,\bar6|\bar2,\bar3|\bar5)$ is concordant with $\{2,5,7,8,11\}$, but not concordant with $\{2,5,7,8,12\}$}
    \label{concordance example}
\end{figure}

\begin{theorem} \textup{\cite[Theorem 5.10]{L}} \label{Emb}
Let $e\in \overline{E}_n$. Then the collection of boundary measurements 
$$\sum\limits_{(\sigma,I)}L_{\sigma},$$
here summation is over all $\sigma$, which are concordant with $I$, considered as a set of Plucker coordinates, defines a point in the totally non-neggative Grassmannian $Gr_{\geq 0}(n-1,2n)$.
\end{theorem}
We will denote the sum $\sum\limits_{(\sigma,I)}L_{\sigma}$ as $\Delta_I^\bullet$.
\begin{example} \label{bulcon-ex}
  By Definition \ref{def-ordcon} for any $n$  the following holds:
  \begin{itemize}
      \item $\Delta_{246\dots2n-2}^\bullet=L_{unc};$
      \item $\Delta_{135\dots 2n-3}^\bullet=L_{123\dots n}.$
  \end{itemize} 
\end{example}
\begin{remark} \label{Plucker coord as dimer partition function}
    The variable $\Delta_I^\bullet$ has an independent combinatorial definition as the dimer partition function of all almost perfect matchings (see \cite[Section 3.2]{L} and \cite[Section 4.1]{L3}) on a bipartite network $N(\Gamma, \omega)$ (see Section \ref{pbn-ap}) with the boundary $I.$ 
\end{remark}

\begin{definition}\label{def:concordance}
Let $I$ be a $(n-1)$-element subset of $\{1,\ldots,2n\}$ and $\sigma\in \mathcal{NC}_n$

Define a $(\binom{2n}{n-1}\times C_n)$-matrix $A_n=(a_{I\sigma})$, where $I$ is a $(n-1)$-element subset of $\{1,\ldots,2n\}$, $\sigma\in \mathcal{NC}_n$ and $C_n$ is the $n$-th Catalan number as follows:

$$a_{I\sigma}=\begin{cases}
1, \text{\ if\ }  \sigma\ \text{is concordant with}\ I; \\
0, \text{\ otherwise.}
\end{cases}$$

Let the {\itshape concordance space} $H$ be the column space of $A_n$, it is a subspace of $\bigwedge^{n-1}\mathbb{R}^{2n}$ and $\mathrm{dim}\, H=C_n$. Denote by $\mathbb PH$ the projectivization of $H$.

The standard basis vectors of $\mathbb R^{2n}$ we denote by $e_i$. For $I=(i_1,i_2,\ldots,i_{n-1})$ we denote by $e_I$ a standard vector $e_{i_1}\wedge e_{i_2}\wedge\ldots\wedge e_{i_{n-1}}$. Thus vectors of the form $e_I=e_{i_1}\wedge e_{i_2}\wedge \ldots \wedge e_{i_{n-1}}$ where  $i_1<i_2<\ldots<i_{n-1}$ form a basis of the space $\bigwedge^{n-1}\mathbb{R}^{2n}$; and vectors $w_\sigma=\sum\limits_I a_{I \sigma}e_I\in\bigwedge^{n-1}\mathbb{R}^{2n}$ form a basis of the space $H$, we call them {\itshape concordance vectors}.
\end{definition}

\begin{theorem} \textup{\cite[Theorem 5.8]{L}} \label{Image of the Lam's embedding}
The image of the set $\overline{E}_n$ in $Gr(n-1,2n)_{\geq 0}$ is the intersection: 
$$Gr(n-1,2n)_{\geq 0}\cap\mathbb P H.$$
\end{theorem}

Here we identify $Gr(n-1,2n)_{\geq 0}$ and its Plucker embedding. It was proven by Chepuri-George-Speyer \cite[Theorem 1.5]{CGS} and Bychkov-Gorbounov-Kazakov-Talalaev \cite[Theorem 3.3]{BGKT} that in fact the image of $\overline{E}_n$ under the Lam's embedding lies in the totally nonnegative Lagrangian Grassmannian. Now we are going to describe the embedding $\overline{E}_n$ into the totally nonnegative Grassmannian $Gr(n-1,2n)_{\geq 0}$ following \cite{BGKT}. This method allows to refine Theorem \ref{Emb} since it proves that the image of $\overline{E}_n$ lies in the Lagrangian Grassmannian $LG(n-1,V)$ for a certain $2n-2$-dimensional subspace $V$. We will describe the second approach from \cite{CGS} in Section \ref{sec:main1}.
\begin{definition} \label{Omega_n}
Let $e\in E_n$ and denote by $M_R(e)=(x_{ij})$ its response matrix. Define a point in $Gr(n-1,2n)$ associated to $e$ as the row space of the matrix:
$$\Omega_n(e)=\left(\begin{matrix}
x_{11} & 1 & -x_{12} & 0 & x_{13} & 0 & \ldots & (-1)^n \\
-x_{21} & 1 & x_{22} & 1 & -x_{23} & 0 & \ldots & 0 \\
x_{31} & 0 & -x_{32} & 1 & x_{33} & 1 & \ldots & 0 \\
\vdots & \vdots & \vdots & \vdots & \vdots & \vdots & \ddots & \vdots &  
\end{matrix}\right).$$
\end{definition}

Given a matrix $M$ of size $n-1\times 2n$ we will denote by $\Delta_I(M)$ the minor of the matrix $M$ corresponding to the set of columns indexed by the set $I=\{i_1,\ldots,i_{n-1}\}$.

\begin{theorem} \textup{\cite[Theorem 3.3]{BGKT}} \label{Theorem about Omega_n}
The row space of $\Omega_n(e)$ defines the same point in the $Gr_{\geq 0}(n-1,2n)$ as the point defined by $e(\Gamma,\omega)$ under the Lam's embedding. More precisely, the set of maximal minors of the matrix $\Omega'_n(e)$ obtained from $\Omega_n(e)$ by deleting the last row is related to the choice of Plucker coordinates $\Delta_I^\bullet$ in Lam's embedding as $\Delta_I(\Omega_n(e))=\frac{\Delta_I^\bullet}{L_{unc}}$.
\end{theorem}

 Now define the subspace $V$ and the symplectic form $\Lambda$ which will define the Lagrangian subspace for electrical networks. 
$$V=\{v\in \mathbb R^{2n}|\sum\limits_{i=1}^{n}(-1)^iv_{2i}=0,\ \sum\limits_{i=1}^n(-1)^iv_{2i-1}=0\}.$$
Fix a basis for the subspace $V$:
\begin{equation} \label{Basis of V}
v_1=(1,0,1,0,\ldots,0,0,0),\ v_2=(0,1,0,1,\ldots,0,0,0),\ldots,v_{2n-2}=(0,0,0,0,\ldots,1,0,1).
\end{equation}

To make our notations consistent with a combinatorics of the points in Grassmannian corresponding to electrical networks we will use the following definition of the symplectic group:
$$Sp(2n,\mathbb R)=\{M\in Mat_{2n\times 2n}(\mathbb R): M\Lambda M^T=\Lambda\}$$
for any symplectic form $\Lambda$.
It means that we take the symplectic group for the form $\Lambda^{-1}$ in the sense of the standard definition.

\begin{lemma} \textup{\cite[Lemma 4.2]{BGKT}} \label{Rows in V}
All the rows of the matrix $\Omega_n(e)$ belong to the subspace $V$. 
\end{lemma}

Expanding the rows of the matrix $\Omega_n(e)$ in the basis (\ref{Basis of V}) we obtain the matrix $\widetilde{\Omega}_n(e)=\Omega_n(e)B_n^{-1}$, where $B_n$ is the matrix whose rows are vectors $v_i$. This parameterization was extended to the cactus networks as well, see \cite[Section 3.2]{BGKT}.

Fix  the symplectic form on $V$
\begin{equation} \label{8}
\Lambda_{2n-2} = \left(\begin{array}{cccccc}
\phantom{\vdots}0 & 1 & 0 & \cdots &  \cdots & 0   \\
\phantom{\vdots}-1 & 0 & -1 & 0 & & \vdots   \\
\phantom{\vdots}0 & 1 & 0 & 1 & \ddots & \vdots  \\
\vdots & \ddots  & \ddots  & \ddots & \ddots & 0   \\
 \vdots &    & \ddots &  1 & 0 & 1  \\
\phantom{\vdots}0 &  \cdots & \cdots  & 0 &  -1 & 0
\end{array}\right),
\end{equation}
i.e. $\Lambda_{2n-2}$ is a $2n-2\times 2n-2 $ matrix.

\begin{theorem} \label{Embedding into LG} \textup{\cite[Theorem 4.4]{BGKT}}
For $e\in \overline{E}_n$ the matrix $\widetilde{\Omega}_n(e)$ defines a point of $LG_{\geq 0}(n-1,V)$. In other words, the following identity holds:
$$\widetilde{\Omega}_n(e)\Lambda_{2n-2}\widetilde{\Omega}_n(e)^T=0.$$
\end{theorem}

We now describe the embedding of $\overline{E}_n$ into the totally nonnegative Isotropic Grassmannian $IG^{\widetilde{\Lambda}_{2n}}(n+1,2n)_{\geq 0}$ following Chepuri-George-Speyer \cite{CGS}, where the skew-symmetric form $\widetilde{\Lambda}_{2n}$ is given by
$$\widetilde{\Lambda}_{2n}((x_{\bar1},x_{\tilde{1}},x_{\bar2},x_{\tilde{2}},\ldots, x_{\bar n},x_{\tilde{n}}),(y_{\bar 1},y_{\tilde{1}},y_{\bar2},y_{\tilde{2}},\ldots,y_{\bar n},y_{\tilde{n}}))=$$
$$\sum\limits_{i=1}^n(x_{\bar i}y_{\tilde{i}}-x_{\tilde{i}}y_{\bar i})+\sum\limits_{j=1}^{n-1}(x_{\overline{j+1}}y_{\tilde{j}}-x_{\tilde{j}}y_{\overline{j+1}})+(-1)^n(x_{\bar1}y_{\tilde{n}}-x_{\tilde{n}}y_{\bar1}),$$
where we use a natural identification $[2n]=[\overline n]\sqcup[\widetilde n]$ given by $\{1,2,3,4\ldots, 2n-1,2n\}=\{\overline1,\widetilde1,\overline2,\widetilde2,\ldots,\overline n,\widetilde n\}$.

The form $\widetilde{\Lambda}_{2n}$ has a two dimensional kernel, spanned by the vectors $(0,1,0,-1,\ldots,0,(-1)^{n-1})$ and $(1,0,-1,0,\ldots,(-1)^{n-1},0)$. Note that any of these vectors is orthogonal to $V$ and since $\dim(V)=2n-2$ we get $\mathbb R^{2n}=V\bigoplus\ker\widetilde{\Lambda}_{2n}$.
\begin{theorem} \cite[Theorem 1.6 and Theorem 1.8]{CGS} \label{Chepuri-George-Speyer embedding}
Let $e(\Gamma,\omega)\in E_n$ and denote by $M_R$ its response matrix. Choose an arbitrary matrix $S$ such that:
\begin{equation}\label{eq:S_ij}
-M_R=\left(\begin{matrix}
S_{1n}-S_{11} & S_{2n}-S_{21} & \cdots & S_{nn}-S_{n1} \\
S_{11}-S_{12} & S_{21}-S_{22} & \cdots & S_{n1}-S_{n2} \\
\vdots & \vdots & \cdots & \vdots \\
S_{1(n-1)}-S_{1n} & S_{2(n-1)}-S_{2n} & \cdots & S_{n(n-1)}-S_{nn}
\end{matrix}\right).
\end{equation}
Then the matrix $MD\in IG^{\widetilde{\Lambda}_{2n}}(n+1,2n)_{\geq 0}$, where $M$ is $(n+1\times 2n)$ matrix defined by
$$M=\left(\begin{matrix}
0 & 1 & 0 & 1 & \ldots & 0 & 1 \\
1 & S_{11} & 0 & S_{12} & \cdots & 0 & S_{1n} \\
0 & S_{21} & 1 & S_{22} & \cdots & 0 & S_{2n} \\
\vdots & \vdots & \vdots & \vdots & \cdots & \vdots & \vdots \\
0 & S_{n1} & 0 & S_{n2} & \cdots & 1 & S_{nn}
\end{matrix}\right)$$
and
$D=diag(1,1,-1,-1,1,1,\ldots,(-1)^{(n\,\mathrm{mod}\,2)+1},(-1)^{(n\,\mathrm{mod}\,2)+1})$ is the $2n\times 2n$ diagonal matrix.
\end{theorem}
Similar parametrization was constructed for cactus networks, see \cite[Theorem 1.9]{CGS}.
\subsection{Electrical Lie theory and Lam group} 
Electrical Lie theory was introduced in \cite{L2} and subsequently studied in \cite{YS} and \cite{BGG}. The relationship between the Lam's embedding of the set $E_n$ into the totally nonnegative Grassmannian $Gr(n-1,2n)_{\geq 0}$ and electrical Lie theory could be summarized by Theorem \ref{Lam_group_acting} below (see \cite[Section 5.4]{L} for the details). A different approach relating networks and totally nonnegative Lagrangian Grassmannians was discovered in \cite{Kar}.

Define the cyclic operator $\chi$ acting on Grassmannian $Gr(k,m)$ as in \cite[Section 3.1]{L}.
$$\chi: (v_1,v_2,\ldots,v_m)\to (v_2,\ldots,v_m,(-1)^{k-1}v_1),$$
where $v_i$ for $ i\in1,\ldots,m$ are the columns of the matrix representing the point in the Grassmannian $Gr(k,m)$. Denote by $s$ the $2n\times 2n$-matrix of the operator $\chi$ written for $k=n-1,\ m=2n$:
\begin{equation} \label{cyclic operator s}
s=\left(\begin{matrix}
0 & 1 & 0 & 0 & \cdots & 0 \\
0 & 0 & 1 & 0 & \cdots & 0 \\
0 & 0 & 0 & 1 & \cdots & 0 \\
\vdots & \vdots & \vdots & \vdots & \ddots & \vdots \\
0 & 0 & 0 & 0 & \cdots & 1 \\
(-1)^{n} & 0 & 0 & 0 & \cdots & 0
\end{matrix}\right).
\end{equation}

\begin{definition} \label{Lam_group}
Introduce some auxiliary matrices:

$\bullet$ if $i\in 1,\dots,2n-1$ then $x_i(t)$ is the upper-triangular matrix with ones on the main diagonal and the only one non-zero entry $(x_i)_{i,i+1}=t$ above the main diagonal,

$\bullet$ if $i\in 1,\dots,2n-1$ then $y_i(t)$ is the low-triangular matrix with ones on the main diagonal and the only one non-zero entry $(y_i)_{i+1,i}=t$ below the main diagonal,

$\bullet$ if $i=2n$ then $x_{2n}(t)=sx_1(t)s^{-1},$

$\bullet$ if $i=2n$ then $y_{2n}(t)=sy_1(t)s^{-1}.$



Finally, define $u_i(t)=x_i(t)y_{i-1}(t)=y_{i-1}(t)x_i(t)$ for $i\in 1,\ldots,2n$, where the indices taken mod $2n$. The group generated by all $u_i(t)$ where $i\in1,\ldots,2n$ is called the {\itshape Lam group}. It is a representation of the electrical Lie group of type $A$, see \cite[Proposition 5.13]{L}. A {\itshape Lam algebra} is the Lie algebra of the Lam group, it is generated by the set $\mathfrak{u}_i=\log(u_i(t))$, where $i\in1,\ldots,2n$.
\end{definition}

The following theorem explains the combinatorial meaning of Lam's generators in the context of electrical networks.
\begin{theorem} \textup{\cite[Proposition 5.12]{L}} \label{Lam_group_acting}

Let $e\in E_n$ be an electrical network and $N_1(\Gamma,\omega)$ be associated with it bipartite network (see Section \ref{pbn-ap}). Likewise, for the electrical networks $u_{2k-1}(t)(e)$ and $u_{2k}(t)(e)$ let $N_2(u_{2k-1}\Gamma,u_{2k-1}\omega)$ and $N_2(u_{2k}\Gamma,u_{2k}\omega)$ be the associated bipartite networks. Then, the points $X(N_1)$ and $X(N_2)$ of Grassmannian defined by $N_1(\Gamma,\omega),\ N_2(u_{2k-1}\Gamma,u_{2k-1}\omega)$ and $N_2(u_{2k}\Gamma,u_{2k}\omega)$ are related to each other as follows:

$$X(N_2)=X(N_1)u_{2k-1}(t),\ X(N_2)=X(N_1)u_{2k}(t),$$
where $u_i(t)$ are generators of the Lam group from the Definition \ref{Lam_group}.
\end{theorem}
See \textup{\cite[Pp. 23-25]{BGKT}} for the illustrations of the statement above.

\begin{theorem} \cite[Theorem 5.4]{BGKT} \label{Theorem 5.4}
The operators $u_i(t)|_V$ preserve the symplectic form $\Lambda_{2n-2}$ defined by \eqref{8}. Moreover the restriction of the Lam group to the subspace V generated by the set of matrices $u_i(t)|_V,\ i = 1,\ldots, 2n$ is the representation of the symplectic group $Sp(2n-2)$.
\end{theorem}

\section{Comparing different compactifications}\label{sec:main1}

The main goal of this section is to compare this embeddings given in Theorem \ref{Theorem about Omega_n} and Theorem \ref{Chepuri-George-Speyer embedding}. We will do it in two ways: through symplectic linear algebra and through combinatorics.

\subsection{Through linear algebra}\label{sec:linearalgebra}
Consider an electrical network $e(\Gamma,\omega)\in E_n$ with the response matrix $M_R=(x_{ij})_{i,j\in[n]}$. Two different embeddings from Theorem \ref{Theorem about Omega_n} and from Theorem \ref{Chepuri-George-Speyer embedding} gives two different matrices $\Omega_n(e)$ and $MD$ representing the points corresponding to the electrical network $e(\Gamma,\omega)$ in Grassmannians $Gr(n+1,2n)_{\geq 0}$ and $Gr(n-1,2n)_{\geq 0}$ respectively. The natural task is to relate these approaches. Below we suggest several ways to do it. 

\begin{proposition} \label{Ortogonality of embeddings}
The subspaces of $\mathbb R^{2n}$ representing the points $MD\in Gr(n+1,2n)_{\geq 0}$ and $\Omega_n(e)\widetilde{D}\in Gr(n-1,2n)_{\geq 0}$ are orthogonal with respect to the standard bilinear form:
$$MD(\Omega_n(e)\widetilde{D})^T=0,$$
where 
$\widetilde{D}=diag(-1,1,-1,1,\ldots, -1,1)$.
\end{proposition}
\begin{proof}
The proof is a straightforward matrix multiplication. 
\end{proof}

\begin{lemma} \cite[part of Theorem 4.4]{BGKT} \label{Omega is isotropic under Lambda}
 The subspace defined by the row space of the matrix $\Omega_n(e)$ is isotropic w.r.t. the symplectic form $\overline{\Lambda}_{2n}:=\Lambda_{2n}^{-1}$ $(2n\times 2n$ matrix$)$ given by \eqref{8}. In other words, the following identity holds
 $$\Omega_n(e)\Lambda_{2n}^{-1}(\Omega_n(e))^T=0.$$
\end{lemma}
\begin{proof}
    By the straightforward matrix multiplication we get that $\overline{\Lambda}_{2n}:=\Lambda_{2n}^{-1}=T_{2n}\lambda_{2n}T_{2n}^T$ in the notations of \cite[Section 4]{BGKT}.
    Thus the statement holds by the argument from the proof of \cite[Theorem 4.4]{BGKT}.
\end{proof}

\begin{remark}
    The form $\overline{\Lambda}_{2n}$ is not the unique skew-symmetric form w.r.t. which the matrix $\Omega_n(e)$ is isotropic. For example, another such form is $\overline{\overline\Lambda}_{2n}$ whose inverse $\overline{\overline\Lambda}_{2n}^{-1}$ is given by $$\overline{\overline\Lambda}_{2n}^{-1}((x_{\bar1},x_{\tilde{1}},x_{\bar2},x_{\tilde{2}},\ldots, x_{\bar n},x_{\tilde{n}}),(y_{\bar 1},y_{\tilde{1}},y_{\bar2},y_{\tilde{2}},\ldots,y_{\bar n},y_{\tilde{n}}))=$$
$$\sum\limits_{i=1}^n(x_{\bar i}y_{\tilde{i}}-x_{\tilde{i}}y_{\bar i})+\sum\limits_{j=1}^{n-1}(x_{\overline{j+1}}y_{\tilde{j}}-x_{\tilde{j}}y_{\overline{j+1}})-(-1)^n(x_{\bar1}y_{\tilde{n}}-x_{\tilde{n}}y_{\bar1}),$$
where we use a natural identification $[2n]=[\overline n]\sqcup[\widetilde n]$ given by $\{1,2,3,4\ldots, 2n-1,2n\}=\{\overline1,\widetilde1,\overline2,\widetilde2,\ldots,\overline n,\widetilde n\}$.
\end{remark}

\begin{corollary} \label{cor: inclusion} 
  The subspace defined by the column space of the matrix $\overline\Lambda_{2n}^{-1}(MD)^T$ is the $\overline\Lambda_{2n}$-symplectic complement to the subspace defined by row space of the matrix $\Omega_n(e)\widetilde{D}$ and therefore we have the inclusion of the correspondent subspaces.  
\end{corollary}

\begin{proof}
    By Proposition \ref{Ortogonality of embeddings} we have that 
    $$MD(\Omega_n(e)\widetilde{D})^T=0,\text{ i.e. }\langle MD,\Omega_n(e)\widetilde{D}\rangle_{Euc}=0,$$ 
    where by $\langle\cdot,\cdot\rangle_{Euc}$ we mean the standard Euclidean bilinear form on $\mathbb R^{2n}$. 
    
    Or, equivalently,
    $$\Omega_n(e)\widetilde{D}(MD)^T=0,\text{ i.e. }\langle \Omega_n(e)\widetilde{D},MD\rangle_{Euc}=0.$$ 
    Multiplying by $\overline\Lambda_{2n}\overline\Lambda_{2n}^{-1}$ we obtain
    $$\Omega_n(e)\widetilde{D}\overline\Lambda_{2n}\overline\Lambda_{2n}^{-1}(MD)^T=0,\text{ i.e. }\langle\Omega_n(e)\widetilde{D},MD(\overline\Lambda_{2n}\overline\Lambda_{2n}^{-1})^T\rangle_{Euc}=0.$$
    Or, equivalently,
    $$\langle\Omega_n(e)\widetilde{D},(MD(\Lambda_{2n}^{-1})^T)\Lambda_{2n}^T\rangle_{Euc}=0.$$
    In other words, $(\Omega_n(e)\widetilde{D})^{\perp_{\overline\Lambda_{2n}}}=\Lambda_{2n}^{-1}(MD)^T$.
    By Proposition \ref{Omega is isotropic under Lambda} $\Omega_n(e)$ is isotropic for the form $\Lambda_{2n}$.  Because of a specific form of the diagonal matrix $\widetilde{D}$ we have that $\Omega_n(e)\widetilde{D}$ is also isotropic for the form $\Lambda_{2n}$, therefore $\Omega_n(e)\widetilde{D}\subset (\Omega_n(e)\widetilde{D})^{\perp_{\overline\Lambda_{2n}}}$. Since we proved already that $(\Omega_n(e)\widetilde{D})^{\perp_{\overline\Lambda_{2n}}}=\overline\Lambda_{2n}^{-1}(MD)^T$, we conclude that the subspace defined by the row space of the matrix $\Omega_n(e)\widetilde{D}$ lies inside the subspace defined by the column space of the matrix $\overline\Lambda_{2n}^{-1}(MD)^T$. 
\end{proof}

The next example shows this embedding turns out to be very explicit. We expect that the same is true for any $n$.

\begin{example} \label{rem: inclusion}
For $n=3$ the matrices $\Omega_3(e)\widetilde{D}$ and $\overline\Lambda_{2\cdot3}^{-1}$ are given by

$$\Omega_3(e)\widetilde{D}=\left(\begin{array}{rcrcrc}
-S_{11}+S_{13} & 1 & S_{21}-S_{23} & 0 & -S_{31}+S_{33} & -1 \\
-S_{11}+S_{12} & 1 & S_{21}-S_{22} & 1 & -S_{31}+S_{32} & 0 \\
S_{12}-S_{13} & 0 & -S_{22}+S_{23} & 1 & S_{32}-S_{33} & 1 
\end{array}\right),$$

$$\overline\Lambda_{2\cdot 3}^{-1}(MD)^T=\left(\begin{array}{cccc}
1 & S_{11} & S_{21} & S_{31} \\
0 & -1 & 1 & 0 \\
0 & S_{11}-S_{12} & S_{21}-S_{22} & S_{31}-S_{32} \\
0 & 0 & 1 & -1 \\
0 & -S_{12}+S_{13} & -S_{22}+S_{23} & -S_{32}+S_{33} \\
0 & 0 & 0 & -1
\end{array}\right),$$

$$(\overline\Lambda_{2\cdot 3}^{-1}(MD)^T)^T=\left(\begin{array}{crcrcr}
1 & 0 & 0 & 0 & 0 & 0 \\
S_{11} & -1 & S_{11}-S_{12} & 0 & -S_{12}+S_{13} & 0 \\
S_{21} & 1 & S_{21}-S_{22} & 1 & -S_{22}+S_{23} & 0 \\
S_{31} & 0 & S_{31}-S_{32} & -1 & -S_{32}+S_{33} & -1
\end{array}\right).$$
Since the response matrix is symmetric \cite[Theorem 2 and Theorem 3]{CIM} and by Theorem \ref{Chepuri-George-Speyer embedding} is equal to the matrix

$$-\left(\begin{matrix}
S_{13}-S_{11} & S_{23}-S_{21} & S_{33}-S_{31}  \\
S_{11}-S_{12} & S_{21}-S_{22} & S_{31}-S_{32}  \\
S_{12}-S_{13} & S_{22}-S_{23} & S_{32}-S_{33}  
\end{matrix}\right),$$
using this symmetry and adding the row $i,\ i\in\{3,4\}$ to the row $1$, multiplied by the relevant factor, of the matrix $(\overline\Lambda_{2\cdot3}^{-1}(MD)^T)^T$, we obtain the rows $2$ and $3$ of the matrix $\Omega_3(e)\widetilde{D}$. Note also that by Theorem \ref{Theorem about Omega_n} $\text{rank}(\Omega_3(e))=3-1$ and the alternating sum of the rows of $\Omega_3(e)$ is equal to zero. Thus the row space of the matrix $\Omega_3(e)\widetilde{D}$ lies in the row space of the matrix $(\overline\Lambda_{2\cdot3}^{-1}(MD)^T)^T$. In other words, the row space of the matrix $\Omega_3(e)\widetilde{D}$ lies in the column space of the matrix $\overline\Lambda_{2\cdot3}^{-1}(MD)^T$.
\end{example}

\begin{corollary}
Consider a cactus network $e' \in \overline{E}_n$ and associated with it matrices $M_1$ and $M_2$ which define points of $Gr_{\geq 0}(n-1, 2n)$ and $Gr_{\geq 0}(n+1, 2n)$ correspondingly. Then Proposition \ref{Ortogonality of embeddings}, Lemma \ref{Omega is isotropic under Lambda}, Corollary \ref{cor: inclusion} are  hold if we replace $\Omega_n(e)$ and $MD$ with $M_1$ and $M_2$ correspondingly.
        \begin{proof}
        Let us prove Proposition \ref{Ortogonality of embeddings} for a connected cactus network $e\in\overline{E}_n$. $e$ can be obtained from some $e'\in E_n$ tending some of its conductances to the infinity. By the continuity argument from \cite[Subsection 3.2]{BGKT} and the proof of \cite[Theorem 1.9]{CGS} it is possible to replace the matrices $\Omega_n(e)$ and $MD$ by the sufficient matrices $\Omega_{n,R}(e')$ and $M_TD$ (see Theorem \ref{Theorem about Omegar_n} and Remark \ref{thm: Speyer for cactus} for their precise form). Then it is possible to take a limit $R_{ij}\to\infty$. Orthogonality of the subspaces does not depend on the choice of the matrices, which define them, thus we have $M_TD(\Omega_{n,R}(e')\widetilde{D})^T=0$. By tending some $R_{ij}$ to zero we prove Proposition \ref{Ortogonality of embeddings} for a connected cactus network $e$ . For a disconnected cactus network we have to deal with each of the block of the corresponding matrices separately in the same way. The proofs of Lemma \ref{Omega is isotropic under Lambda}, Corollary \ref{cor: inclusion} are absolutely analogous. 
    \end{proof}
\end{corollary}

\subsection{Through combinatorics} \label{th-comb}

A famous Theorem \ref{m-pos} due to Postnikov provides an embedding of Postnikov's networks into the totally nonnegative Grassmannian. On the other hand, given an electrical network $e(\Gamma,\omega)$ one can associate to it a Postnikov's network as follows:
$$e(\Gamma,\omega)\stackrel{\text{Temperley's trick}}{\longmapsto}N(\Gamma)\stackrel{\text{Definition \ref{postn-net}}}{\longmapsto} PN(\Gamma,\omega',O)$$
(see Section \ref{pbn-ap} for precise definitions). Combining these constructions together we get an embedding $E_n\hookrightarrow Gr_{\ge 0}(n-1,2n)$. It was shown in \cite[Theorem 3.3 and Theorem 3.6]{BGKT} that this embedding coincides with the one from Theorem \ref{Theorem about Omega_n} given by $e(\Gamma,\omega)\mapsto \Omega_n(e)$. However, the construction above is not canonical since it depends on the choice of colors for the vertices in the definition of generalized Temperley's trick. The goal of this subsection is to show that turning black vertices into white and vice versa produces the embedding from Theorem \ref{Chepuri-George-Speyer embedding}. The operation of obtaining a bipartite graph by reversing the color of the vertices as above we call a dual generalized Temperley's trick (see Definition \ref{temp_gen}).

 \begin{definition} \cite[Section 1]{CGS} \label{co-con}
     We call an $(n+1)$-element subset $I \subset [1, \dots, 2n]$ {\em co-concordant} with a non-crossing partition $\sigma$ if $I$ contains exactly one element in each part of $\sigma$ and each part of the dual partition $\tilde{\sigma}.$
 \end{definition}
Note that in \cite[Section 1]{CGS} Definition \ref{co-con} was formulated in a slightly different way (see also \cite[Section 2.4]{CGS}).

In terms of the co-concordant spaces 
the analogue of Theorem \ref{Emb} has the form:
 \begin{theorem} \cite[Theorem 2.5]{CGS} \label{thm:co-Emb}
 Let $e \in \overline{E}_n.$ Then the collection of boundary measurements
$$\Delta^{\circ}_{I}:= \sum\limits_{(\sigma, I)} L_{\sigma},$$
 defines a point in  $Gr_{\geq 0}(n + 1, 2n)$, where the summation is over all such $\sigma,$ which are co-concordant with $I$. 
 \end{theorem}
 \begin{remark}
     Similarly to Remark \ref{Plucker coord as dimer partition function} the variable $\Delta^\circ_I$ has an independent combinatorial definition as the dimer partition function of all 
     almost perfect matchings
     on bipartite network $N^d(\Gamma, \omega)$ (see Definition \ref{temp_gen}) with the boundary $I$.
 \end{remark}

Taking into account  Definition \ref{co-con} and repeating the proof of \cite[Lemma 2.27]{BGKT}, we obtain:

\begin{lemma} \label{lemmal}

 For  $e(\Gamma,\omega)\in E_n$ and  $N(\Gamma, \omega)$ as above the following holds:
 $$L_{unc}=\Delta_{[2n] \setminus R}^{\circ},$$ where $R$ is any set of even indices
 such that  $|R|=n-1$.

 $$L_{ij}:=L_{ij|*|*\dots}=\Delta_{[2n]\setminus (\{2i-1\}\cup N)}^{\circ},$$ where $N$ is the set of all even indices except of the two closest to  
 $2j-1$ and $|N|=n-2.$ 
 
 $$L_{kk}:=\sum_{i\neq k}L_{ik}=\Delta_{[2n]\setminus (\{2k-1\}\cup T)}^{\circ},$$ where $T$ is the set of all even indices except of the two closest to $2k-1$ and $|T|=n-2$. 
 \end{lemma}

For an electrical network $e(\Gamma,\omega) \in E_n$ let us consider associated with it the Postnikov network $PN(\Gamma,\omega',O_2)$ (see Definition \ref{postn-net}), where $O_2$ is a perfect orientation with the set of sources  $O_2=\{1, 2, 3, 5, \dots, 2n-1 \}$. Note that a perfect orientation $O_2$ exists due to the fact that $\Delta_{[2n]\setminus R}^{\circ}=L_{unc}\neq 0$ (see \cite[Remark 2.7]{BGKT} and the proof of \cite[Theorem 3.3]{BGKT}). Then Postnikov's network $PN(\Gamma,\omega',O_2)$ defines the same point of $Gr_{\geq 0}(n+1, 2n)$ as the collection of boundary measurements $\Delta^{\circ}_{I}$, more explicitly:
\begin{theorem} \label{pl-cor}
   For any $e \in E_n$ the following holds:
    $$\frac{\Delta^{\circ}_I}{L_{unc}}=\Delta_I(M_B),$$
    where $M_B$ is the extend boundary measurement matrix (see Definition \ref{extend}) of $PN(\Gamma,\omega',O_2).$
\end{theorem}
\begin{proof}
    The proof is analogous to the proof of \cite[Corollary 3.5]{BGKT}.
\end{proof}
Now, we are ready to obtain Chepuri--George--Speyer embedding. 
\begin{lemma} \label{lem:second row}
    Consider  $PN(\Gamma,\omega',O_{2})$ then the row of $M_B$ corresponding to the source labeled by $2$ is:
    $$(0, 1, 0, -1, 0, 1, 0, -1, \dots).$$
\end{lemma}
\begin{proof}
  Directly from Definition \ref{extend} 
  we obtain:
    $$a_{2 \to 2}=1,\; a_{2 \to 2k-1}=0.$$
On the other hand, for any $k \neq 1$ we have:
$$
\Delta_{\{2k\}\cup O_2 \setminus \{2\} }(M_B) =\frac{\Delta_{\{2k\}\cup O_2 \setminus\{2\} }^{\circ}}{L_{unc}}=1,
$$
here the first equality follows from Theorem \ref{pl-cor} and the second equality follows from Lemma \ref{lemmal}.
Moreover,  from Definition \ref{extend} we obtain:
$$\Delta_{\{2k\}\cup O_2 \setminus \{2\} }(M_B)=(-1)^sa_{2 \to 2k},$$
where $s$ is the number of the elements in $O_2$  between the source labeled by $2$ and the sink labeled by $2k.$ By the direct computations we conclude that $a_{2 \to 2k}=(-1)^{k-1}.$ 
\end{proof}

The following theorem is a main result of this subsection. It is an analog of \cite[Theorem 3.3 and Theorem 3.6]{BGKT} and provide an embedding from Theorem \ref{Chepuri-George-Speyer embedding} in terms of the dual generalized Temperley's trick and the Postnikov's embedding.

\begin{theorem} \label{thm:explicitcoord}
    Posntikov network $PN(\Gamma,\omega',O_{2})$ defines the point in $Gr_{\geq 0}(n+1, 2n)$ spanned by vector-rows of the following matrix:
    \begin{equation} \label{ch-sp}
    X=\left(\begin{matrix}
0 & 1 & 0 & 1 & \ldots & 0 & 1 \\
1 & S_{11} & 0 & S_{12} & \cdots & 0 & S_{1n} \\
0 & S_{21} & 1 & S_{22} & \cdots & 0 & S_{2n} \\
\vdots & \vdots & \vdots & \vdots & \cdots & \vdots & \vdots \\
0 & S_{n1} & 0 & S_{n2} & \cdots & 1 & S_{nn}
\end{matrix}\right)D,    
    \end{equation}
    
where $D=diag(1, 1, -1, -1, \dots)$ and $S_{ij}$ are defined by Equation \eqref{eq:S_ij}.
\end{theorem}

\begin{proof}
According to Definition \ref{extend} and Lemma \ref{lem:second row} the extended boundary measurements matrix of $PN(\Gamma,\omega',O_{2})$ has the form:   
  
     $$M_B=\left(\begin{matrix}
1 & F_{11} & 0 & F_{12} & \cdots & 0 & F_{1n} \\
0 & 1 & 0 & -1 &  \cdots & 0 & (-1)^{n-1} \\
0 & F_{21} & 1 & F_{22} & \cdots & 0 & F_{2n} \\
\vdots & \vdots & \vdots & \vdots & \cdots & \vdots & \vdots \\
0 & F_{n1} & 0 & F_{n2} & \cdots & 1 & F_{nn}
\end{matrix}\right),$$
where $F_{ij}=a_{2i-1 \to 2j}.$

 Let us swap the first and second rows of $M_B$ to obtain the matrix $M_{B}''.$ Then it is easy to verify that $M_{B}''$ can be rewritten as follows: 
$$M_{B}''=D'\left(\begin{matrix}
0 & 1 & 0 & 1 &  0 & 1 & \ldots \\
1 & F_{11} & 0 & -F_{12} & 0 & F_{13} & \ldots \\
0 & -F_{21} & 1 & F_{22} & 0 & -F_{23}& \ldots \\
\vdots & \vdots & \vdots & \vdots & \cdots & \vdots & \vdots \\
\end{matrix}\right)D,$$
where $D'=diag(1, 1, -1, 1, -1, \dots)$ of size $n+1$.

Since $PN(\Gamma,\omega',O_{2})$ defines the point in $Gr_{\geq 0}(n+1, 2n)$ spanned by the rows of the  matrix $M_B$ (see Theorem \ref{m-pos}), the same point is defined by the matrix $M_{B}' = D'M_B''.$

Note that for any $I\subset[2n]$
\begin{equation} \label{eq:right sign}
\Delta_I(M_B')=\det(D')\Delta_I(M_B'')=-\det(D')\Delta_I(M_B)=(-1)^{\binom{n}{2}+1}\Delta_I(M_B),
\end{equation}
where the first equality holds by the multiplicativity of determinant, the second equality holds because a swap of two rows changes the sign of determinant, and the third equality holds because $\det(D')=(-1)^{\binom{n}{2}}$.

Let us proof that $S'_{ij}:=(-1)^{i+j}F_{ij}$ satisfies the conditions for $S_{ij}$. Indeed:
\begin{multline} \label{xij_1}
S'_{kk}-S'_{kk-1}=(-1)^{\binom{n}{2}+1}\Delta_{[2n]\setminus (\{2k-1\}\cup T)}(M_{B}')=\Delta_{[2n]\setminus (\{2k-1\}\cup T)}(M_{B})=\\
=\frac{\Delta_{[2n]\setminus (\{2k-1\}\cup T)}^{\circ}}{L_{unc}}=\frac{L_{kk}}{L_{unc}}=x_{kk},   
\end{multline}
 \begin{multline} \label{xij_2}
 S'_{ij-1}-S'_{ij}=(-1)^{\binom{n}{2}+1}\Delta_{[2n]\setminus (\{2i-1\}\cup N)}(M_{B}')=\Delta_{[2n]\setminus (\{2i-1\}\cup N)}(M_{B})=\\
 =\frac{\Delta_{[2n]\setminus (\{2i-1\}\cup N)}^{\circ}}{L_{unc}}=\frac{L_{ij}}{L_{unc}}=-x_{ij},    
 \end{multline}
here $N$ and $T$ are the same as in Lemma \ref{lemmal} and all indices in $S'_{ij}$ are from the set $\{1,\ldots n\}\,(\mathrm{mod}\, n)$, where $0 (\text{mod }n)=n$.
The first equalities in \eqref{xij_1} and \eqref{xij_2} hold besause of the exact form of $M_B'$ (see the proof of Theorem \cite[Theorem 1.8]{CGS}), the second equalities hold due to \eqref{eq:right sign}, the third equalities hold due to Theorem \ref{pl-cor}, the fourth equalities hold due to Lemma \ref{lemmal} and the last equalities hold due to \cite[Proposition 2.8]{KW 2011}.
\end{proof}
Combining Theorem \ref{pl-cor} and Theorem \ref{thm:explicitcoord} we obtain the following:
\begin{corollary}\label{cor:maincor}
    For every $e \in E_n$ the dual generalized Temperley's trick defines a point in $Gr_{\geq 0}(n+1, 2n)$ which might be parameterized by a matrix defined by \eqref{ch-sp}. In other words, the composition 
    $$e(\Gamma,\omega)\stackrel{\text{dual Temperley's trick}}{\longmapsto}N^d(\Gamma)\stackrel{\text{\text{Definition \ref{postn-net}}}}{\longmapsto} PN(\Gamma,\omega',O_2)\stackrel{\text{Postnikov's embedding}}{\longmapsto}M_B$$
    (see Section \ref{pbn-ap} for precise definitions) coincides with the embedding from Theorem \ref{Chepuri-George-Speyer embedding}, i.e. it maps $e(\Gamma,\omega)\mapsto MD$.
\end{corollary}

\section{Symplectic Lie algebra and electrical networks}\label{sec:main2}

\subsection{Concordance and symplectic forms} \label{sec: concordance property and symplectic forms}

The symplectic form \eqref{8} was discovered in the context of electrical algebras in \cite{BGG} where the generalized electrical algebras were introduced. It was used subsequently in the context of electrical networks in 
\cite{BGKT}. 
The connection of this form to the combinatorics developed in \cite{L} has been
 mysterious until now. The goal of this section is to explain its appearance in terms of combinatorics of concordance vectors.










Given a vector space $W$ and a skew symmetric form $\omega:W\times W\to \mathbb R$ one can define the following convolution operator:
\begin{equation} \label{convolution}
Q:\bigwedge^k W\to\bigwedge^{k-2}W
\end{equation}

$$Q(x_1\wedge x_2\wedge \ldots \wedge x_k)=\sum\limits_{i<j}\omega(x_i,x_j)(-1)^{i+j-1}x_1\wedge\ldots\wedge \widehat{x}_i\wedge\ldots\wedge\widehat{x}_j\wedge\ldots\wedge x_k.$$

\begin{theorem} \textup{\cite[Theorem 5.9]{BGKT}}\label{Fund_repr}
The following holds

$\bullet$ $H$ is the subspace of $\bigwedge^{n-1}V$.

$\bullet$ For $W=V,\ \omega=\Lambda_{2n-2}$ and $k=n-1,\ H$ coincides with the kernel of the convolution operator $Q$.

$\bullet$ The subspace $H$ is invariant under the action of the Lam group and hence the
action of $Sp(2n-2)$. Moreover, $H$ is the space of the fundamental representation of the group $Sp(2n-2)$ which corresponds to the last vertex of the Dynkin diagram.
\end{theorem}

One could ask whether the form for which $H\subset \ker Q$ is unique. The following remark states that in general the uniqueness is not the case.
\begin{remark} Consider following skew symmetric forms on space $\mathbb R^6$:
$$\left(\begin{array}{rrrrrr}
0 & 1 & 1 & -1 & -1 & 1 \\
-1 & 0 & 1 & 0 & -1 & -1 \\
-1 & -1 & 0 & 1 & 0 & -1 \\
1 & 0 & -1 & 0 & 1 & 1 \\
1 & 1 & 0 & -1 & 0 & 1 \\
-1 & 1 & 1 & -1 & -1 & 0 \\
\end{array}\right)\phantom{aaaaaaaaaa}\left(\begin{array}{rrrrrr}
0 & -1 & 0 & 0 & 0 & 1 \\
1 & 0 & 1 & 0 & 0 & 0 \\
0 & -1 & 0 & -1 & 0 & 0 \\
0 & 0 & 1 & 0 & 1 & 0 \\
0 & 0 & 0 & -1 & 0 & -1 \\
-1 & 0 & 0 & 0 & 1 & 0 \\
\end{array}\right).$$
One could check that the concordance vectors $w_{1|23},\ w_{13|2},\ w_{12|3},\ w_{1|2|3},\ w_{123}$ belong to $\ker Q$.
\end{remark}

However, the theorem below states that after the restriction to $V$ such a form is unique up to a scalar. This fact and the combinatorial proof of this theorem solve the mystery of the appearance of the form $\Lambda_{2n-2}$.
\begin{theorem} \label{Classification}
Let $\omega$ be a skew symmetric form on $V$ such that the kernel of the convolution $Q$ with $\omega$ on the space $\bigwedge^{n-1} V$ is equal to $H$. Then $\omega=a\cdot\Lambda_{2n-2}$ for some $a\in \mathbb R\setminus\{0\}$. 
\end{theorem}

Theorem \ref{Fund_repr} states that $H\subset\bigwedge^{n-1} \mathbb R^{2n}$ actually lies within $\bigwedge^{n-1}V$. 
 Algorithm \ref{Algoritm_embedding} below gives a precise expression of decomposing $w_{\sigma}\in H$ in terms of the basis vectors \eqref{Basis of V} of $V$. Among of an independent interest this algorithm is a key step in the proof of Theorem \ref{Classification}. 
 
\begin{algorithm} \label{Algoritm_embedding}
The basis vector $w_{\sigma}$ corresponding to the partition $\sigma$ written in the basis \eqref{Basis of V} is obtained by the following steps.

\begin{enumerate}

\item[\bf{1.}] Merge all the vertices of the initial non-crossing partition and its dual in the one set $[2n]$:
$$(\sigma|\widetilde\sigma) = (i_1 i_2 i_3 \ldots i_{k-2} i_{k-1} i_k|i_{k+1} i_{k+2} i_{k+3}\ldots i_{l-2} i_{l-1} i_l|\ldots |i_{p+1} i_{p+2} i_{p+3}\ldots i_{2n-2} i_{2n-1} i_{2n}),$$
where $i_1<i_{k+1}<i_{l+1}<\ldots<i_p$.

\item[\bf{2.}] Consider the following expression:
\begin{equation} \label{br2} \begin{array}{l}
    (i_1\ i_2)(i_2\ i_3)\ldots (i_{k-2}\ i_{k-1})(i_{k-1}\ i_k)(i_{k+1}\ i_{k+2})(i_{k+2}\ i_{k+3})\ldots \\ \ldots (i_{l-2}\ i_{l-1})(i_{l-1}\ i_l)\ldots (i_{p+1}\ i_{p+2})(i_{p+2}\ i_{p+3}) \ldots
(i_{2n-2}\ i_{2n-1})(i_{2n-1}\ i_{2n}) \end{array}
\end{equation}
Here we use the following notation: instead of $e_{i_s}$ we write $i_s$ and skip "$\wedge$" between brackets.

\item[\bf{3.}] Reorder brackets in the way of increasing first indices in each bracket, and then the signs of $e_{i_s}$ should be placed as follows.
Consider an arbitrary bracket $(i_x\ i_{x+1})$ in the expression (\ref{br2}). Before the first term $i_x$ we should put the sign "+". Before the second term $i_{x+1}$ we should put the sign "+" if the number of integers that lie in the interval $(i_x;i_{x+1})$ and have the same parity as $i_x$ and $i_{x+1}$ is even, and the sign "-" otherwise. Note that such a definition is correct since in the merged numbering even numbers correspond to the initial non-crossing partition, while odd --- to it's dual, and thus $i_x$ and $i_{x+1}$ always have the same parity.

\item[\bf{4.}] Finally, we rewrite the expression from {\bf 3.} in terms of basis \eqref{Basis of V} as follows:
$$(e_{i_x}\pm e_{i_{x+1}})=(v_{i_x}-v_{i_{x}+2}+v_{i_{x}+4}-\ldots (-1)^{\frac{i_{x+1}-2-i_x}{2}}v_{i_{y}-2}).$$
\end{enumerate}

\end{algorithm}

\begin{example}\label{ex:algorithm}
Let us illustrate Algorithm \ref{Algoritm_embedding} on our running example $\sigma=(\bar{1},\bar{4},\bar{6}|\bar{2},\bar{3}|\bar{5})$.

The united partition $(\sigma|\widetilde\sigma)$ of the set $[12]$ corresponding to $\sigma$ has the form:
$$(1,7,11|2,6|3,5|4|8,10|9|12).$$
The expression for the vector $\omega_{\sigma}$ as in (\ref{br2}) is:
$$(1\ 7)(7\ 11)(2\ 6)(3\ 5)(8\ 10).$$
Following the Algorithm we reorder the brackets in the way of increasing first indices in each bracket:
$$(1\ 7)(2\ 6)(3\ 5)(7\ 11)(8\ 10).$$
The signs before the first term in each bracket is "$+$", so it remains to put the signs before the second terms in each bracket, which should be done as follows:

The set of odd numbers which lie in the interval $(1,7)$ is $\{3,5\}$. The cardinality of this set is an even number, so we write $(1+7)$.

The set of odd numbers which lie in the interval $(2,6)$ is $\{4\}$. The cardinality of this set is an odd number, so we write $(2-6)$.

The set of odd numbers which lie in the interval $(3,5)$ is empty. The cardinality of this set is an even number, so we write $(3+5)$.

The set of odd numbers which lie in the interval $(7,11)$ is $\{9\}$. The cardinality of this set is an odd number, so we write $(7-11)$.

The set of odd numbers which lie in the interval $(8,10)$ is empty. The cardinality of this set is an even number, so we write $(8+10)$.

Taking all above together we obtain the following expression for the vector $w_{\sigma}$:
$$w_{\sigma}=(e_1+e_7)\wedge(e_2-e_6)\wedge(e_3+e_5)\wedge(e_7-e_{11})\wedge(e_8+e_{10})=$$
$$(v_1-v_3+v_5)\wedge(v_2-v_4)\wedge v_3\wedge(v_7-v_9)\wedge v_8.$$

\end{example}

\begin{remark}
A different algorithm to decompose a concordance vector as a pure wedge product is established in \cite[Lemma 3.4]{CGS}. However, a different definition of the concordance is used in that paper. Thus it is not immediately clear how to rewrite that algorithm in Lam's language directly. Moreover, the exact form of the restriction from the Isotropic Grassmannian $IG(n+1,2n)$ appeared in that paper to the Lagrangian Grassmannian was not studied yet. For our purposes the possibility to write a concordance vector as an element of $\bigwedge^{n-1}V$ is crucial.
\end{remark}
 It is known from Theorem \ref{Fund_repr} that the concordance space $H$ lies inside $\bigwedge^{n-1}V$, however Algorithm \ref{Algoritm_embedding} gives explicit expressions of the basis vectors of the space $H$ in terms of basis vectors of the space $V$. Thus Algorithm \ref{Algoritm_embedding} gives another proof of the fact that $H\subset \bigwedge^{n-1}V$, that is an additional nice application of our construction. 
\begin{lemma} \label{diagonal}
Let $\omega$ be a skew-symmetric form on $V$ such that the kernel of the convolution with $Q$ on the space $\bigwedge^{n-1} V$ contains $H$. Then 
$$\omega(i,j)=
\begin{cases}
(-1)^i a,& \text{\ if\ }  j=i+1, \\
0,& \text{\ if\ } j>i+1,
\end{cases}$$
for some $a\in \mathbb R$.
\end{lemma}

\begin{proof}[Proof of Theorem \ref{Classification}]
Lemma \ref{diagonal} implies that if $\omega$ is a skew symmetric form on $V$ such that the kernel of the convolution with $Q$ on the space $\bigwedge^{2n-2} V$ contains $H$, then the form $\omega$ has the type $\omega=a\cdot\Lambda_{2n-2}$ for some $a\in \mathbb R$. The choice of the parameter $\omega(1,1)=a$ uniquely defines the form $\omega$. The form is non-degenerate if $a\ne 0$, so by Theorem \ref{Fund_repr} the concordance space $H$  is contained in the kernel of the convolution $ker Q$ with such form $\omega$. Moreover, the equality $\ker Q=H$ is satisfied.  If $a=0$ we get a zero form, and thus the kernel of the convolution with it, is the zero space. The zero space can not contain the concordance space $H$ whose dimension equals to the Catalan number $C_{n+1}$. This provides the sufficiency condition.
\end{proof}


\begin{definition} \label{def: hollow network}
Let $p_{\sigma}$ denote the point in $Gr_{\ge 0}(n-1,2n)\cap\mathbb PH$ with
$$L_{\sigma'}(p_{\sigma})=\begin{cases}
1, \text{\ if\ }  \sigma'=\sigma \\
0, \text{\ otherwise.}
\end{cases}$$
For a non-crossing partition $\sigma$ of $n$ define a {\itshape hollow} cactus network $P_\sigma$ as a cactus network obtained from an empty network on $n$ boundary vertices by the identification of boundary vertices according to $\sigma$.
\end{definition}
\begin{proposition} \cite[Proposition 4.11]{L} \label{concordance vector is a cactus network}
The point $P_\sigma$ maps to $p_\sigma$ under the Lam's embedding from Theorem \ref{Emb}.
\end{proposition}

\begin{proposition} \cite[Proposition 4.11]{L} \label{empty cactus maps to p} 
 Under the Plucker embedding $p_\sigma\in Gr_{\ge 0}(n-1,2n)\cap\mathbb PH$ maps to $w_\sigma \in \bigwedge^{n-1}\mathbb{R}^{2n}.$
\end{proposition}
\begin{proof}
Denote by $P_{\sigma}$ the hollow cactus network from Proposition \ref{concordance vector is a cactus network} which maps to the vector $p_{\sigma}$ under the Lam's embedding. 

Thus the point $p_{\sigma}$ has the following Plucker coordinates:
$$\Delta_I^\bullet=\sum\limits_{(\sigma',I)}L_{\sigma'}=a_{I\sigma}L_{\sigma}=a_{I\sigma}.$$
Therefore the Plucker embedding maps $p_{\sigma}$ to $\sum\limits_{I\in(^{\ 2n}_{[n-1]})}\Delta_I^\bullet e_I = \sum\limits_{I\in(^{\ 2n}_{[n-1]})}a_{I\sigma}e_I$. This expression is indeed the definition of the concordance vector $w_{\sigma}$. 
\end{proof}

The following statement refines the result of Algorithm \ref{Algoritm_embedding}. We will prove a stronger statement in Theorem \ref{theorem:cactus-plu}.
\begin{corollary} \label{cor: coeff are of the same sign}
Plucker coordinates of a concordance vector $w_{\sigma}|_V$ thought as a point in $LG(n-1,V)$ are positive.
\end{corollary}
\begin{proof}
By Theorem \ref{Embedding into LG} $p_{\sigma}|_V\in LG_{\ge 0}(n-1,V)$. This means that all Plucker coordinates of the point $p_\sigma|_V$ in the $\bigwedge^{n-1} V$ corresponding to the empty cactus network $P_{\sigma}$ of the same sing, but as proved in Proposition \ref{empty cactus maps to p} this point is exactly $w_{\sigma}$.

Consider the collection $I=\{l_1,\ldots,l_{n-1}\}$ of the first terms from each bracket in expression \eqref{br2}. Note that $l_1<\ldots<l_{n-1}$. Thus by the construction the output of Algorithm \ref{Algoritm_embedding} will have sign $"+"$ in front of $v_I$. Since all signs have the same sign and one of them is positive, we conclude that all them are positive.
\end{proof}

\begin{remark} \label{rem: algorithm vs Plucker coordinates}
    Taking into account Proposition \ref{empty cactus maps to p} we obtain an additional meaning of Algorithm \ref{Algoritm_embedding} since it follows that Algorithm \ref{Algoritm_embedding} provides a representative to a point $p(\sigma)|_V$ under the Plucker embedding. Thus Corollary \ref{cor: coeff are of the same sign} and Theorem \ref{theorem:cactus-plu} might be thought as a description of Plucker coordinates of the points $p(\sigma)|_V$. We will develop this approach in Section \ref{subsec: lagrangian concordance}.
\end{remark}


\subsection{Proofs of Algorithm \ref{Algoritm_embedding} and Lemma \ref{diagonal}} \label{main proofs}

To prove Algorithm \ref{Algoritm_embedding} we need the following lemma due to Lam.
\begin{lemma} \cite[Lemma 2.2]{L} \label{number of parts}
If $\sigma$ and $\widetilde\sigma$ are dual non-crossing partitions, then $|\sigma|+|\widetilde\sigma|=n+1$.
\end{lemma}

\begin{proof}[Proof of Algorithm \ref{Algoritm_embedding}]

\ 

\textbf{Brackets decomposition}. We regroup the terms in the expression (\ref{br2}) according their belonging to connected components of $(\sigma|\widetilde\sigma)$:
\begin{equation} \label{br2 brackets}
[(i_1\ i_2)(i_2\ i_3)\ldots (i_{k-2}\ i_{k-1})(i_{k-1}\ i_k)]\ [(i_{k+1}\ i_{k+2})(i_{k+2}\ i_{k+3})\ldots (i_{l-2}\ i_{l-1})(i_{l-1}\ i_l)]\ \ldots\ 
\end{equation}
$$[(i_{p+1}\ i_{p+2})(i_{p+2}\ i_{p+3}) \ldots
(i_{2n-2}\ i_{2n-1})(i_{2n-1}\ i_{2n})],$$
where each squared bracket corresponds to one of the connected components.


Given a non-crossing partition $(\sigma|\widetilde\sigma)=\{B_1,\ldots,B_t\}$ where by $B_i$ we denote connected components of the merged partition as in Definition \ref{groves, ncp}, one can write the decomposition of the concordance vector $w_{\sigma}$ in terms of the standard basis $e_1,\ldots,e_{2n}$ as follows:

\begin{equation} \label{formula for a concordance vector}
w_{\sigma}=\sum\limits_{I_1,\ldots,I_t}(e_{I_1}\wedge \ldots\wedge e_{I_t})_{lex},
\end{equation}
where $\forall 1\le j \le t\ I_j\subset B_i$ such that $|I_j|=|B_j|-1$. By $()_{lex}$ we denote the lexicographical reordering of all factors in the brackets. This is a straightforward from Definition \ref{def:concordance}.

Note that up to signs the formula \eqref{br2 brackets} provides exactly the same as formula \eqref{formula for a concordance vector}. Indeed, for $B_j=\{i_1,\ldots,i_x,\ldots,i_k\}$ the subset $I_{j}:=B_j\setminus \{i_x\}$ could be obtained if and only if we firstly choose all the indices except for $i_x$ in all brackets which contain $i_x$. Then note that in order to obtain non zero wedge product, the choice of the indices in $B_i$, which do not contain $i_x$, is unique due to the rule of the construction of expression \eqref{br2 brackets}. Thus the expression (\ref{br2}) contains exactly that terms which appears in $w_{\sigma}$.

\textbf{Signs arrangement}. We now prove that the signs in the expression (\ref{br2}) should be arranged as in the algorithm's description. For that we have to compare the rule for sign's arrangement in (\ref{br2}) with signs in (\ref{formula for a concordance vector}). Starting from the expression (\ref{formula for a concordance vector}) we want to obtain (\ref{br2}) by taking common factors out of brackets. Suppose that we have two wedge products, differ only in one factor: $i_x$ in one wedge and $i_{x+1}$ in the other one. To split off the bracket $(i_x\ i_{x+1})$ we move $i_{x+1}$ in the same place in order as $i_x$, and then move them both to the first places in the corresponding wedge products. So up to a common sign in front of the bracket the sign in front of $i_{x+1}$ inside the bracket is the number of the replacements that we need to move $i_{x+1}$ to the place of $i_x$. We shall compute that sign:
\begin{multline}
\text{the sign that is obtained by replacing}\ i_{x+1}\ \text{on the place of}\  i_x= \\
\phantom{aaaaaaaaaaaaaaaaaaaaaaaaa.}=\#\text{brackets}\ (i_f\ i_{f+1})\ \text{in}\ (\ref{br2}),\ \text{where}\ i_x<i_f\ \text{or}\ i_{f+1}<i_{x+1}\\
\phantom{aaaaaaaaaaaaaaaaaaaaaaaaa.}=\#\text{brackets}\ (i_f\ i_{f+1})\ \text{in}\ (\ref{br2}),\ \text{where}\ i_x<i_f<i_{f+1}<i_{x+1}\\
=\#\text{arbitrary brackets}\ (i_f\ i_{f+1})\ \text{in the expression $\Phi$},\phantom{..}
\end{multline}
where by $\Phi$ we denote the expression \eqref{br2} written for the non-crossing partition which is induced by $\sigma$ on the shaded area in Fig. \ref{ras}.

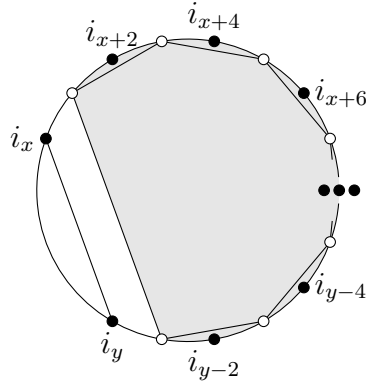
\begin{figure}[ht]
\centering
\begin{tikzpicture}
     \fill[gray!20!white] (140:2) -- node [left] {}(-100:2) arc [start angle=-100, end angle=140, radius=2];
    \draw (-5:2) arc [start angle=355, end angle=5, radius=2];

    \draw (20:2) -- node [left] {}(60:2);
    \draw (60:2) -- node [left] {}(100:2);
    \draw (100:2) -- node [left] {}(140:2);
    \draw (-20:2) -- node [left] {}(-60:2);
    \draw (-60:2) -- node [left] {}(-100:2);
    \draw (140:2) -- node [left] {}(-100:2);
    \draw (160:2) -- node [left] {}(-120:2);
    \draw (20:2) -- node [left] {}(12:1.95);
    \draw (-20:2) -- node [left] {}(-12:1.95);

    \filldraw (0:2) node [left] {} circle (2pt);
    \filldraw (0:2.2) node [above] {} circle (2pt);
    \filldraw (0:1.8) node [right] {} circle (2pt);

    \filldraw[fill=white] (20:2) node [left] {} circle (2pt);
    \filldraw (40:2) node [right] {$i_{x+6}$} circle (2pt);
    \filldraw[fill=white] (60:2) node [right] {} circle (2pt);
    \filldraw[fill=black] (80:2) node [above] {$i_{x+4}$} circle (2pt);
    \filldraw[fill=white] (100:2) node [right] {} circle (2pt);
    \filldraw[fill=black] (120:2) node [above] {$i_{x+2}$} circle (2pt);
    \filldraw[fill=white] (140:2) node [above] {} circle (2pt);
    \filldraw (160:2) node [left] {$i_x$} circle (2pt);

     \filldraw[fill=white] (-20:2) node [right] {} circle (2pt);
    \filldraw (-40:2) node [right] {$i_{y-4}$} circle (2pt);
     \filldraw[fill=white] (-60:2) node [right] {} circle (2pt);
    \filldraw (-80:2) node [below] {$i_{y-2}$} circle (2pt);
     \filldraw[fill=white] (-100:2) node [right] {} circle (2pt);
    \filldraw (-120:2) node [below] {$i_{y}$} circle (2pt);

\end{tikzpicture}
\caption{Shaded area and bijection between brackets and nodes}
    \label{ras}
\end{figure}

The first equality in above is satisfied because each wedge product between the places of $i_x$ and $i_{x+1}$ appears if and only if exist the bracket where $i_x<i_f$ or $i_{f+1}<i_{x+1}$.

The second equality in above is satisfied because in a non-crossing partition edges which connect vertices inside different parts of the  partition are not intersect, so  it could not be that only one of $i_f,i_{f+1}$ belongs to $[i_x;i_{x+1}]$.



The third equality in above is satisfied because any non-crossing partition on the shaded area in the Fig. \ref{ras} could be completed to the non-crossing partition on the initial $[2n]$ vertices.

Thus we have a reduction which allows count arbitrary brackets in the expression $\Phi$ instead of counting brackets with some restrictions in the initial expression (\ref{br2}) in order to follow the sign. In order to count the brackets in the expression $\Phi$ we count them for some special non-crossing partition on the shaded area and prove that this number is independent on the partition. 

Consider the non-crossing partition on the shaded area drawn on Fig. \ref{ras}. For this non-crossing partition brackets in the expression $\Phi$ are in a bijection with the integers which lie on the interval $(i_x;i_{x+1})$ and have the same parity as $i_x$ and $i_{x+1}$. 

It remains to prove that the number of the brackets in the expression $\Phi$ is independent from the non-crossing partition. By the Lemma \ref{number of parts} the number of parts $|(\sigma|\widetilde\sigma)|=|\sigma|+|\widetilde\sigma|=n+1$ is independent from $\sigma$. On the other hand, the number of the brackets in the expression $\Phi$ equal to the difference between the number of vertices in the shaded area and the number of parts in the union of non-crossing partition and it's dual on the shaded area. It follows that the number of brackets is independent from the partition as well. Thus since the number of the brackets in the expression $\Phi$ for the non-crossing partition from Fig. \ref{ras} on the shaded area is equal to the number of the integers which lie on the interval $(i_x;i_{x+1})$ and have the same parity as $i_x$ and $i_{x+1}$ and is independent from the non-crossing partition, the same is true for an arbitrary non-crossing partition.

We now have to follow the common sign between each of the bracket $(i_x\ i_{x+1})$ which might be obtained by taking common factors out of brackets while transferring from (\ref{formula for a concordance vector}) to (\ref{br2}). However, we could take common factors out of the brackets starting from the first wedge-factor in any term and thus never obtain any changes of the sign in front of the any of the bracket. Thus the initial sign "$+$" in front of each of the bracket is being preserved.

\textbf{From the standard basis to the basis (\ref{Basis of V})}. To finish the proof it remains to explain how the expression (\ref{br2}) could be rewritten in terms of the basis (\ref{Basis of V}). We do it for an arbitrary bracket $(i_x\ i_{x+1})$:
$$(e_{i_x}\pm e_{i_{x+1}})=v_{i_{x}}-v_{i_{x}+2}+\ldots\pm v_{i_{x+1}-2}.$$
Since the sign "$\pm$" depends on the number of the integers which lie in the interval $(i_x;i_{x+1})$ and have  the same parity as $i_x$ and $i_{x+1}$ and the terms in the RHS of the expression above run over exactly these integers, the expression above is satisfied.
\end{proof}

\begin{proof}[Proof of Lemma \ref{diagonal}, case $j=i+1$]
Consider the non-crossing partition $(\sigma|\widetilde\sigma)=(\ldots\ i\ i+4|i+1\ i+3\ \ldots)$ drawn on the Fig. \ref{diag}.

\begin{figure}[ht]
\centering
\begin{tikzpicture}
    \draw (0,0) circle (2);

    \draw (120:2) -- node [left] {}(-60:2);
    \draw (80:2) -- node [left] {}(-20:2);

    \filldraw (-0.75,-0.3) node [left] {} circle (2pt);
    \filldraw (-0.92,-0.3) node [above] {} circle (2pt);
    \filldraw (-0.58,-0.3) node [right] {} circle (2pt);
    
    \filldraw (120:2) node [above] {$i$} circle (2pt);
    \filldraw (-60:2) node [below] {$i+1$} circle (2pt);
    \filldraw (30:2) node [right] {$i+2$} circle (2pt);
    \filldraw[fill=white] (80:2) node [above] {$i+3$} circle (2pt);
    \filldraw[fill=white] (-20:2) node [right] {$i+4$} circle (2pt);

\end{tikzpicture} \phantom{aaaaaaaaaa}
\begin{tikzpicture}
    \draw (0,0) circle (2);

    \draw (120:2) -- node [left] {}(-60:2);
    \draw (80:2) -- node [left] {}(-20:2);

    \filldraw (-0.75,-0.3) node [left] {} circle (2pt);
    \filldraw (-0.92,-0.3) node [above] {} circle (2pt);
    \filldraw (-0.58,-0.3) node [right] {} circle (2pt);
    
    \filldraw[fill=white] (120:2) node [above] {$i$} circle (2pt);
    \filldraw[fill=white] (-60:2) node [below] {$i+1$} circle (2pt);
    \filldraw[fill=white] (30:2) node [right] {$i+2$} circle (2pt);
    \filldraw (80:2) node [above] {$i+3$} circle (2pt);
    \filldraw (-20:2) node [right] {$i+4$} circle (2pt);

\end{tikzpicture}
\caption{Non-crossing partition $\sigma=(\ldots|i\ i+4|i+1\ i+3|i+2|\ldots)$}
    \label{diag}
\end{figure}
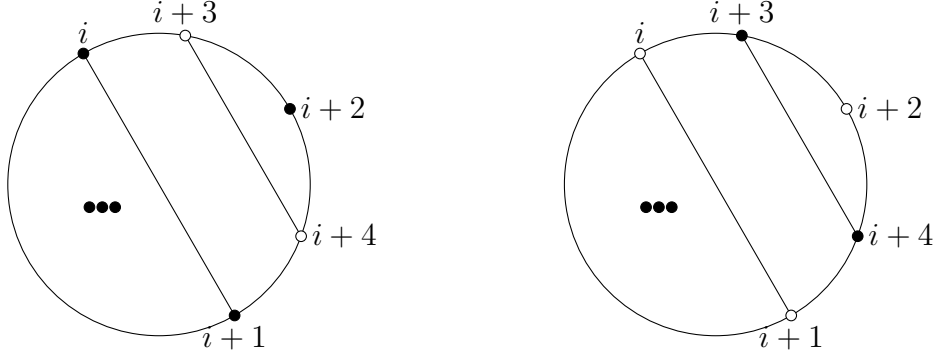

From Algorithm \ref{Algoritm_embedding} follows that the partition's part consisting from $2$ elements form a bracket in the expression \ref{br2}. Thus the expression for the concordance vector $w_{\sigma}$ looks as follows:
\begin{equation} \label{Lemma 1}
w_{\sigma}=(e_i-e_{i+4})\wedge(e_{i+1}+e_{i+3})\wedge(\ldots)=(v_i-v_{i+2})\wedge v_{i+1}\wedge(\ldots)=
\end{equation}
$$(v_i\wedge v_{i+1}+v_{i+1}\wedge v_{i+2})\wedge(\ldots)=v_i\wedge v_{i+1}\wedge(\ldots)+v_{i+1}\wedge v_{i+2}\wedge (\ldots).$$

We now convolve the expression above with the form $\omega$:
\begin{equation} \label{Qonvolution Lemma 1}
Q(w_{\sigma})=\omega(i,i+1)(-1)^{(1+2-1)}(\ldots)+\omega(i+1,i+2)(-1)^{(1+2-1)}(\ldots)+\ldots=
\end{equation}
$$\omega(i,i+1)(\ldots)+\omega(i+1,i+2)(\ldots)+\ldots.$$
We equate the expression above to $0$ because $w_{\sigma}\in\ker Q$ and compute the coefficient in front of one of the terms in $(\ldots)$. Assume for a moment that $i,i+2$ do not appear in $(\ldots)$, then we have:
$$\omega(i,i+1)+\omega(i+1,i+2)=0.$$

We now come back to the general case when $i,i+2$ might appear in $(\ldots)$. If in $(\ldots)$ appears a term which does not contain $i$ and $i+2$, when the computation of the coefficient in front of this term ends up with required equality $\omega(i,i+1)+\omega(i+1,i+2)=0.$ If there are no such term, i.e. each term in $(\ldots)$ contains $i$ or $i+2$, then after cutting out all $v_i\wedge v_i$ and $v_{i+2}\wedge v_{i+2}$ in the expression
$$(v_i\wedge v_{i+1}+v_{i+1}\wedge v_{i+2})\wedge(\ldots)=v_i\wedge v_{i+1}\wedge(\ldots)+v_{i+1}\wedge v_{i+2}\wedge (\ldots),$$ 
we fall into the first case considered.

\end{proof}
\begin{proof}[Proof of Lemma \ref{diagonal}, case $j>i+1$]
Let us consider the following concordance vectors corresponding to the non-crossing partitions from Fig. \ref{null}:
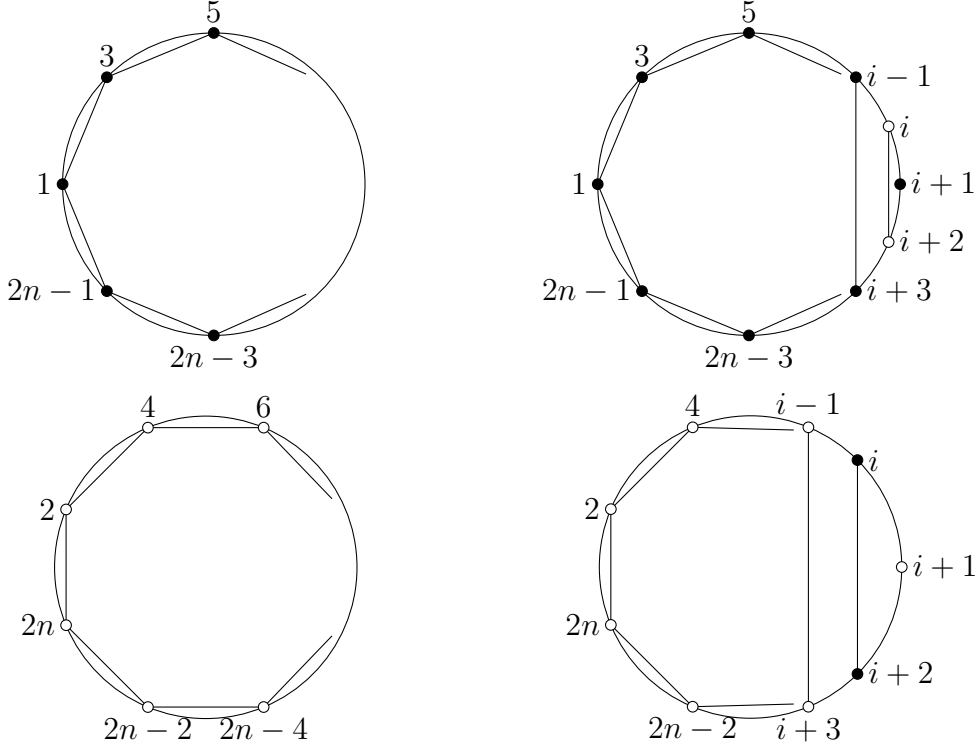
\begin{figure}[ht]
\centering
\begin{tikzpicture}
    \draw (0,0) circle (2);

    \filldraw (90:2) node [above] {$5$} circle (2pt);
    \filldraw (135:2) node [above] {$3$} circle (2pt);
    \filldraw (180:2) node [left] {$1$} circle (2pt);
    \filldraw (-135:2) node [left] {$2n-1$} circle (2pt);
    \filldraw (-90:2) node [below] {$2n-3$} circle (2pt);

    \draw (90:2) -- node [left] {}(50:1.9);
    \draw (90:2) -- node [left] {}(135:2);
    \draw (135:2) -- node [left] {}(180:2);
    \draw (180:2) -- node [left] {}(-135:2);
    \draw (-135:2) -- node [left] {}(-90:2);
    \draw (-90:2) -- node [left] {}(-50:1.9);

\end{tikzpicture}\phantom{aaaaaaaaaa}
\begin{tikzpicture}
    \draw (0,0) circle (2);

    \draw (90:2) -- node [left] {}(50:1.9);
    \draw (90:2) -- node [left] {}(135:2);
    \draw (135:2) -- node [left] {}(180:2);
    \draw (180:2) -- node [left] {}(-135:2);
    \draw (-135:2) -- node [left] {}(-90:2);
    \draw (-90:2) -- node [left] {}(-50:1.9);
    \draw (-45:2) -- node [left] {}(45:2);
    \draw (-22.5:2) -- node [left] {}(22.5:2);

    \filldraw (45:2) node [right] {$i-1$} circle (2pt);
    \filldraw (90:2) node [above] {$5$} circle (2pt);
    \filldraw (135:2) node [above] {$3$} circle (2pt);
    \filldraw (180:2) node [left] {$1$} circle (2pt);
    \filldraw (-135:2) node [left] {$2n-1$} circle (2pt);
    \filldraw (-90:2) node [below] {$2n-3$} circle (2pt);
    \filldraw (-45:2) node [right] {$i+3$} circle (2pt);
    \filldraw[fill=white] (22.5:2) node [right] {$i$} circle (2pt);
    \filldraw[fill=white] (-22.5:2) node [right] {$i+2$} circle (2pt);
    \filldraw (0:2) node [right] {$i+1$} circle (2pt);
 
\end{tikzpicture}\phantom{a}

\begin{tikzpicture}
    \draw (0,0) circle (2);

    \draw (67.5:2) -- node [left] {}(28.5:1.9);
    \draw (-67.5:2) -- node [left] {}(-28.5:1.9);
    \draw (67.5:2) -- node [left] {}(112.5:2);
    \draw (112.5:2) -- node [left] {}(157.5:2);
    \draw (-157.5:2) -- node [left] {}(157.5:2);
    \draw (-112.5:2) -- node [left] {}(-157.5:2);
    \draw (-112.5:2) -- node [left] {}(-67.5:2);

    \filldraw[fill=white] (67.5:2) node [above] {$6$} circle (2pt);
    \filldraw[fill=white] (112.5:2) node [above] {$4$} circle (2pt);
    \filldraw[fill=white] (157.5:2) node [left] {$2$} circle (2pt);
    \filldraw[fill=white] (-67.5:2) node [below] {$2n-4$} circle (2pt);
    \filldraw[fill=white] (-112.5:2) node [below] {$2n-2$} circle (2pt);
    \filldraw[fill=white] (-157.5:2) node [left] {$2n$} circle (2pt);

\end{tikzpicture}\phantom{aaaaaaaaaaaa}
\begin{tikzpicture}
    \draw (0,0) circle (2);

    \draw (72.5:1.90) -- node [left] {}(112.5:2);
    \draw (112.5:2) -- node [left] {}(157.5:2);
    \draw (-157.5:2) -- node [left] {}(157.5:2);
    \draw (-112.5:2) -- node [left] {}(-157.5:2);
    \draw (-112.5:2) -- node [left] {}(-72.5:1.9);
    \draw (67.5:2) -- node [left] {}(-67.5:2);
    \draw (45:2) -- node [left] {}(-45:2);

    \filldraw[fill=white] (67.5:2) node [above] {$i-1$} circle (2pt);
    \filldraw[fill=white] (112.5:2) node [above] {$4$} circle (2pt);
    \filldraw[fill=white] (157.5:2) node [left] {$2$} circle (2pt);
    \filldraw[fill=white] (-67.5:2) node [below] {$i+3$} circle (2pt);
    \filldraw[fill=white] (-112.5:2) node [below] {$2n-2$} circle (2pt);
    \filldraw[fill=white] (-157.5:2) node [left] {$2n$} circle (2pt);
    \filldraw (45:2) node [right] {$i$} circle (2pt);
    \filldraw (-45:2) node [right] {$i+2$} circle (2pt);
    \filldraw[fill=white] (0:2) node [right] {$i+1$} circle (2pt);

\end{tikzpicture}
    \caption{Non-crossing partitions from the proof of Lemma \ref{diagonal}, case $j>i+1$.}
    \label{null}
\end{figure}
$$w_{13\ldots 2n-1|2|4|\ldots|2n}=(e_1+e_3)\wedge(e_3+e_5)\wedge\ldots\wedge(e_{2n-3}+e_{2n-1})=v_1\wedge v_3\wedge\ldots\wedge v_{2n-3}$$

$$w_{1|24\ldots 2n|3|5|\ldots|2n-1}=(e_2+e_4)\wedge(e_4+e_6)\wedge\ldots\wedge(e_{2n-2}+e_{2n})=v_2\wedge \wedge\ldots\wedge v_{2n-2}$$

$$w_{13\ldots i-1,i+3,i+5\ldots 2n-1|i,i+2|2|4|\ldots|i-2|i+1|i+4|\ldots|2n}=$$
$$(e_1+e_3)\wedge(e_3+e_5)\wedge\ldots\wedge(e_{i-3}+e_{i-1})\wedge(e_{i-1}-e_{i+3})\wedge(e_{i+3}+e_{i+5})\wedge\ldots\wedge(e_{2n-3}+e_{2n-1})\wedge(e_i+e_{i+2})=$$
$$v_1\wedge v_3\ldots\wedge v_{i-3}\wedge (v_{i-1}-v_{i+1})\wedge v_{i+3}\wedge\ldots\wedge v_{2n-3}\wedge v_{i}$$

$$w_{24\ldots i-1,i+3,i+5\ldots 2n|i,i+2|1|3|\ldots|i-2|i+1|i+2|\ldots|2n-1}=$$
$$(e_2+e_4)\wedge(e_4+e_6)\wedge\ldots\wedge(e_{i-3}+e_{i-1})\wedge(e_{i-1}-e_{i+3})\wedge(e_{i+3}+e_{i+5})\wedge\ldots\wedge(e_{2n-2}+e_{2n})\wedge(e_i+e_{i+1})=$$
$$v_2\wedge v_4\ldots\wedge v_{i-3}\wedge (v_{i-1}-v_{i+1})\wedge v_{i+3}\wedge\ldots\wedge v_{2n-2}\wedge v_{i}$$

In all four expressions we obtained each index occurres exactly once, so we get the following.

Convolving the first vector with the symplectic form we get: $\forall i,j\in \{1,3,\ldots,2n-3\},\ i+1<j\ \omega(i,j)=0$.

Convolving the second vector with the symplectic form we get: $\forall i,j\in \{2,4,\ldots,2n-2\},\ i+1<j\ \omega(i,j)=0$.

Convolving the third vector with the symplectic form we get: $\forall i\in \{2,4,\ldots,2n-2\},\ j\in \{1,3,\ldots,2n-3\},\ i+1<j\ \omega(i,j)=0$.

Convolving the fourth vector with the symplectic form we get: $\forall i\in \{1,3,\ldots,2n-3\},\ j\in \{2,4,\ldots,2n-2\},\ i+1<j\ \omega(i,j)=0$.

\end{proof}

To help the reader follow all details and see this natural occurrence of the symplectic form $\Lambda_{2n-2}$ from the concordance we provide our exposition with a simple example $n=3$.
\begin{example} \label{Example of algorithm}

Let us decompose the vectors $w_{\sigma}$ via a basis (\ref{Basis of V}) of the space $V$ following Algorithm \ref{Algoritm_embedding}:
$$w_{1|23}=e_2\wedge e_3+e_2\wedge e_5+e_3\wedge e_6 +e_5\wedge e_6=(e_2-e_6)\wedge (e_3+e_5)=(v_2-v_4)\wedge v_3$$
$$w_{13|2}=e_1\wedge e_2+e_1\wedge e_4+e_2\wedge e_5+e_4\wedge e_5=(e_1-e_5)\wedge (e_2+e_4)=(v_1-v_3)\wedge v_2$$
$$w_{12|3}=e_1\wedge e_4+e_3\wedge e_4+e_1\wedge e_6+e_3\wedge e_6=(e_1+e_3)\wedge (e_4+e_6)=v_1\wedge v_4$$  
$$w_{1|2|3}=e_2\wedge e_4+e_4\wedge e_6+e_2\wedge e_6=(e_2+e_4)\wedge(e_4+e_6)=v_2\wedge v_4$$
$$w_{123}=e_1\wedge e_3+e_3\wedge e_5+e_1\wedge e_5=(e_1+e_3)\wedge (e_3+e_5)=v_1\wedge v_3$$

Now we can explicitly write down the condition $w_{\sigma}\in \ker Q$ from Theorem \ref{Classification}:
$$w_{1|23}\in \ker Q\text{, i.e.}\ Q(w_{1|23})=\omega(2,3)+\omega(3,4)=0 \Longrightarrow \omega(2,3)=-\omega(3,4)$$
$$w_{13|2}\in \ker Q\text{, i.e.}\ Q(w_{13|2})=\omega(1,2)+\omega(2,3)=0 \Longrightarrow \omega(1,2)=-\omega(2,3)$$
$$w_{12|3}\in \ker Q\text{, i.e.}\ Q(w_{12|3})=\omega(1,4)=0 \Longrightarrow \omega(1,4)=0$$
$$w_{1|23}\in \ker Q\text{, i.e.}\ Q(w_{1|2|3})=\omega(2,4)=0 \Longrightarrow \omega(2,4)=0$$
$$w_{123}\in \ker Q\text{, i.e.}\ Q(w_{123})=\omega(1,3)=0 \Longrightarrow \omega(1,3)=0$$

Denote $\omega(1,2)$ as $a$ and obtain:
$$\omega=\left(\begin{matrix}
0 & a & 0 & 0\\
-a & 0 & -a & 0 \\
0 & a & 0 & a \\
0 & 0 & -a & 0
\end{matrix}\right).$$
\end{example}

\subsection{Lagrangian concordance} \label{subsec: lagrangian concordance}
Analyzing Algorithm \ref{Algoritm_embedding} we are able to obtain  the explicit relation between Plucker coordinates of a point $LG_{\geq 0}(n-1, V)$ associated with $e \in \overline{E}_n$ and the grove measurements. At the beginning we give an explicit form for the concordance vectors $w_\sigma|_V$.

\begin{definition} \label{def:envelope}
For two connected components $B_i=\{i_1,\ldots,i_k\}\subseteq[2n]$ and $B_j=\{j_1,\ldots,j_l\}\subseteq[2n]$ of a non-crossing partition $\sigma$ (or $\widetilde\sigma$) we say that a chord $\{i_x,i_{x+1}\}\subseteq B_i$ such that $\{i_x,i_{x+1}\}\ne\{i_1,i_k\}$ {\itshape envelopes} $B_j$ if $B_j$ lies in the area bounded by the chord $\{i_x,i_{x+1}\}$ and the boundary circle.
\end{definition}

\begin{theorem} \label{theorem:cactus-plu}
Plucker coordinates of a concordance vector $w_{\sigma}|_V$ thought as a point of $LG(n-1,V)$ are equal either to $0$ or to $1$.
\end{theorem}
\begin{proof}
Consider any connected component $B_i=\{i_1,\ldots,i_k\}\subseteq[2n]$ of $\sigma$ (or $\widetilde\sigma$). By the second step of Algorithm \ref{Algoritm_embedding} any $\{i_x,i_{x+1}\}\subseteq B_i$ such that $\{i_x,i_{x+1}\}\ne\{i_1,i_k\}$ corresponds to some bracket in the expression for $w_\sigma\in\bigwedge^{n-1}V$. According to the fourth step of Algorithm \ref{Algoritm_embedding} this bracket is 
\begin{equation} \label{eq:step 4}
(v_{i_x}-v_{i_{x}+2}+v_{i_{x}+4}-\ldots (-1)^{\frac{i_{x+1}-2-i_x}{2}}v_{i_{x+1}-2}).
\end{equation}
Note that summands in different brackets of the form \eqref{eq:step 4} for two different chords either do not intersect (by the vectors $v_i$) or one contains another as summands. We are in the second case if and only if the chord corresponding to the first bracket envelopes the connected component $B_j$ to which the chord corresponding to the second bracket belongs to.

The point $p_\sigma|_V\in LG_{\ge 0}(n-1,V)$ corresponding to $w_\sigma\in\bigwedge^{n-1}V$ is given by the span of the set of vectors given by the collection of wedge-factors in the expression for $w_\sigma\in\bigwedge^{n-1}V$, which is a pure wedge-product. If $\{i_x,i_{x+1}\}$ envelopes $B_i$ then by summing vectors corresponding to $B_j$ with the vector corresponding to $\{i_x,i_{x+1}\}$ we can obtain the set of vectors whose terms are different among each other without changing the hall span. Thus writing back a pure wedge-product from this set of vectors we will obtain a wedge-product whose wedge-factors do not contain common summands. Thus after expanding the brackets and rewriting this wedge-product in the standard basis $\{v_{i_1}\wedge\ldots\wedge v_{i_{n-1}}|i_1<\ldots <i_{n-1}\}$ there will be no common terms. Therefore all the nonzero coefficients of the final form of \eqref{eq:step 4} are equal either to $1$ or to $-1$. By Corollary \ref{cor: coeff are of the same sign} all these coefficiens are positive. Thus all of them are equal to $1$.
\end{proof}

We now illustrate the proof of Theorem \ref{theorem:cactus-plu} by the following example.
\begin{example} \label{ex:proof of coeff}
Consider a non-crossing partition $\sigma=(\bar1|\bar2,\bar5,\bar8|\bar3|\bar4|\bar6,\bar7)$ (see Fig. \ref{fig:proof of coeff}). The expression for a concordance vector $w_\sigma$ provided by Algorithm \ref{Algoritm_embedding} is given by
\begin{equation}
\label{eq:wsigma1}
w_\sigma=(v_2-v_4+v_6-v_8+v_{10}-v_{12}+v_{14})\wedge(v_3-v_5+v_7)\wedge v_4\wedge v_6\wedge (v_9-v_{11}+v_{13})\wedge (v_{10}-v_{12})\wedge v_{11}.
\end{equation}

Note that the chord $\{2,16\}$ envelopes the connected components $\{4,6,8\}$ and $\{10,14\}$, and the chord $\{9,15\}$ envelopes the connected component $\{11,13\}$. As it was claimed in the proof of Theorem \ref{theorem:cactus-plu}, the bracket corresponding to $\{2,16\}$ contains as the terms the brackets corresponding to $\{4,6,8\}$, i.e. $(v_2-v_4+v_6-v_8+v_{10}-v_{12}+v_{14})$ contains $v_4$ and $v_6$. Similarly for $\{2,16\}$ envelopes $\{10,14\}$: the bracket corresponding to $\{2,16\}$ contains as the terms the brackets corresponding to $\{10,14\}$, i.e. $(v_2-v_4+v_6-v_8+v_{10}-v_{12}+v_{14})$ contains $v_{10}-v_{12}$. Finally, for $\{9,15\}$ envelopes $\{11,13\}$: the bracket corresponding to $\{9,15\}$ contains as the terms the brackets corresponding to $\{11,13\}$, i.e. $(v_9-v_{11}+v_{13})$ contains $v_{11}$. Thus we rewrite \eqref{eq:wsigma1} in such a way that any two brackets do not contain common terms:
$$w_\sigma=(v_2+v_{14})\wedge(v_3-v_5+v_7)\wedge v_4\wedge v_6\wedge (v_9+v_{13})\wedge (v_{10}-v_{12})\wedge v_{11}.$$
Expanding all the brackets we obtain an expression for $w_\sigma$ with all coefficients equal $1$.
\end{example}

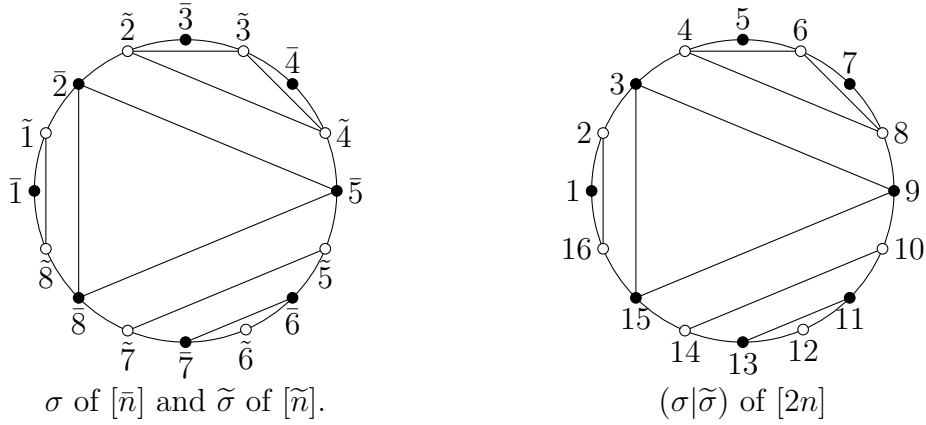
\begin{figure}[ht]
\centering
\begin{tikzpicture}
    \draw (0,0) circle (2);

    \draw (157.5:2) -- node [left] {}(202.5:2);
    
    \draw (135:2) -- node [left] {}(0:2);
    \draw (225:2) -- node [left] {}(135:2);
    \draw (225:2) -- node [left] {}(0:2);

    \draw (112.5:2) -- node [left] {}(22.5:2);
    \draw (112.5:2) -- node [left] {}(67.5:2);
    \draw (67.5:2) -- node [left] {}(22.5:2);

    \draw (337.5:2) -- node [left] {}(247.5:2);
    \draw (315:2) -- node [left] {}(270:2);

    \filldraw[fill=black] (0:2) node [right] {$\bar5$} circle (2pt);
    \filldraw[fill=white] (22.5:2) node [right] {$\tilde4$} circle (2pt);
    \filldraw[fill=black] (45:2) node [above] {$\bar4$} circle (2pt);
    \filldraw[fill=white] (67.5:2) node [above] {$\tilde3$} circle (2pt);
    \filldraw[fill=black] (90:2) node [above] {$\bar3$} circle (2pt);
    \filldraw[fill=white] (112.5:2) node [above] {$\tilde2$} circle (2pt);
    \filldraw[fill=black] (135:2) node [left] {$\bar2$} circle (2pt);
    \filldraw[fill=white] (157.5:2) node [left] {$\tilde1$} circle (2pt);
    \filldraw[fill=black] (180:2) node [left] {$\bar1$} circle (2pt);
    \filldraw[fill=white] (202.5:2) node [below] {$\tilde8$} circle (2pt);
    \filldraw[fill=black] (225:2) node [below] {$\bar8$} circle (2pt);
    \filldraw[fill=white] (247.5:2) node [below] {$\tilde7$} circle (2pt);
    \filldraw[fill=black] (270:2) node [below] {$\bar7$} circle (2pt);
    \filldraw[fill=white] (293.5:2) node [below] {$\tilde6$} circle (2pt);
    \filldraw[fill=black] (315:2) node [below] {$\bar6$} circle (2pt);
    \filldraw[fill=white] (337.5:2) node [below] {$\tilde5$} circle (2pt);

    \draw (0,-2.8) node {$\sigma$ of $[\bar n]$ and $\widetilde{\sigma}$ of $[\widetilde{n}]$.};
\end{tikzpicture} \phantom{aaaaaaaaaa}
\begin{tikzpicture}
    \draw (0,0) circle (2);

    \draw (157.5:2) -- node [left] {}(202.5:2);
    
    \draw (135:2) -- node [left] {}(0:2);
    \draw (225:2) -- node [left] {}(135:2);
    \draw (225:2) -- node [left] {}(0:2);

    \draw (112.5:2) -- node [left] {}(22.5:2);
    \draw (112.5:2) -- node [left] {}(67.5:2);
    \draw (67.5:2) -- node [left] {}(22.5:2);

    \draw (337.5:2) -- node [left] {}(247.5:2);
    \draw (315:2) -- node [left] {}(270:2);

    \filldraw[fill=black] (0:2) node [right] {$9$} circle (2pt);
    \filldraw[fill=white] (22.5:2) node [right] {$8$} circle (2pt);
    \filldraw[fill=black] (45:2) node [above] {$7$} circle (2pt);
    \filldraw[fill=white] (67.5:2) node [above] {$6$} circle (2pt);
    \filldraw[fill=black] (90:2) node [above] {$5$} circle (2pt);
    \filldraw[fill=white] (112.5:2) node [above] {$4$} circle (2pt);
    \filldraw[fill=black] (135:2) node [left] {$3$} circle (2pt);
    \filldraw[fill=white] (157.5:2) node [left] {$2$} circle (2pt);
    \filldraw[fill=black] (180:2) node [left] {$1$} circle (2pt);
    \filldraw[fill=white] (202.5:2) node [left] {$16$} circle (2pt);
    \filldraw[fill=black] (225:2) node [below] {$15$} circle (2pt);
    \filldraw[fill=white] (247.5:2) node [below] {$14$} circle (2pt);
    \filldraw[fill=black] (270:2) node [below] {$13$} circle (2pt);
    \filldraw[fill=white] (293.5:2) node [below] {$12$} circle (2pt);
    \filldraw[fill=black] (315:2) node [below] {$11$} circle (2pt);
    \filldraw[fill=white] (337.5:2) node [right] {$10$} circle (2pt);

    \draw (0,-2.8) node {$(\sigma|\widetilde{\sigma})$ of $[2n]$};
\end{tikzpicture}
    \caption{Non-crossing partition $\sigma=(\bar1|\bar2,\bar5,\bar8|\bar3|\bar4|\bar6,\bar7)$ of $[\bar n]$, $\widetilde\sigma$ of $[\widetilde n]$ and $(\sigma|\widetilde\sigma)$ of $[2n]$. }
    \label{fig:proof of coeff}
\end{figure}

Keeping in mind Corollary \ref{rem: algorithm vs Plucker coordinates} one can restate Theorem \ref{theorem:cactus-plu} as follows.
\begin{corollary}
     For a hollow cactus network $P(\sigma) \in \overline{E}_n$ and a point $p(\sigma) \in Gr_{\geq 0}(n-1, 2n)$ associated with it, all Plucker coordinates of $p(\sigma)|_{V} \in LG_{\geq 0}(n-1, 2n-2)$ are equal either $0$ or $1.$ 
\end{corollary}

The following four definitions are to construct a formal language to describe the combinatorial constructions in Algorithm \ref{Algoritm_embedding} and the proof of Theorem \ref{theorem:cactus-plu}. We will use this language to define a Lagrangian version of concordance (Definition \ref{def: Lagrangian concordance}) and state Corollary \ref{thm: Langrangian concordance}.
\begin{definition} \label{def:extension}
    Let $\sigma$ be a non-crossing partition  and $\overline{B}_j=\{(\overline{i}_1, \overline{i}_2, \dots, \overline{i}_k) \in \sigma| \ k>1\}$ be its connected component, then \textit{extension} $\mathrm{ext}(\overline{B}_j)$ is the following sequence of indices:
    $$\mathrm{ext}(\overline{B}_j)=(\mathrm{ext}(\overline{i}_1\overline{i}_2), \mathrm{ext}(\overline{i}_2\overline{i}_3), \dots, \mathrm{ext}(\overline{i}_{k-1}\overline{i}_k)  ),$$
    where  $\mathrm{ext}(\overline{i}_{l}\overline{i}_{l+1})=(2i_l-1, 2i_l+1, 2i_l+3, \dots, 2i_{l+1}-3). $ 

       Let $\widetilde{\sigma}$ be a dual non-crossing partition to $\sigma$ and let $\widetilde{B}_j=\{(\widetilde{i}_1, \widetilde{i}_2, \dots, \widetilde{i}_k) \in \widetilde{\sigma}| \ k>1\}$ be its connected component, then \textit{extension} $\mathrm{ext}(\overline{B}_j)$ is the following sequences of indices:
         
         $$\mathrm{ext}(\widetilde{B}_j)=(\mathrm{ext}(\widetilde{i}_1\widetilde{i}_2), \mathrm{ext}(\widetilde{i}_2\widetilde{i}_3), \dots, \mathrm{ext}(\widetilde{i}_{k-1}\widetilde{i}_k)  ),$$
         where $\mathrm{ext}(\widetilde{i}_{l}\widetilde{i}_{l+1})=(2i_l, 2i_l+2, 2i_l+4, \dots, 2i_{l+1}-2). $ 
   
\end{definition}

The following definition is a reformulation of Definition \ref{def:envelope} in the sense of Definition \ref{def:extension}.
\begin{definition}
Let $\sigma$ be a non-crossing partition  
and  $\overline{B}_{j_1} \in \sigma$ and $\overline{B}_{j_2} \in \sigma$ 
be two its connected components of size greater than $1$, then we will say that  $\overline{B}_{j_2}$ 
{\itshape envelopes}  $\overline{B}_{j_1}$ 
if there is $(\overline{i}_{l-1}\overline{i}_{l}) \in \overline{B}_{j_2} $ 
such that $\overline{B}_{j_1} \subset \mathrm{ext}(\overline{i}_{l-1}\overline{i}_{l})$.
For $\overline{B}_{j_2}$
envelopes $\overline{B}_{j_1}$ we will use the notation 
$ \overline{B}_{j_1} \in \mathrm{env}(\overline{B}_{j_2})$.

Similarly, we define that $\widetilde{B}_{j_2}$ 
{\itshape envelopes}  $\widetilde{B}_{j_1}$ for connected components $\widetilde{B}_{j_1},\widetilde{B}_{j_2}$ of $\widetilde{\sigma}$.
\end{definition}

Since we consider only non-crossing partitions  for any   $ \overline{B}_{j_1} \in \mathrm{env}(\overline{B}_{j_2})$ 
there is only one $(\overline{i}_{l-1}\overline{i}_{l}) \in \overline{B}_{j_2} $ 
which envelopes  $ \overline{B}_{j_1}$ 
. Similarly for the dual case.

\begin{definition} \label{def: Lext}
    Consider a non-crossing partition $\sigma.$ The \textit{Lagrangian extension} of  $\sigma$  is the  union of the index sequences $\mathrm{Lext}(\sigma)=\bigsqcup \limits_{\overline{B}_k \in \sigma: \ |\overline{B}_k| >1 } \mathrm{Lext}(\overline{B}_k) \bigsqcup \limits_{\widetilde{B}_k \in \widetilde{\sigma}: \ |\widetilde{B}_k| >1 }\mathrm{Lext}(\widetilde{B}_k),$ where $\mathrm{Lext}(\overline{B}_k)$ and $\mathrm{Lext}(\widetilde{B}_k)$ are defined as follows:

            For $\overline{B}_j=\{(\overline{i}_1, \overline{i}_2, \dots, \overline{i}_k) \in \sigma| \ k>1\}$ define $$\mathrm{Lext}(\overline{B}_j)=( \mathrm{ext}(\overline{i}_1\overline{i}_2)  , \mathrm{ext}(\overline{i}_2\overline{i}_3), \dots, \mathrm{ext}(\overline{i}_{k-1}\overline{i}_k) ) \setminus \bigcup \limits_{\overline{B}_l \in \mathrm{env}(\overline{B}_j)} \mathrm{ext}(\overline{B}_l);$$ 
            where $\mathrm{ext}(\overline{i}_{s-1}\overline{i}_{s} ) \setminus \bigcup \limits_{\overline{B}_l \in \mathrm{env}(\overline{B}_j)} $ is a \textit{sub-component} of the component $\mathrm{Lext}(\sigma);$ 

            and $\mathrm{Lext}(\widetilde{B}_k)$ is defined similarly.
        
      
    
\end{definition}

Analyzing the proof of Theorem \ref{theorem:cactus-plu} we obtain the analogue of the concordance for $LG_{\geq 0}(n-1,2n-2)$ (compare Definition \ref{def: Lagrangian concordance} with Definition \ref{co-con}). Roughly speaking, Definition \ref{def: Lagrangian concordance} provides a combinatorial description of nonzero Plucker coordinates of the point $p_\sigma|_V\in LG_{\ge 0}(n-1,V)$ associated with the concordance vector $w_\sigma$.

\begin{definition} \label{def: Lagrangian concordance}
    We will say that a subset $I \subset [2n-2], \ I=|n-1|$ is \textit{Lagrangian concordant} with $\sigma$ if $I$ shares exactly one element with each sub-component of $\mathrm{Lext}(\sigma).$
\end{definition}

Continuing Example \ref{ex:proof of coeff} we obtain:
\begin{example} 
Consider a non-crossing partition $\sigma=(\bar1|\bar2,\bar5,\bar8|\bar3|\bar4|\bar6,\bar7)$ (see Fig. \ref{fig:proof of coeff}).

Then:

$\bullet$ $\text{ext}(\overline{B}_1)=((3,5,7),(9,11,13))$ and $\text{ext}(\overline{B}_2)=(11)$. Note that $\overline{B}_2\in\text{env}(\overline{B}_1)$.

Thus $\text{Lext}(\overline{B}_1)=((3,5,7)-\text{ext}(\overline{B}_2),(9,11,13)-\text{ext}(\overline{B}_2))=((3,5,7),(9,13))$.

$\bullet$ $\text{ext}(\widetilde{B}_1)=(2,4,6,8,10,12)$, $\text{ext}(\widetilde{B}_2)=(4,6)$ and $\text{ext}(\widetilde{B}_3)=(10,12)$. Note that $\widetilde{B}_2,\widetilde{B}_3\in\text{env}(\widetilde{B}_1)$.

Thus $\text{Lext}(\widetilde{B}_1)=\text{ext}(\widetilde{B_1})-(\text{ext}(\widetilde{B}_2)\cup\text{ext}(\widetilde{B}_3))=(2,8)$, $\text{Lext}(\widetilde{B}_2)=(4,6)$ and $\text{Lext}(\widetilde{B}_3)=(10,12)$.

Taking all together we conclude that $\text{Lext}(\sigma)=(11)(3,5,7)(9,13)(4)(6)(10,12)(2,8)$.
The set of subsets $I\subset[2n-2]$, $|I|=n-1$ which are Lagrangian concordant with $\sigma$ can be obtained by choosing one element from each bracket. One can compare these $I$ with the nonzero terms in the expression for $w_\sigma$ from Example \ref{ex:proof of coeff}.

\end{example}

\begin{lemma} \label{lemma: Lagrng Plucker}
    Denote by $\Delta_I(w_\sigma|_V)$ the Plucker coordinates of the point $w_\sigma|_V\in LG_{\ge 0}(n-1,V)$, then for any given non-crossing partition $\sigma$ we have:
    $$\{I\subset[2n-2]|\ I=|n-1|, I\ \text{is Lagrangian concordant with}\ \sigma\}=\{I|\ \Delta_I(w_\sigma|_V)\ne 0\}.$$
\end{lemma}
\begin{proof}
    Definition \ref{def: Lagrangian concordance} provides a combinatorial description of the process of enumerating all non-zero Plucker coordinates of the point $w_\sigma|_V$, which was described in the proof of Theorem \ref{theorem:cactus-plu}.
\end{proof}

The statement below is the Lagrangian analog to the formula for Plucker coordinates for the Lam's embedding (see Theorem \ref{Emb} or Theorem \ref{thm:co-Emb}).
\begin{corollary} \label{thm: Langrangian concordance}
    Consider $e \in \overline{E}_n$ and a point $Y \in LG_{\geq 0}(n-1, V)$ associated with $e \in \overline{E}_n,$ then the following holds: 
    $$\Delta_I(Y)=\sum \limits_{(I,  \mathrm{Lext}(\sigma))} L_{\sigma}, $$
    where the sum is over all $\sigma$ such that $I \subset [2n-2], \ |I|=n-1$ is Lagrangian concordant with $\sigma.$
\end{corollary}
\begin{proof}
Indeed, 
    $$Y=\sum \limits_{\sigma}L_\sigma w_\sigma|_V.$$
    By Lemma \ref{lemma: Lagrng Plucker} we have
    $$w_\sigma|_V=\sum \limits_{(I, \mathrm{Lext}(\sigma))}e_I,$$
     where $e_I=e_{i_1}\wedge\ldots\wedge e_{i_{n-1}}| I=\{i_1, \dots,i_{n-1}\}$  and the sum is over all $I$ such that  $I$  is Lagrangian concordant with $\sigma.$ Therefore we obtain
$$Y=\sum \limits_{\sigma}L_{\sigma}\sum \limits_{(I,  \mathrm{Lext}(\sigma))}e_I=\sum \limits_{I} \sum \limits_{(I,  \mathrm{Lext}(\sigma))} L_{\sigma} e_I=\sum \limits_{I} \left(\sum \limits_{(I,  \mathrm{Lext}(\sigma))} L_{\sigma} \right)e_I.$$

\end{proof}

\subsection{The special property of the concordance basis} \label{sec: special property}

The main goal of this subsection is to describe a special property of the action of Lam's generators on the concordance vectors. 
Note that the generators $\mathfrak{u}_2|_V,\ldots,\mathfrak{u}_{2n-1}|_V$ generate the symplectic Lie algebra $\mathfrak{sp}_{2n-2}$ (see the proof of Theorem \ref{Theorem 5.4}), therefore the results of this subsection hold for $\mathfrak{sp}_{2n-2}$ as well. We state it precisely in Corollary \ref{from Lam to sp}.

Define the following combinatorial operation on non-crossing partitions of $[\bar n]$.

\textbf{Isolating a boundary vertex}: Given a non-crossing partition $(\sigma|\widetilde\sigma)$ and an integer number $i\in[2n]$ define a new non-crossing partition $g_i\sigma$ obtained from $(\sigma|\widetilde\sigma)$ by isolating the vertex $i$ (see Fig. \ref{isolating a boundary vertex}). 

In the case when $i$ is an isolated vertex in the non-crossing partition $(\sigma|\widetilde\sigma)$ the operation $g_i$ acts trivially. 

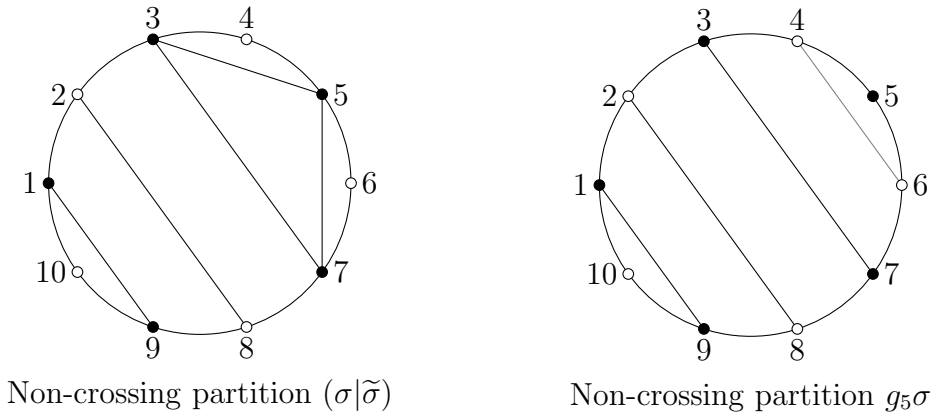
\begin{figure}[ht]
\centering
\begin{tikzpicture}
    \draw (0,0) circle (2);

    \draw (36:2) -- node [left] {}(108:2);
    \draw (108:2) -- node [left] {}(-36:2);
    \draw (36:2) -- node [left] {}(-36:2);
    \draw (144:2) -- node [left] {}(-72:2);
    \draw (180:2) -- node [left] {}(-108:2);

    \filldraw[fill=white] (0:2) node [right] {$6$} circle (2pt);
    \filldraw[fill=black] (36:2) node [right] {$5$} circle (2pt);
    \filldraw[fill=white] (72:2) node [above] {$4$} circle (2pt);
    \filldraw[fill=black] (108:2) node [above] {$3$} circle (2pt);
    \filldraw[fill=white] (144:2) node [left] {$2$} circle (2pt);
    \filldraw[fill=black] (180:2) node [left] {$1$} circle (2pt);
    \filldraw[fill=black] (-36:2) node [right] {$7$} circle (2pt);
    \filldraw[fill=white] (-72:2) node [below] {$8$} circle (2pt);
    \filldraw[fill=black] (-108:2) node [below] {$9$} circle (2pt);
    \filldraw[fill=white] (-144:2) node [left] {$10$} circle (2pt);

    \draw (0,-2.8) node {Non-crossing partition $(\sigma|\widetilde\sigma)$};
\end{tikzpicture}\phantom{aaaaaaaaaa}\begin{tikzpicture}
    \draw (0,0) circle (2);

    \draw (108:2) -- node [left] {}(-36:2);
    \draw[draw=gray] (0:2) -- node [left] {}(72:2);
    \draw (144:2) -- node [left] {}(-72:2);
    \draw (180:2) -- node [left] {}(-108:2);

    \filldraw[fill=white] (0:2) node [right] {$6$} circle (2pt);
    \filldraw[fill=black] (36:2) node [right] {$5$} circle (2pt);
    \filldraw[fill=white] (72:2) node [above] {$4$} circle (2pt);
    \filldraw[fill=black] (108:2) node [above] {$3$} circle (2pt);
    \filldraw[fill=white] (144:2) node [left] {$2$} circle (2pt);
    \filldraw[fill=black] (180:2) node [left] {$1$} circle (2pt);
    \filldraw[fill=black] (-36:2) node [right] {$7$} circle (2pt);
    \filldraw[fill=white] (-72:2) node [below] {$8$} circle (2pt);
    \filldraw[fill=black] (-108:2) node [below] {$9$} circle (2pt);
    \filldraw[fill=white] (-144:2) node [left] {$10$} circle (2pt);

    \draw (0,-2.8) node {Non-crossing partition $g_5\sigma$};
\end{tikzpicture}
    \caption{Isolating the boundary vertex $5$ in $\sigma=(\bar{1},\bar{5}|\bar{2},\bar{3},\bar{4})$}
    \label{isolating a boundary vertex}
\end{figure}

We will prove the following remarkable property of the concordance basis.
\begin{theorem} \label{cristal V}
The action of the Lam generators $\mathfrak{u}_i$ (see Definition \ref{Lam_group}) maps a concordance vector $w_\sigma$ to another concordance vector or to the zero. More precisely, $\mathfrak{u}_i(w_{\sigma})=0$ if the vertex $i$ is an isolated vertex in $\sigma$ or it is dual $\widetilde{\sigma}$ and $\mathfrak{u}_i(w_{\sigma})=w_{g_i\sigma}$ otherwise.

The same is true for the operators $\mathfrak{u}_i|_V$.
\end{theorem}

\begin{remark}
Note that the statement of Theorem \ref{cristal V} always holds for the Chevalley generators. 
\end{remark}


The following lemma describes the changes of grove coordinates under the action of Lam generator $u_i(a)$.

\begin{lemma} \cite[Proposition 5.15]{L} \label{changes of grove coordinates}

Let $X=\sum\limits_{\sigma}L_{\sigma}w_{\sigma}$ be a point in $Gr(n-1,2n)_{\ge 0}\cap \mathbb PH$. 

If $i=2\bar{k}-1$ is odd, then define 
$$L'_{\sigma}=\begin{cases}
L_{\sigma}+a\sum\limits_{\kappa}L_{\kappa}, \text{\ if\ }  \bar{k} \text{\ is isolated in\ } \sigma\\
L_{\sigma}+0, \text{\ if\ }  \bar{k} \text{\ is not isolated in\ } \sigma
\end{cases}$$

where the summation is over non-crossing partitions $\kappa$ obtained from $\sigma$ by merging $\bar{k}$ with any of the parts of $\sigma$.

If $i=2\bar{k}$ is even, then define 
$$L'_{\sigma}=\begin{cases}
L_{\sigma}+a\sum\limits_{\kappa}L_{\kappa}, \text{\ if\ }  \widetilde{k} \text{\ is isolated in\ } \widetilde\sigma\\
L_{\sigma}+0, \text{\ if\ }  \widetilde{k} \text{\ is not isolated in\ } \widetilde\sigma
\end{cases}$$

where the summation is over non-crossing partitions $\kappa$ obtained from $\sigma$ by merging $\widetilde{k}$ with any of the parts of $\widetilde\sigma$.

Then $u_i(a)X=\sum\limits_{\sigma}L'_{\sigma}w_{\sigma}$.

\end{lemma}

Note that the isolating operation $g_i$ is a left inverse of the merging operation in vertex $i$. Thus merging operation is injective.

Since $u_i^2(a)=0$ we have $u_i(a)=exp(a\mathfrak{u}_i)=1+a\mathfrak{u}_i$. Thus the second terms (divided by $a$) in the formulas above correspond to the changes of grove coordinates under the action of Lam algebra generator $\mathfrak{u}_i$. This observation leads to a different proof of Theorem \ref{cristal V}.

\begin{proof}[Proof of Theorem \ref{cristal V}]
Consider a concordance vector $w_{\sigma^0}$ corresponding to some fixed non-crossing partition $\sigma^0$ and a Lam algebra generator $\mathfrak{u}_i$ with odd $i=2\bar{k}-1$; the case of even $i=2\bar{k}$ is similar. Denote by $\{L'_{\sigma}|\sigma\in\mathcal{NC}_n\}$ the set of grove coordinates of the point $\mathfrak{u}_iw_{\sigma^0}$, then by Lemma \ref{changes of grove coordinates} and the discussion above we have:
\begin{equation}
\label{eq:Lsigma'}
L'_{\sigma}=\begin{cases}
\sum\limits_{\kappa}L_{\kappa}, \text{\ if\ }  \bar{k} \text{\ is isolated in\ } \sigma\\
0, \text{\ if\ }  \bar{k} \text{\ is not isolated in\ } \sigma
\end{cases}
\end{equation}

where the summation is over non-crossing partitions $\kappa$ obtained from $\sigma$ by merging $\bar{k}$ with any of the parts of $\sigma$.

By equation \eqref{eq:Lsigma'} $L'_{\sigma}\ne 0$ if and only if $\bar{k}$ is isolated vertex in $\sigma$ and there is $\kappa$ in the summation set such that $L_{\kappa}\ne 0$. Note that $L_{\sigma^0}$ is the only one nonzero grove measurement of the point defined by $w_{\sigma^0}$, therefore $L'_{\sigma}\ne 0$ if and only if $\bar{k}$ is an isolated vertex in $\sigma$ and $\kappa=\sigma^0$ appears in the summation set.

If the vertex $\bar{k}$ is isolated in $\sigma^0$, then $\kappa=\sigma^0$ will not appear in the summation sets since the vertex $\bar{k}$ cannot be isolated in any merged non-crossing partition.

If the vertex $\bar{k}$ is not isolated in $\sigma^0$, then $\kappa=\sigma^0$ will appear in exactly one summation sum. Indeed, it appears in grove measurements, defined by the point $w_{g_i{\sigma^0}}$, and it can not appear in other cases since merging is an injective operation. Thus $\mathfrak{u}_iw_{\sigma^0}=w_{g_i(\sigma^0)}$ if $\bar{k}$ is not an isolated vertex in $\sigma^0$.

\end{proof}

 The following bilinear form $\langle\cdot,\cdot\rangle$ defined on the concordance space $H$ was studied in a related context in \cite{KW 2011}, \cite{KW}. Since there are two independent structures defined on $H$, namely, this form $\langle\cdot,\cdot\rangle$ and the action of the Lam algebra, it is a natural question to understand in which sense they are compatible with each other. Proposition \ref{Invariant form under Lam algebra} gives a possible answer to this question.  
\begin{definition}
There is a bilinear form on the vector space of formal linear combinations of non-crossing partitions of $[\bar n]$, which is isomorphic to $H$:
 if $\tau,\sigma\in\mathcal{NC}_n$, $\langle\tau,\sigma\rangle$ takes the value $1$ or $0$ and is equal to $1$ if and only if the following two conditions are satisfied:

1. The number of parts of $\tau$ and $\sigma$ add up to $n+1$.

2. 
The partition on $[\bar n]$ defined by the union of the partitions $\tau$ and $\sigma$ consists of one part containing all the nodes.

For example, $\langle\bar1\bar2\bar3|\bar4,\bar2\bar4|\bar1|\bar3\rangle=1$ but $\langle\bar1\bar2|\bar3\bar4,\bar1\bar2|\bar3|\bar4\rangle=0$.

We extend the bilinear form from non-crossing partitions to concordance vectors: $\langle w_{\tau},w_{\sigma}\rangle:=\langle \tau,\sigma\rangle.$

\end{definition}

\begin{proposition} \label{invariance of g_i}
Let $\tau,\sigma\in\mathcal{NC}_n$ such that the vertex $i$ is either non-isolated or isolated in each of them simultaneously. 
Then for the form $\langle\cdot,\cdot\rangle$ the following identity holds:
$$\langle g_i\tau,\sigma\rangle=\langle\tau,g_i\sigma\rangle.$$
\end{proposition}

\begin{proof}
Suppose that $i$ is odd, otherwise change $\tau,\sigma$ to $\widetilde{\tau},\widetilde{\sigma}$ everywhere below. We first consider the case when the vertex $i$ is not isolated in neither $\tau$ nor $\sigma$. In such a case, the action of the operation $g_i$ increases the number of parts of the non-crossing partition by $1$, i.e. $|g_i\sigma|=|\sigma|+1$. The added part is the vertex $i$ by itself, which in $g_i\sigma$ forms the part by itself.
Thus 
\begin{equation}\label{eq:taugtau}
    |g_i\tau|+|\sigma|=|\tau|+1+|\sigma|=|\tau|+|g_i\sigma|,
\end{equation}
so the number of parts in $g_i\tau\cup\sigma$ is equal to $n+1$ if and only if the number of parts in $\tau\cup g_i\sigma$ is equal to $n+1$. Applying $g_i$ to $\tau$ preserves all 
parts of the partition of $[\bar n]$ defined by the union of $\tau$ and $\sigma$ except that parts which contain the vertex $i$. However, since the vertex $i$ is not isolated in neither $\tau$ nor $\sigma$, there is some part in $\tau$ containing $i$ and some other vertex;
 and the same for $\sigma$. It follows that the number of parts in $g_i\tau\cup\sigma$ is equal to one if and only if the same holds in $\tau\cup g_i \sigma$. Thus both conditions of the form $\langle\cdot,\cdot\rangle$ hold simultaneously in the left hand side of \eqref{eq:taugtau} if and only if they hold in the right hand side of \eqref{eq:taugtau}, so the form is invariant under the action of $g_i$.

In the case when the vertex $i$ is isolated in both $\tau$ and $\sigma$ their union never consists of a one part, because $i$ forms a part by itself. Thus both sides of \eqref{eq:taugtau} equal identically to zero.
\end{proof}

\begin{remark}
The assertion of Proposition \ref{invariance of g_i} is not true when $i$ is an isolated vertex of only one of two partitions: let $(\tau|\widetilde{\tau})=(137|2|46|5|8)$ and $(\sigma|\widetilde{\sigma})=(1|248|3|57|6)$. Then $\langle g_5\tau,\sigma\rangle=1$ while $\langle\tau,g_5 \sigma\rangle=0$.
\end{remark}

One can reformulate Proposition \ref{invariance of g_i} in an algebraic way as follows. Note that here we are not assuming any isolating condition on $i$ any more.
\begin{proposition}\label{Invariant form under Lam algebra}
For the form $\langle\cdot,\cdot\rangle$ the following identity holds:
$$\langle \mathfrak{u} w_{\tau},w_{\sigma}\rangle=\langle w_{\tau},\mathfrak{u}w_{\sigma}\rangle.$$
\end{proposition} 
\begin{proof}
We start with several reductions. It is enough to prove the assertion for the arbitrary commutator of generating elements. Thus by induction it is enough to prove the statement for the commutator $[\mathfrak{u}_i,\mathfrak{u}_j]$ of length $1$. 
Therefore it is enough to prove the following:
$$\langle (\mathfrak{u}_i\mathfrak{u}_j-\mathfrak{u}_j\mathfrak{u}_i) w_{\tau},w_{\sigma}\rangle=\langle w_{\tau},(\mathfrak{u}_i\mathfrak{u}_j-\mathfrak{u}_j\mathfrak{u}_i)w_{\sigma}\rangle.$$


In the case of $i$ is a non-isolated vertex in both $\tau$ and $\sigma$ the assertion follows from Theorem \ref{cristal V}:
$$\langle \mathfrak{u}_iw_{\tau},w_{\sigma}\rangle=\langle w_{g_i\tau},w_{\sigma}\rangle=\langle g_i\tau,\sigma\rangle=\langle\tau,g_i\sigma\rangle=\langle w_{\tau},w_{g_i\sigma}\rangle=\langle w_{\tau},\mathfrak{u}_iw_{\sigma}\rangle.$$
In the case when $i$ is an isolated vertex in both $\tau$ and $\sigma$ the generators $\mathfrak{u}_i$ act by $0$, so the assertion holds trivially.

It remains to consider the case when $i$ is an isolated vertex in the exactly one non-crossing partition. Without lost of generality we assume that $i$ is an isolated vertex in $\tau$ and non-isolated in $\sigma$. Then by Theorem \ref{cristal V} $\mathfrak{u}_i w_{\tau}=0$ and hence $\langle \mathfrak{u}_i w_{\tau},w_{\sigma}\rangle=0$. On the other hand, $i$ is an isolated vertex in both $\tau$ and $g_i\sigma$, so in their union $i$ forms an equivalence class by itself and thus $\langle w_{\tau},\mathfrak{u}_i w_{\sigma} \rangle=\langle \tau,g_i \sigma\rangle=0$.
\end{proof}

\begin{corollary}
\label{from Lam to sp}
The assertions of Theorem \ref{cristal V} and Proposition \ref{Invariant form under Lam algebra} are actually true for the symplectic Lie algebra $\mathfrak{sp}_{2n-2}$.
\end{corollary} 
\begin{proof}
It follows directly from the proof of Theorem \ref{Theorem 5.4} where it is proved that the generators $\mathfrak{u}_2|_V,\ldots,\mathfrak{u}_{2n-1}|_V$ generate the symplectic Lie algebra $\mathfrak{sp}_{2n-2}$.
\end{proof}

\section{Parametrizations of cactus networks}\label{sec:cactus}

\subsection{Parametrizations via effective resistance matrices}
In this section we present a parametrization of the totally nonnegative Grassmaniannain corresponding to electrical networks by the matrix of effective resistances. We will use it in order to get a precise form for the representatives of points of the totally nonnegative Grassmannian corresponding to cactus networks and to obtain an electrical description of Algorithm \ref{Algoritm_embedding}.
\begin{definition} \label{def-dual}
    Let $e(\Gamma, \omega)  \in E_n$ be an electrical network, then $e^{*}(\Gamma^{*}, \omega^{*})  \in E_n$ is the  {\it dual network} to $e$ if the following holds:
    \begin{itemize}
        \item a graph $\Gamma^{*}$ is a dual  graph to $\Gamma;$
        \item  for any two corresponding edges of $f \in E(\Gamma)$ and $ f^{*} \in E(\Gamma^{*}) $ their  conductivities related as follows $\omega(f)\omega^{*}(f^{*})=1.$  
    \end{itemize}
\end{definition}
\begin{remark} \label{remark:dual cactus network}
    The last definition can be naturally extended on cactus networks. Indeed it is enough to assume that if $\omega(f)=0$ then $\omega^{*}(f^{*})=\infty$ and vice versa. This agreement has the following combinatorial interpretation: a deletion of an edge $f$ of a graph $\Gamma$ is a contraction of the corresponded edge $f^{*}$ of a graph $\Gamma^{*}$ and vice versa. In other words, if a cactus network $e$ can be obtained from a network $e'$ by tending to zero (or infinity) conductivity of an edge $f,$ then $e^{*}$ can be obtained from a network $(e')^{*}$ by tending to infinity (or zero) conductivity of the corresponded edge $f^{*}.$ In particular, for a hollow cactus network $P(\widetilde{\sigma})=(P(\sigma))^*$. Note that for an arbitrary cactus network it is not the case, see \cite[Remark 2.7]{CGS}.
\end{remark}
The construction of the dual electrical network can be interpreted as the cyclic shift on $Gr_{\geq 0 }(n-1, 2n)$ (for $Gr_{\geq 0}(n+1, 2n)$ see \cite[Lemma 2.8]{CGS}):
\begin{lemma} \label{dual-gr}
  Let $e(\Gamma, \omega)  \in E_n$ be an electrical network  and $X(e), \ X(e^{*})$ are the points of $Gr_{\geq 0}(n-1, 2n)$ defined by the  Temperley trick for the networks $e(\Gamma, \omega) \in E_n$ and $e^{*}(\Gamma^{*}, \omega^{*}) \in E_n$ correspondingly.  Then the following holds:
    $$X(e^{*})=X(e)s^{-1},$$
    where the matrix $s$ is defined by the formula \eqref{cyclic operator s}.
    
    The same holds for points of $Gr_{\geq 0}(n+1, 2n)$ defined by the   generalized Temperley trick  for the networks $e$ and $e^{*}$. 
\end{lemma}
\begin{proof}
 Since the action of $s^{-1}$ shifts Plucker coordinates of $X(e)$ by $1:$  $$\Delta_{i_1\dots i_{n-1}}(X(e))=\Delta_{(i_1-1 \ mod \ 2n)\dots (i_{n-1}-1\ mod \ 2n)}(X(e)s^{-1}),$$ we conclude that the point $X(e)s^{-1} \in Gr_{\geq 0}(n-1, 2n)$ is defined by the network $N'$, which is  obtained from $N(\Gamma, \omega)$ by  counterclockwise shifting of a numeration of its boundary vertices by $1$ as it is shown in Fig. \ref{pic:networks}.

Let us apply the gauge transformation 
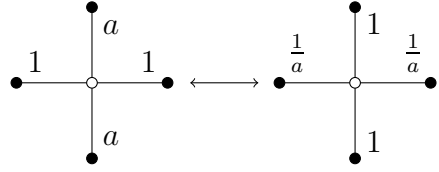
\begin{figure}[ht]
\centering
\begin{tikzpicture}

     \filldraw (-1,0) node [above] {} circle (2pt);
    \filldraw (1,0) node [left] {} circle (2pt);
    \filldraw (0,-1) node [below] {} circle (2pt);
    \filldraw (0,1) node [below] {} circle (2pt);

    \draw (-1,0) -- node [above left] {$1$}(0,0);
    \draw (0,0) -- node [above right] {$1$}(1,0);
    \draw (0,-1) -- node [below right] {$a$}(0,0);
    \draw (0,0) -- node [above right] {$a$}(0,1);

    \filldraw[fill=white] (0,0) node [below] {} circle (2pt);

 
    \draw[<->] (1.3,0) -- node {} (2.2,0);
\end{tikzpicture}
\begin{tikzpicture}

    \filldraw (-1,0) node [above] {} circle (2pt);
    \filldraw (1,0) node [left] {} circle (2pt);
    \filldraw (0,-1) node [below] {} circle (2pt);
    \filldraw (0,1) node [below] {} circle (2pt);

    \draw (-1,0) -- node [above left] {$\frac{1}{a}$}(0,0);
    \draw (0,0) -- node [above right] {$\frac{1}{a}$}(1,0);
    \draw (0,-1) -- node [below right] {$1$}(0,0);
    \draw (0,0) -- node [above right] {$1$}(0,1);

    \filldraw[fill=white] (0,0) node [below] {} circle (2pt);


\end{tikzpicture}
\caption{Gauge transformation}
    \label{pic:gauge}
\end{figure}
in each black vertex of degree $4$ to $N'$ (as it is shown in Fig. \ref{pic:gauge}) to obtain the network $N''$. Since gauge transformations do not change a point of a Grassamannian (see \cite[Lemma 4.6]{L3}), we conclude that $X(e)s^{-1}$  is defined by $N''.$ On the other hand, it is easy to see that $N''=N(\Gamma^{*}, \omega^{*}).$  That ends the proof. Note that by the continuity arguments we obtain that Lemma also holds if we assume that $e, e^{*}\in \overline{E}_n.$
\end{proof}


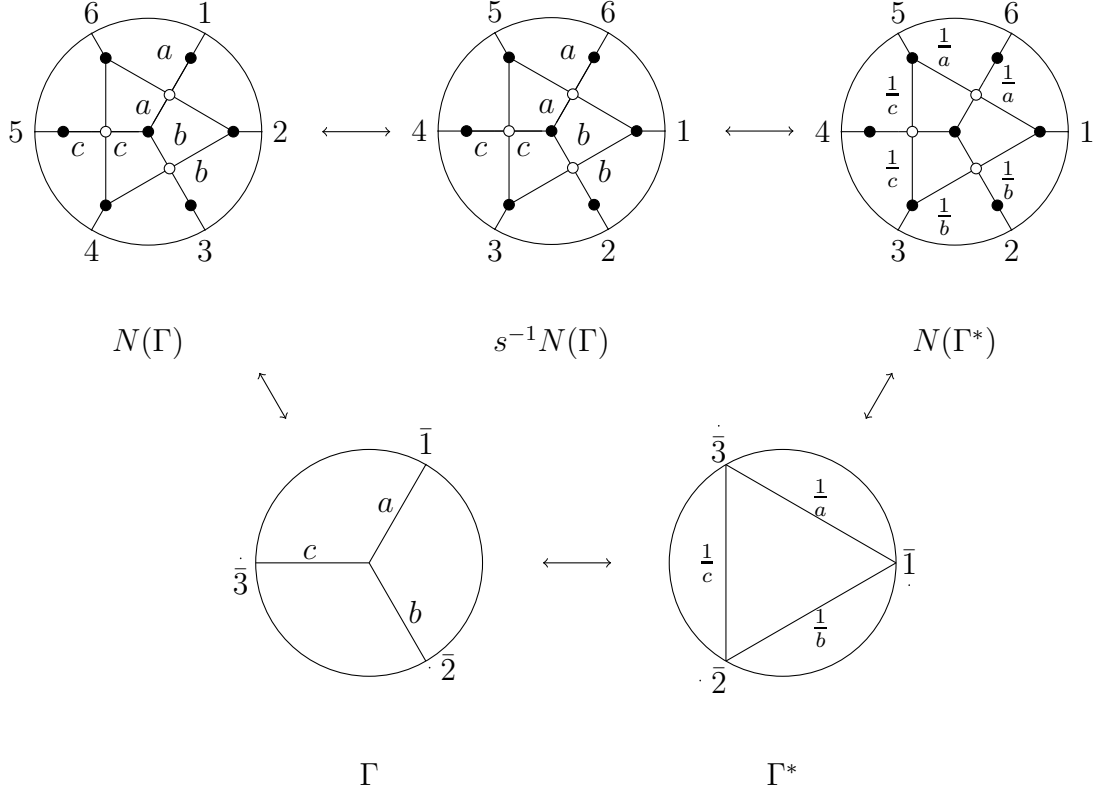
\begin{figure}
    \centering
\begin{tikzpicture}
    \draw (0,0) circle (1.5);

    \draw (0:0) -- node [left] {$a$}(60:0.75);
    \draw (0:1.125) -- node [above] {}(0:1.5);
    \draw (0:0) -- node [below] {$c$}(180:0.75);
    \draw (120:1.125) -- node [right] {}(120:1.5);
    \draw (0:0) -- node [above right] {$b$}(300:0.75);
    \draw (240:1.125) -- node [right] {}(240:1.5);
    
    \draw (60:0.75) -- node [above left] {$a$}(60:1.125);
    \draw (60:0.125) -- node [right] {}(60:1.5);
    \draw (60:0.5675) -- node [right] {}(0:1.125);
    \draw (60:0.5675) -- node [above] {}(120:1.125);

    \draw (180:0.75) -- node [below] {$c$}(180:1.125);
    \draw (180:0.125) -- node [right] {}(180:1.5);
    \draw (120:1.125) -- node [left] {} (-0.5625,0);
    \draw (-0.5625,0) -- node [left] {}(240:1.125);

    \draw (300:1.125) -- node [above right] {$b$}(300:0.75);
    \draw (300:1.125) -- node [right] {}(300:1.5);
    \draw (300:0.5625) -- node [right] {}(0:1.125);
    \draw (300:0.5625) -- node [] {}(240:1.125);


    \filldraw[fill=white] (0:1.5) node [right] {$2$} circle (0.1pt);
    \filldraw[fill=white] (60:1.5) node [above] {$1$} circle (0.1pt);
    \filldraw[fill=white] (120:1.5) node [above] {$6$} circle (0.1pt);
    \filldraw[fill=white] (180:1.5) node [left] {$5$} circle (0.1pt);
    \filldraw[fill=white] (240:1.5) node [below] {$4$} circle (0.1pt);
    \filldraw[fill=white] (300:1.5) node [below] {$3$} circle (0.1pt);

    \filldraw (0:0) node [right] {} circle (2pt);
    
    \filldraw (0:1.125) node [right] {} circle (2pt);
    \filldraw[fill=white] (60:0.5625) node [right] {} circle (2pt);
    \filldraw (120:1.125) node [right] {} circle (2pt);
    \filldraw[fill=white] (180:0.5625) node [right] {} circle (2pt);
    \filldraw (240:1.125) node [right] {} circle (2pt);
    \filldraw[fill=white] (300:0.5625) node [right] {} circle (2pt);

     \filldraw (60:1.125) node [right] {} circle (2pt);
     \filldraw (180:1.125) node [right] {} circle (2pt);
     \filldraw (300:1.125) node [right] {} circle (2pt);


    \draw (0,-2.8) node {$N(\Gamma)$};

     \draw[<->] (2.3,0) -- node {} (3.2,0);
\end{tikzpicture} 
\begin{tikzpicture}
    \draw (0,0) circle (1.5);

    \draw (0:0) -- node [left] {$a$}(60:0.75);
    \draw (0:1.125) -- node [above] {}(0:1.5);
    \draw (0:0) -- node [below] {$c$}(180:0.75);
    \draw (120:1.125) -- node [right] {}(120:1.5);
    \draw (0:0) -- node [above right] {$b$}(300:0.75);
    \draw (240:1.125) -- node [right] {}(240:1.5);
    
    \draw (60:0.75) -- node [above left] {$a$}(60:1.125);
    \draw (60:0.125) -- node [right] {}(60:1.5);
    \draw (60:0.5675) -- node [right] {}(0:1.125);
    \draw (60:0.5675) -- node [above] {}(120:1.125);

    \draw (180:0.75) -- node [below] {$c$}(180:1.125);
    \draw (180:0.125) -- node [right] {}(180:1.5);
    \draw (120:1.125) -- node [left] {} (-0.5625,0);
    \draw (-0.5625,0) -- node [left] {}(240:1.125);

    \draw (300:1.125) -- node [above right] {$b$}(300:0.75);
    \draw (300:1.125) -- node [right] {}(300:1.5);
    \draw (300:0.5625) -- node [right] {}(0:1.125);
    \draw (300:0.5625) -- node [] {}(240:1.125);


    \filldraw[fill=white] (0:1.5) node [right] {$1$} circle (0.1pt);
    \filldraw[fill=white] (60:1.5) node [above] {$6$} circle (0.1pt);
    \filldraw[fill=white] (120:1.5) node [above] {$5$} circle (0.1pt);
    \filldraw[fill=white] (180:1.5) node [left] {$4$} circle (0.1pt);
    \filldraw[fill=white] (240:1.5) node [below] {$3$} circle (0.1pt);
    \filldraw[fill=white] (300:1.5) node [below] {$2$} circle (0.1pt);

    \filldraw (0:0) node [right] {} circle (2pt);
    
    \filldraw (0:1.125) node [right] {} circle (2pt);
    \filldraw[fill=white] (60:0.5625) node [right] {} circle (2pt);
    \filldraw (120:1.125) node [right] {} circle (2pt);
    \filldraw[fill=white] (180:0.5625) node [right] {} circle (2pt);
    \filldraw (240:1.125) node [right] {} circle (2pt);
    \filldraw[fill=white] (300:0.5625) node [right] {} circle (2pt);

     \filldraw (60:1.125) node [right] {} circle (2pt);
     \filldraw (180:1.125) node [right] {} circle (2pt);
     \filldraw (300:1.125) node [right] {} circle (2pt);


    \draw (0,-2.8) node {$s^{-1}N(\Gamma)$};
    \draw[<->] (2.3,0) -- node {} (3.2,0);
\end{tikzpicture}
\begin{tikzpicture}
    \draw (0,0) circle (1.5);

    \draw (0:0) -- node [above] {}(60:0.75);
    \draw (0:1.125) -- node [above] {}(0:1.5);
    \draw (0:0) -- node [right] {}(180:0.75);
    \draw (120:1.125) -- node [right] {}(120:1.5);
    \draw (0:0) -- node [right] {}(300:0.75);
    \draw (240:1.125) -- node [right] {}(240:1.5);
    
    \draw (60:0.75) -- node [right] {}(60:1.5);
    \draw (60:0.5675) -- node [above] {$\frac{1}{a}$}(0:1.125);
    \draw (60:0.5675) -- node [above] {$\frac{1}{a}$}(120:1.125);

    \draw (180:0.75) -- node [right] {}(180:1.5);
    \draw (120:1.125) -- node [left] {$\frac{1}{c}$} (-0.5625,0);
    \draw (-0.5625,0) -- node [left] {$\frac{1}{c}$}(240:1.125);

    \draw (300:1.5) -- node [right] {}(300:0.75);
    \draw (300:0.5625) -- node [below] {$\frac{1}{b}$}(0:1.125);
    \draw (300:0.5625) -- node [below] {$\frac{1}{b}$}(240:1.125);


    \filldraw[fill=white] (0:1.5) node [right] {$1$} circle (0.1pt);
    \filldraw[fill=white] (60:1.5) node [above] {$6$} circle (0.1pt);
    \filldraw[fill=white] (120:1.5) node [above] {$5$} circle (0.1pt);
    \filldraw[fill=white] (180:1.5) node [left] {$4$} circle (0.1pt);
    \filldraw[fill=white] (240:1.5) node [below] {$3$} circle (0.1pt);
    \filldraw[fill=white] (300:1.5) node [below] {$2$} circle (0.1pt);

    \filldraw (0:0) node [right] {} circle (2pt);
    
    \filldraw (0:1.125) node [right] {} circle (2pt);
    \filldraw[fill=white] (60:0.5625) node [right] {} circle (2pt);
    \filldraw (120:1.125) node [right] {} circle (2pt);
    \filldraw[fill=white] (180:0.5625) node [right] {} circle (2pt);
    \filldraw (240:1.125) node [right] {} circle (2pt);
    \filldraw[fill=white] (300:0.5625) node [right] {} circle (2pt);

     \filldraw (60:1.125) node [right] {} circle (2pt);
     \filldraw (180:1.125) node [right] {} circle (2pt);
     \filldraw (300:1.125) node [right] {} circle (2pt);


    \draw (0,-2.8) node {$N(\Gamma^*)$};
\end{tikzpicture}
\phantom{aaa}\begin{tikzpicture}
    \draw (0,0) circle (1.5);

    \draw (0:0) -- (60:1.5);
    \draw (0:0) --  (180:1.5);
    \draw (0:0) -- node [right] {$b$}(300:1.5);
    \filldraw (300:1.6) node [right] {$\bar 2$} circle (0.1pt);
    \filldraw (60:1.5) node [above] {$\bar 1$} circle (0.1pt);
    \filldraw (178:1.7) node [below] {$\bar 3$} circle (0.1pt);
     \filldraw (170:0.8) node  {$c$};
     \filldraw (74:0.8) node  {$a$};
    \draw[<->] (120:2.2) -- (120:2.9);
    \draw (0,-2.8) node {$\Gamma$};
    \draw[<->] (2.3,0) -- node {} (3.2,0);
\end{tikzpicture}\phantom{aaa}
\begin{tikzpicture}
    \draw (0,0) circle (1.5);

     \draw (240:1.5) -- (0:1.5);
\draw (0:1.5) -- (120:1.5);
\draw (120:1.5) -- (240:1.5);
    \filldraw (235:1.9) node [right] {$\bar 2$} circle (0.1pt);
    \filldraw (-10:1.7) node [above] {$\bar 1$} circle (0.1pt);
    \filldraw (115:2.0) node [below] {$\bar 3$} circle (0.1pt);
    \filldraw (300:1.0) node  {$\frac{1}{b}$};
\filldraw (60:1.0) node  {$\frac{1}{a}$};
    \filldraw (180:1.0) node  {$\frac{1}{c}$};
     
    \draw[<->] (60:2.2) -- (60:2.9);
    \draw (0,-2.8) node {$\Gamma^*$};
\end{tikzpicture}

\caption{Action of the operator $s^{-1}$ on the network}
    \label{pic:networks}
\end{figure}

It is clear that $\Omega_n(e')$ is defined for any cactus network $e' \in \overline{E}_n,$ which is obtained from a network $e \in E_n$ by setting some of its conductivities equal to $0.$ There is a natural way to adjust Definition \ref{Omega_n}  for the explicit parameterization of the points of $Gr_{\geq 0}(n-1,2n)$ associated with   cactus networks $e'$  obtained from   a network $e \in E_n$ by setting some of its conductivities to infinity. 

\begin{definition} \label{def:eff-resist}
Let $e \in E_n$ and a set boundary potentials $\textbf{U} = (U_1, \dots , U_n)$ be such that: 
\begin{equation} \label{eq-resist}
    M_R\textbf{U} = -e_i + e_j,
\end{equation}
 where $e_k, \ k \in \{1, \dots  , n\}$ is the  standard basis of $\mathbb{R}^n.$ 
 The quantity $R_{ij}=|U_i - U_j|$ is called the {\it effective resistance}  between the boundary vertices $i$ and $j$ (in particular, $R_{ii}=0$). This difference does not depend on the choice of a solution of the equation \eqref{eq-resist}.
\end{definition}

\begin{definition} \label{OmegaR_def}
Let $e\in E_n,$ then   define a point in $Gr(n-1,2n)$ associated to $e$ as the row space of the matrix:
\begin{equation} \label{eq:omega_n,r}
  \Omega_{n,R}(e)=\left(\begin{matrix}
1 & m_{11} & 1 &  -m_{12} & 0 & m_{13} & 0 & \ldots  \\
0 & -m_{21} & 1 & m_{22} & 1 & -m_{23} & 0 & \ldots \\
0 & m_{31} & 0 & -m_{32} & 1 & m_{33} & 1 & \ldots \\
\vdots & \vdots & \vdots & \vdots & \vdots & \vdots &  \vdots & \ddots 
\end{matrix}\right),
\end{equation}
where $m_{ij}= -\frac{1}{2}(R_{i,j}+R_{i+1,j+1}-R_{i,j+1}-R_{i+1,j})$.

\end{definition}

\begin{theorem} \label{Theorem about Omegar_n}
The row space of $\Omega_{n,R}(e)$ defines the same point in the $Gr_{\geq 0}(n-1,2n)$ as the point defined by $e(\Gamma,\omega)$ under the Lam's embedding and the following holds:
$$\Delta_I(\Omega'_{n,R}(e))=\frac{\Delta_I^\bullet}{L_{123\dots n}}.$$
\end{theorem}
  By $\Omega'_{n,R}(e)$ we denote a matrix obtained from $\Omega_{n,R}(e)$ by the deletion of the first row. 
\begin{proof}
    Consider $e \in E_n$ and its dual network $e^{*}$. By Lemma \ref{dual-gr} $\Omega_n(e^{*})s$  and $\Omega_n(e)$ define the same point of the Grassmannian $Gr_{\ge 0}(n-1,2n)$.

Note that
    \begin{equation*} \label{ch-sp1}
    \Omega_n(e^{*})s
=\left(\begin{matrix}
1 & x^{*}_{11} & 1 &  -x^{*}_{12} & 0 & x^{*}_{13} & 0 & \ldots  \\
0 & -x^{*}_{21} & 1 & x^{*}_{22} & 1 & -x^{*}_{23} & 0 & \ldots \\
0 & x^{*}_{31} & 0 & -x^{*}_{32} & 1 & x^{*}_{33} & 1 & \ldots \\
\vdots & \vdots & \vdots & \vdots & \vdots & \vdots &  \vdots & \ddots 
\end{matrix}\right),
    \end{equation*}
where $x^*_{ij}$ are elements of the response matrix of dual electrical network $e^*$.
Then note that by \cite[Proposition 2.9]{KW 2011} we have:
 $$x^{*}_{ij}= -\frac{1}{2}(R_{i,j}+R_{i+1, j+1}-R_{i, j+1}-R_{i+1, j}).$$

Thus 
\begin{equation} \label{eq:relation between matrices}
\Omega_{n,R}(e)=\Omega_n(e^{*})s.
\end{equation}

Note also that by Theorem \ref{Theorem about Omega_n} $\text{rank}(\Omega_n(e^*))=n-1$ and the alternating sum of the rows of $\Omega_n(e)$ is equal to zero. Thus using \eqref{eq:relation between matrices} we conclude that the same properties are satisfied for $\Omega_{n,R}(e)$. Thus the row spaces of the matrices $\Omega_{n,R}(e)$ and $\Omega'_{n,R}(e)$ are the same.

To finish the proof it remains to notice that $\Delta_{135\dots 2n-3}(\Omega'_n(e^{*})s)=1$  and   $\Delta_{135\dots 2n-3}^\bullet=L_{123\dots n}$ (see Example \ref{bulcon-ex}).  
\end{proof}

\begin{remark} \label{thm: Speyer for cactus}
    Theorem \ref{Theorem about Omegar_n} provides a matrix which represents a cactus network in $Gr_{\ge0}(n-1,2n)$. A matrix representing a cactus network in $Gr_{\ge0}(n+1,2n)$ was constructed in \cite[Theorem 1.6 and Theorem 1.8]{CGS}.
\end{remark}

\begin{proposition}\cite[Proposition A.3]{KW 2011} \label{th: about inverse resp} 
Let $e(\Gamma, \omega)\in E_n$ be a planar circular electrical network on a connected graph $\Gamma$. Denote by $M'_R(e)$ the matrix obtained from $M_R(e)$ by deleting  the last row and column, then the following holds
 \begin{itemize}
 \item $M'_R(e)$ is invertible; 
 \item The matrix elements of its inverse are given by the formula
 \begin{equation*}
M'_R(e)^{-1}_{ij}=\begin{cases}
   R_{in},\, \text{if}\,\, i=j   \\
   \frac{1}{2}(R_{in}+R_{jn}-R_{ij}),\, \text{if}\,\, i\not = j,\\
    
      \end{cases}
    \end{equation*} 
\end{itemize}

\end{proposition}

\begin{lemma} \label{th:about overline_omega_n,r}
    Consider a network $e(\Gamma, \omega) \in E_n$ on a connected graph $\Gamma,$ then the representative of the corresponding point in $Gr_{\geq 0}(n-1,2n)$ could be chosen in the following form:
\begin{equation} \label{eq:overline_omega_n,r}
  \left(\begin{matrix}
1 & \cdot & 1 &  \cdot & 0 & \cdot & 0 & \ldots & 0 & \cdot  \\
0 & \cdot & 1 &  \cdot & 1 & \cdot & 0 & \ldots & 0 & \cdot  \\
0 & \cdot & 0 &  \cdot & 1 & \cdot & 1 & \ldots & 0 & \cdot  \\

\vdots & \vdots & \vdots & \vdots & \vdots & \vdots &  \vdots & \ddots & \vdots &  \vdots
\end{matrix}\right),
\end{equation}
where all elements in even columns depend only on $R_{i,j}$.
\end{lemma}
\begin{proof}
    Consider a $n-1\times 2n$ matrix $\Omega''_n(e)$ obtained from $\Omega_n(e)$ by deleting its last row. Consider also a $n-1\times 2n$ matrix 
    \begin{equation*}    
     \Omega'''_n(e)=\left(\begin{matrix}
x_{11} & 1 & x_{12} &  0 & x_{13} &  0 & \ldots & x_{1n} & (-1)^n  \\
x_{21} & 1 & x_{23} &  1 & x_{24} &  0 & \ldots & x_{2n} & 0  \\
x_{31} & 0 & x_{32} &  1 & x_{33} &  1 & \ldots & x_{3n} & 0  \\
\vdots & \vdots & \vdots & \vdots & \vdots & \vdots &  \vdots & \ddots & \vdots \\
\end{matrix}\right)
    \end{equation*}
    obtained from $\Omega_n''(e)$ ignoring the signs in front of $x_{ij}$.
    
    Then one could check that
    \begin{equation}  \label{eq: thm 4.21}  
     \Omega''_n(e)=D_n\Omega_n'''(e)T_n,
    \end{equation}
    where $D_n=\mathrm{diag}\bigl((-1)^{i+1}\bigr)$ is the $n-1 \times n-1$ diagonal matrix and\\  $T_n=\mathrm{diag}\bigl((-1)^{1+1}, 1, (-1)^{1+2}, 1, (-1)^{1+3}, 1, \dots \bigr)$ is the $2n \times 2n$ diagonal matrix.

    By Proposition \ref{th: about inverse resp} the matrix $M'_R(e)$ is invertible, therefore $D_n(M'_R(e))^{-1}\Omega_n'''(e)T_n$ defines the same point of $Gr_{\ge 0}(n-1,2n)$ as $ \Omega''_n(e)$.
    
    
Also note that
\begin{equation}\label{eq:almostID}
(M'_R(e))^{-1}
\Omega_n'''(e)
=\left(\begin{matrix}
1 & \cdot & 0 &  \cdot & 0 &  \cdot & \ldots & -1 & \cdot  \\
0 & \cdot & 1 &  \cdot & 0 &  \cdot & \ldots & -1 & \cdot  \\
0 & \cdot & 0 &  \cdot & 1 &  \cdot & \ldots & -1 & \cdot  \\
\vdots & \vdots & \vdots & \vdots & \vdots & \vdots &  \vdots & \ddots & \vdots 
\end{matrix}\right).
\end{equation}
Denote the $i$-th row of the RHS matrix in \eqref{eq:almostID} by $v_i$ then three matrices
\begin{equation} \label{eq: Omega bar}
\Omega''_n(e),\ \
D_n\left(\begin{matrix}
v_1 \\
v_2 \\
\cdots \\
v_{n-2} \\
v_{n-1}
\end{matrix}\right)T_n,\ \ 
D_n\left(\begin{matrix}
v_1-v_2 \\
v_2-v_3 \\
\cdots \\
v_{n-2}-v_{n-1} \\
v_{n-1}
\end{matrix}\right)T_n
\end{equation}
 define the same point of $Gr_{\geq 0}(n-1, 2n)$. Note that the rightmost matrix in \eqref{eq: Omega bar} is of the required form \eqref{eq:overline_omega_n,r}.
\end{proof}

\subsection{Electrical description of Algorithm \ref{Algoritm_embedding}} \label{subec: electrical proof of algorithm}
In this subsection we give another, less combinatorial and more "electrical", way to construct a concordance vector as an element of $\bigwedge^{n-1}V$ than described in Algorithm \ref{Algoritm_embedding}. However, it turns out that this approach leads to the same answer as Algorithm \ref{Algoritm_embedding}.

Theorem \ref{Theorem about Omegar_n} and Lemma \ref{th:about overline_omega_n,r} provides a way to choose a representative in $Gr_{\ge0}(n-1,2n)$ for $p(\sigma)$. Consider an arbitrary hollow cactus network (see Definition \ref{def:hollow} and Definition \ref{def: hollow network}) $P(\sigma) \in \overline{E}_n$ with a boundary defined by a non-crossing partition $\sigma=(\overline{B}_1, \overline{B}_2, \dots, \overline{B}_k)$ and an electrical network $e(\Gamma, \omega) \in E_n$ on a graph $\Gamma$.
 


 For any connected component $\overline{B}_i=\{\overline{j}_1, \overline{j}_2, \dots, \overline{j}_{k_i} \} \in \sigma,$ which contains at least two boundary vertices, consider $e(\Gamma_{\overline{B}_i})  \in E_{|\overline{B}_i|}$ on the connected subgraph $\Gamma_{\overline{B}_i}$ of $\Gamma.$ Consider also  the vector space $V_R(\overline{B}_i)=\mathrm{span}(v^{R}_{ \overline{j}_1\overline{j}_2}, v^{R}_{ \overline{j}_2\overline{j}_3}, \dots, v^{R}_{ \overline{j}_{k_i-1}\overline{j}_{k_i}}), $ here $v^{R}_{ \overline{j}_{l}\overline{j}_{l+1}}$ is the $l-$th row of $\Omega_n(e(\Gamma,\omega))$. Note that for all $ k \in \overline{B}_i, \ m \notin \overline{B}_i $ holds that $x_{km}=0.$

 Repeating the proof of Lemma \ref{th:about overline_omega_n,r} for the restriction on $V_R(\overline{B}_i)$ we obtain that $V_R(\overline{B}_i)$ is spanned by the rows of the following matrix:
\begin{equation}\label{eq:matrix resistanse}
\bigl(M'_R\bigl(e(\Gamma_{\overline{B}_i})\bigr)\bigr)^{-1}\left(\begin{matrix}
    
v^{R}_{ \overline{j}_1\overline{j}_2}-v^{R}_{ \overline{j}_2\overline{j}_3} \\v^{R}_{ \overline{j}_2\overline{j}_3}- v^{R}_{ \overline{j}_3\overline{j}_4} \\ \dots \\
v^{R}_{ \overline{j}_{k_i-2}\overline{j}_{k_i-1}}-v^{R}_{ \overline{j}_{k_i-1}\overline{j}_{k_i}} \\
v^{R}_{ \overline{j}_{k_i-1}\overline{j}_{k_i}}
\end{matrix}\right). 
\end{equation}
By Proposition \ref{th: about inverse resp} entries of the matrix \eqref{eq:matrix resistanse}
depends only on $\{R_{km}| k, m \in \overline{B}_i\}.$

Thus tending all $R_{km} \to 0$ we have
$$V(\overline{B}_i):=\lim_{R_{km} \to 0}V_R(\overline{B}_i) \subset p(\sigma),$$
and $V(\overline{B}_i)$ is the span of the vectors $v_{\overline{j}_{l}\overline{j}_{l+1}}$ which are the vector-rows of the matrix obtained from the matrix \eqref{eq:matrix resistanse} after tending all $R_{km} \to 0$.


From the proof of Lemma \ref{th:about overline_omega_n,r} we also have that:
\begin{equation}\label{eq:vectors}
 v_{\overline{j}_{l}\overline{j}_{l+1}}=(0, 0, \dots,0, *, 0, \dots, 0, *, 0, \dots, 0),
\end{equation}
where the first non-zero entry $*$ is on the $(2\overline{j}_{l}-1)$-th position, the second non-zero entry $*$ is on the $(2\overline{j}_{l+1}-1)$-th position and both of them have the absolute value $1.$ Multiplying $v_{\overline{j}_{l}\overline{j}_{l+1}}$ by $\pm 1$ we are able to set the first $*$ equal to $1$. Then the second $*$ must be equal to  $(-1)^{\mathrm{sgn}(\overline{j}_{l+1})}$, where $\mathrm{sgn}(\overline{j}_{l+1})$ is uniquely defined by the condition that $v_{\overline{j}_{l}\overline{j}_{l+1}} \in V$. Note that $\mathrm{sgn}(\overline{j}_{l+1})$ coincides with the definition of the sign in the third step of Algorithm \ref{Algoritm_embedding}.  

By repeating above arguments for each $\overline{B}_i \in \sigma: \  |\overline{B}_i|>1$ we obtain that: 
$$\bigoplus \limits_{\overline{B}_i \in \sigma: \ |\overline{B}_i|>1}V(\overline{B}_i) \subset p(\sigma).$$

Now, consider the dual to $P(\sigma)$ hollow cactus network $P(\widetilde{\sigma}) \in \overline{E}_n$ with the boundary defined by the non-crossing partition $\widetilde{\sigma}=({\widetilde{B}_1, \widetilde{B}_2, \dots, \widetilde{B}_l })$  dual to $\sigma$. Then let us repeat the above arguments for each connected component $\widetilde{B}_{i} \in \widetilde{\sigma}: \ |\widetilde{B}_i|>1$, then we have: 
$\bigoplus \limits_{\widetilde{B}_i \in \widetilde{\sigma}: \ |\widetilde{B}_i|>1}V(\widetilde{B}_{i}) \subset p(\widetilde\sigma).$ 
Note that by Remark \ref{remark:dual cactus network} we have $P(\widetilde{\sigma})=(P(\sigma))^*$. 
Let us denote by $(p(\sigma))^*$ the point of $Gr_{\ge0}(n-1,2n)$ corresponding to $(P(\sigma))^*$. Thus we have $\bigoplus \limits_{\widetilde{B}_i \in \widetilde{\sigma}: \ |\widetilde{B}_i|>1}V(\widetilde{B}_{i}) \subset (p(\sigma))^*.$ Identify the space $V(\widetilde{B}_i)$ with a matrix representing it: it is a matrix whose rows define analogously to vectors $v_{\bar{j_l}\bar{j}_{l+1}}$ with the substitution $\bar{i}$ to $\widetilde{i}$ in the description below \eqref{eq:vectors}. Then we obtain: $
\bigoplus \limits_{\widetilde{B}_i \in \widetilde{\sigma}: \ |\widetilde{B}_i|>1}V(\widetilde{B}_{i})s \subset (p(\sigma))^*s.
$ 
Then by Lemma \ref{dual-gr} we have:
$
\bigoplus \limits_{\widetilde{B}_i \in \widetilde{\sigma}: \ |\widetilde{B}_i|>1}V(\widetilde{B}_{i})s \subset p(\sigma).
$

And also for any $\widetilde{B}_{i}=\{\widetilde{j}_1, \widetilde{j}_2, \dots, \widetilde{j}_{k_i} \},\  |\widetilde{B}_{i}|=k_i > 1$  we have that 
$V(\widetilde{B}_i)s=\mathrm{span}(v_{\widetilde{j}_1\widetilde{j}_2}, v_{\widetilde{j}_2\widetilde{j}_3}, \dots, v_{\widetilde{j}_{k_i-1}\widetilde{j}_{k_i}}),$ where
$$v_{\widetilde{j}_{l}\widetilde{j}_{l+1}}=(0, 0, \dots,0, 1, 0, \dots, 0, (-1)^{\mathrm{sgn}(\widetilde{j}_{l+1})}, 0, \dots, 0).$$
Here again the first non-zero entry is  on the $2\widetilde{j}_{l}$-th position, the last non-zero entry  is on the $2\widetilde{j}_{l+1}$-th position and $\mathrm{sgn}(\widetilde{j}_{l+1})$ is again uniquely defined by the condition that $v_{\widetilde{j}_{l}\widetilde{j}_{l+1}} \in V.$

Since $ \sum \limits_{\overline{B}_i \in  \overline{\sigma}: \ |\overline{B}_i|>1 } \mathrm{dim}(V(\overline{B}_i)) + \sum \limits_{\widetilde{B}_i \in  \widetilde{\sigma}: \ |\widetilde{B}_i|>1 } \mathrm{dim}(V(\widetilde{B}_i)) =n-1$ (see the proof of Algorithm \ref{Algoritm_embedding}, part "Sign's arrangement") we have that:
\begin{equation} \label{formula:exp-cac}
p(\sigma)=\bigoplus\limits_{\overline{B}_i \in \sigma: \ |\overline{B}_i|>1 } V(\overline{B}_i) \bigoplus\limits_{\widetilde{B}_i \in \sigma: \ |\widetilde{B}_i|>1 } V(\widetilde{B}_i).     
\end{equation}
Or, equivalently, using the definitions of $V(\overline{B}_i)$ and $V(\widetilde{B}_i)$ we have that the point $p(\sigma)$ could be represented by a $n+1\times 2n$ matrix $M_\sigma$ whose rows are the following set of vectors:
\begin{equation} \label{eq:physical representative of hollow cactus}
\bigsqcup\limits_{\overline{B}_i=\{\overline{j}_1,\ldots,\overline{j}_{k_i}\}\in\sigma:\ |\overline{B}_i|>1}\{v_{\overline{j}_1\overline{j}_2}, v_{\overline{j}_2\overline{j}_3}, \dots, v_{\overline{j}_{k_i-1}\overline{j}_{k_i}}\} \ \bigsqcup\limits_{\widetilde{B}_i=\{\widetilde{j}_1,\ldots,\widetilde{j}_{k_i}\}\in\widetilde{\sigma}:\ |\widetilde{B}_i|>1}\{v_{\widetilde{j}_1\widetilde{j}_2}, v_{\widetilde{j}_2\widetilde{j}_3}, \dots, v_{\widetilde{j}_{k_i-1}\widetilde{j}_{k_i}}\}.
\end{equation}
\begin{proposition} \label{prop:two algorithms coinside}
    The matrix $M_\sigma$ defined by \eqref{eq:physical representative of hollow cactus} coincides with the matrix obtained by writing wedge-factors of a pure wedge product provided by the third step of Algorithm \ref{Algoritm_embedding} applied to $\sigma$ as the rows. 
\end{proposition}
\begin{proof}
    Note that by Proposition \ref{empty cactus maps to p} $w_\sigma$ is a point corresponding to $p(\sigma)$ under the Plucker embedding, the proof is based on the straightforward comparison of the Algorithm \ref{Algoritm_embedding} and the construction of the matrix $M_\sigma$.
\end{proof}

In other words, Proposition \ref{prop:two algorithms coinside} says that the procedure above describes a different way to obtain the same representative for a point of a Grassmannian corresponding to a hollow cactus network as that of provided by Algorithm \ref{Algoritm_embedding}. We call it "electrical" since it is based on an analysis of an effective resistance matrix.

\section{Appendix} \label{sec: appendix}

\subsection{Postnikov's networks and bipartite networks} \label{pbn-ap}

Here we recall a definition of Postnikov's and bipartite networks and their connection with the totally nonnegative Grassmannians, which we have used in Section \ref{th-comb}. 

We now introduce the dual generalized  Temperley's trick. Note that Definition \ref{temp_gen} might be obtained from the definition of the  generalized  Temperley's trick $N(\Gamma, \omega)$ \cite[Section 5.1]{L} just by inverting the colors.     
\begin{definition}  \label{temp_gen}
For each $e\in E_n$ we  define a planar bipartite network $N^d(\Gamma, \omega)$ (here $d$ stands for "dual") embedded into the disk (see Fig. \ref{dual Temperley's trick}). The boundary vertices of $N^d(\Gamma, \omega)$ are as follows:
\begin{itemize}
    \item if $\Gamma$ has boundary vertices $\{\bar 1, \bar 2, . . . , \bar n\}$, then $N^d(\Gamma, \omega)$ will have black boundary vertices $\{1, 2, . . . , 2n\}$, where boundary vertex $\bar i$ is identified with $2i-1$ and vertex $2i$ can be identified with the vertex $\tilde i$ used to label dual non-crossing partitions;
    \item each boundary vertices has the degree equal to  $1$.
\end{itemize}
The interior vertices of $N^d(\Gamma, \omega)$ are defined as follows:
\begin{itemize}
    \item we have a white interior vertex $b_v$ for each interior vertex $v$ of $\Gamma$;
    \item a white interior vertex $b_F$ for each interior face $F$ of $\Gamma$;
    \item a black interior vertex $w_e$ placed at the midpoint of each interior edge $e$ of $\Gamma;$
    \item  for each boundary vertex $\bar i$, we have a white interior vertex $b_i$.
\end{itemize}  
The edges of $N^d(\Gamma, \omega)$ are defined as follows: 
\begin{itemize}
\item  if $v$ is a vertex of an edge $e$ in $\Gamma$, then $b_v$ and $w_e$ are joined, and the weight of this edge is equal
to the weight $\omega(e)$ of $e$ in $\Gamma$, 
\item  if $e$ borders $F$, then $w_e$ is joined to $b_F$ by an edge with weight $1$, 
\item  the vertex $b_i$ is joined by an edge with weight $1$ to the boundary vertex $2i - 1$ in $N^d(\Gamma, \omega)$, and $b_i$
is also joined by an edge with weight 1 to $w_e$ for any edge $e$ incident
to $\bar i$ in $\Gamma$, 
\item  even boundary vertices $2i$ in $N^d(\Gamma, \omega)$ are joined by an edge with weight $1$ to the face vertex $w_F$ of the face $F$ that they lie in.
\end{itemize}
\end{definition}

\begin{figure}
    \centering
\begin{tikzpicture}
    \draw (0,0) circle (2);

    \draw (-120:2) -- node [above] {$b$}(0:2);
    \draw (-240:2) -- node [right] {$c$}(-120:2);
    \draw (0:2) -- node [right] {$a$}(-240:2);
    \filldraw (0:2) node [right] {$\bar 2$} circle (0.1pt);
    \filldraw (120:2) node [above] {$\bar 1$} circle (0.1pt);
    \filldraw (240:2) node [below] {$\bar 3$} circle (0.1pt);

    \draw (0,-2.8) node {$\Gamma$};
\end{tikzpicture}\phantom{a}
\begin{tikzpicture}
    \draw (0,0) circle (2);

    \draw (0:0) -- node [above] {}(60:1);
    \draw (0:1.5) -- node [above] {}(0:2);
    \draw (0:0) -- node [right] {}(180:1);
    \draw (120:1.5) -- node [right] {}(120:2);
    \draw (0:0) -- node [right] {}(300:1);
    \draw (240:1.5) -- node [right] {}(240:2);
    
    \draw (60:1) -- node [right] {}(60:2);
    \draw (60:0.75) -- node [right] {$a$}(0:1.5);
    \draw (60:0.75) -- node [above] {$a$}(120:1.5);

    \draw (180:1) -- node [right] {}(180:2);
    \draw (120:1.5) -- node [left] {$c$} (-0.75,0);
    \draw (-0.75,0) -- node [left] {$c$}(240:1.5);

    \draw (300:2) -- node [right] {}(300:1);
    \draw (300:0.75) -- node [right] {$b$}(0:1.5);
    \draw (300:0.75) -- node [below] {$b$}(240:1.5);


    \filldraw[fill=white] (0:2) node [right] {$3$} circle (0.1pt);
    \filldraw[fill=white] (60:2) node [above] {$2$} circle (0.1pt);
    \filldraw[fill=white] (120:2) node [above] {$1$} circle (0.1pt);
    \filldraw[fill=white] (180:2) node [left] {$6$} circle (0.1pt);
    \filldraw[fill=white] (240:2) node [below] {$5$} circle (0.1pt);
    \filldraw[fill=white] (300:2) node [below] {$4$} circle (0.1pt);

    \filldraw (0:0) node [right] {} circle (2pt);
    
    \filldraw (0:1.5) node [right] {} circle (2pt);
    \filldraw[fill=white] (60:0.75) node [right] {} circle (2pt);
    \filldraw (120:1.5) node [right] {} circle (2pt);
    \filldraw[fill=white] (180:0.75) node [right] {} circle (2pt);
    \filldraw (240:1.5) node [right] {} circle (2pt);
    \filldraw[fill=white] (300:0.75) node [right] {} circle (2pt);

     \filldraw (60:1.5) node [right] {} circle (2pt);
     \filldraw (180:1.5) node [right] {} circle (2pt);
     \filldraw (300:1.5) node [right] {} circle (2pt);


    \draw (0,-2.8) node {$N(\Gamma)$};
\end{tikzpicture}\phantom{a}
\begin{tikzpicture}
    \draw (0,0) circle (2);

    \draw (0:0) -- node [above] {}(60:1);
    \draw (0:1.5) -- node [above] {}(0:2);
    \draw (0:0) -- node [right] {}(180:1);
    \draw (120:1.5) -- node [right] {}(120:2);
    \draw (0:0) -- node [right] {}(300:1);
    \draw (240:1.5) -- node [right] {}(240:2);
    
    \draw (60:1) -- node [right] {}(60:2);
    \draw (60:0.75) -- node [right] {$a$}(0:1.5);
    \draw (60:0.75) -- node [above] {$a$}(120:1.5);

    \draw (180:1) -- node [right] {}(180:2);
    \draw (120:1.5) -- node [left] {$c$} (-0.75,0);
    \draw (-0.75,0) -- node [left] {$c$}(240:1.5);

    \draw (300:2) -- node [right] {}(300:1);
    \draw (300:0.75) -- node [right] {$b$}(0:1.5);
    \draw (300:0.75) -- node [below] {$b$}(240:1.5);


    \filldraw (0:2) node [right] {$3$} circle (0.1pt);
    \filldraw (60:2) node [above] {$2$} circle (0.1pt);
    \filldraw (120:2) node [above] {$1$} circle (0.1pt);
    \filldraw (180:2) node [left] {$6$} circle (0.1pt);
    \filldraw (240:2) node [below] {$5$} circle (0.1pt);
    \filldraw (300:2) node [below] {$4$} circle (0.1pt);

    \filldraw[fill=white] (0:0) node [right] {} circle (2pt);
    
    \filldraw[fill=white] (0:1.5) node [right] {} circle (2pt);
    \filldraw (60:0.75) node [right] {} circle (2pt);
    \filldraw[fill=white] (120:1.5) node [right] {} circle (2pt);
    \filldraw (180:0.75) node [right] {} circle (2pt);
    \filldraw[fill=white] (240:1.5) node [right] {} circle (2pt);
    \filldraw (300:0.75) node [right] {} circle (2pt);

     \filldraw[fill=white] (60:1.5) node [right] {} circle (2pt);
     \filldraw[fill=white] (180:1.5) node [right] {} circle (2pt);
     \filldraw[fill=white] (300:1.5) node [right] {} circle (2pt);


    \draw (0,-2.8) node {$N^d(\Gamma)$};
\end{tikzpicture}
\caption{An electrical network $e(\Gamma,\omega)$, associated to it bipartite networks $N(\Gamma,\omega)$ and $N^d(\Gamma,\omega)$ obtained from $e$ by the usual and the dual generalized Temperley trick.}
    \label{dual Temperley's trick}
\end{figure}
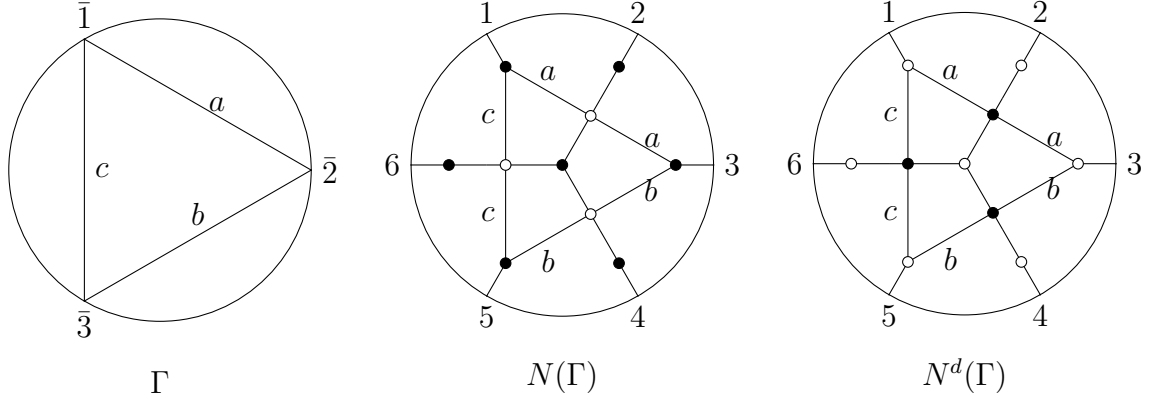

By each bipartite network $N^{d}(\Gamma, \omega)$ we can naturally construct a Postnikov network:
\begin{definition} \label{postn-net}
   Consider an electrical network $e(\Gamma, \omega)$ and associated with it  bipartite network $N^{d}(\Gamma, \omega), $ then a Postnikov network  $PN(\Gamma,\omega',O_2)$ is defined as follows:
   \begin{itemize}
   \item graphs  $N^{d}(\Gamma), PN(\Gamma) $ are the same;
   \item $O_2$ is the perfect edge  orientation, it means that each internal black vertex of $ PN(\Gamma) $  has exactly one outgoing edge and each internal white vertex $ PN(\Gamma) $ has exactly one incoming edge;
   \item  $O_2$ has the set of sources (i.e. the set of all external vertices of $ PN(\Gamma) $ with outgoing orientations)  $\{1,2 , 3, 5, \dots, 2n-1\};$ 
    \item for any $e \in E(PN(\Gamma))$ in the network $PN(\Gamma,\omega',O_2)$ connecting the white and the black vertices the weight $\omega(e)$ is the same as it is in the network $N^{d}(\Gamma, \omega); $
    \item  for any $e \in E(PN(\Gamma))$ in the network $PN(\Gamma,\omega',O_2)$ connecting the black and the white vertices the weight $\omega(e)$ is reciprocal of the weight of the same edge in the network $N^{d}(\Gamma, \omega). $
   \end{itemize}
\end{definition}

A Postnikov network  $PN(\Gamma,\omega',O_2)$ defines an extended matrix of boundary measurements:
\begin{definition} \label{extend}
Consider a Postnikov network $PN(\Gamma, \omega', O_2)$ and denote by $M_{i_rj}$ the 
sum of weights of the directed paths from $i_r \in O_2$ to an external vertex labeled by $j$: 
$$M_{i_rj} := \sum \limits_{p:\ i_r \to j } 
(-1)^{\mathrm{wind}(p)}wt(p),$$ 
where $wt(p)$ is the product of all edges belonging to $p$, and $\mathrm{wind}(p)$ is the winding index of a path $p$,  see \cite[Section 4]{Pos} for more details. 

The extended matrix of boundary measurements $M_B=(a_{rj})$ for a network  $PN(\Gamma, \omega', O_2)$ is an $n+1 \times 2n$ matrix, which is defined as follows:
\begin{itemize}
    \item its submatrix formed by the columns labeling the sources is the identity matrix;
    \item otherwise $a_{rj} = (-1)^sM_{i_rj},$ 
where $s$ is the number of elements from  $ O_2$ between the source labeled $i_r \in O_2$ and the sink labeled $j$. 
\end{itemize}
  Due to the combinatorial sense of entries of $M_B$ we will use the following notation for them: 
  $$a_{i_r \to j}:=a_{rj}.$$
\end{definition}
\begin{example}
    Consider an electrical network $e(\Gamma, \omega)$ as it is shown in Fig.\ref{fig:Postnikov network}, then an extended matrix of boundary measurements $M_B$  associated with this network   has the following form:
\begin{equation*} 
    M_B=
    \left(\begin{matrix}
1 & 0 & 0 & a & 0 & -(a+c)   \\
0 & 1 & 0 & -1 & 0 &  1 \\
0 & 0 & 1 & b+a & 0 &  -a \\
0 & 0 & 0 & b & 1 & c  
\end{matrix}
\right).
\end{equation*} 
 For instance, $a_{44}=a_{5 \to 4}=(-1)^0b.$   
\end{example}
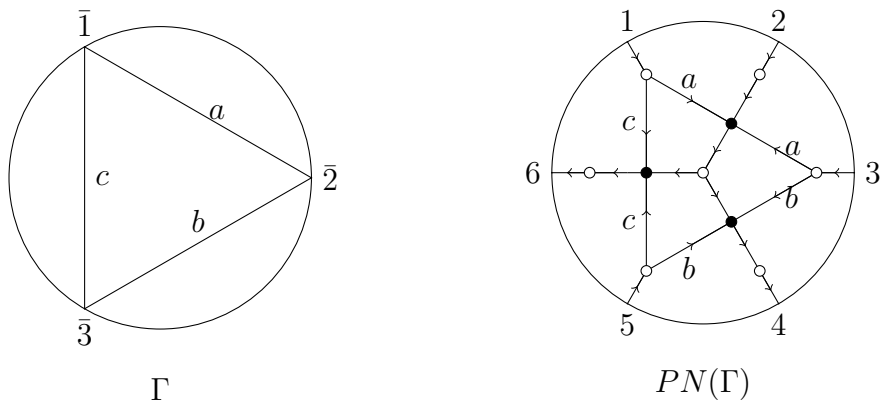
\begin{figure}
    \centering
\begin{tikzpicture}
    \draw (0,0) circle (2);

    \draw (-120:2) -- node [above] {$b$}(0:2);
    \draw (-240:2) -- node [right] {$c$}(-120:2);
    \draw (0:2) -- node [right] {$a$}(-240:2);
    \filldraw (0:2) node [right] {$\bar 2$} circle (0.1pt);
    \filldraw (120:2) node [above] {$\bar 1$} circle (0.1pt);
    \filldraw (240:2) node [below] {$\bar 3$} circle (0.1pt);

    \draw (0,-2.8) node {$\Gamma$};
\end{tikzpicture}\phantom{aaaaaaaaaa}
\begin{tikzpicture}
    \draw (0,0) circle (2);

    \draw (0:0) -- node [above] {}(60:1);
    \draw (0:1.5) -- node [above] {}(0:2);
    \draw (0:0) -- node [right] {}(180:1);
    \draw (120:1.5) -- node [right] {}(120:2);
    \draw (0:0) -- node [right] {}(300:1);
    \draw (240:1.5) -- node [right] {}(240:2);
    
    \draw (60:1) -- node [right] {}(60:2);
    \draw (60:0.75) -- node [right] {$a$}(0:1.5);
    \draw (60:0.75) -- node [above] {$a$}(120:1.5);

    \draw (180:1) -- node [right] {}(180:2);
    \draw (120:1.5) -- node [left] {$c$} (-0.75,0);
    \draw (-0.75,0) -- node [left] {$c$}(240:1.5);

    \draw (300:2) -- node [right] {}(300:1);
    \draw (300:0.75) -- node [right] {$b$}(0:1.5);
    \draw (300:0.75) -- node [below] {$b$}(240:1.5);

     \draw[-<] (0:1.5) -- node [above] {}(0:1.8);
     \draw[->] (0:1.5) -- +(150:0.649);
     \draw[->] (0:1.5) -- +(210:0.649);

     \draw[-<] (60:0.75) -- +(150:0.649);
     \draw[-<] (300:0.75) -- +(210:0.649);

     \draw[-<] (0:0) -- (60:0.375);
     \draw[->] (0:0) -- (180:0.375);
     \draw[->] (0:0) -- (300:0.375);

     \draw[->] (300:0.75) -- (300:1.125);
     \draw[->] (300:1.5) -- (300:1.8);

     \draw[-<] (-0.75,0) -- (-0.75,0.5625);
     \draw[-<] (-0.75,0) -- (-0.75,-0.5625);

     \draw[-<] (240:1.5) -- (240:1.8);
     \draw[->] (180:1.5) -- (180:1.8);
     \draw[-<] (120:1.5) -- (120:1.8);
     \draw[-<] (60:1.5) -- (60:1.8);

     \draw[-<] (180:1.5) -- (180:1.125);
     \draw[->] (60:1.5) -- (60:1.125);
     

    \filldraw (0:2) node [right] {$3$} circle (0.1pt);
    \filldraw (60:2) node [above] {$2$} circle (0.1pt);
    \filldraw (120:2) node [above] {$1$} circle (0.1pt);
    \filldraw (180:2) node [left] {$6$} circle (0.1pt);
    \filldraw (240:2) node [below] {$5$} circle (0.1pt);
    \filldraw (300:2) node [below] {$4$} circle (0.1pt);

    \filldraw[fill=white] (0:0) node [right] {} circle (2pt);
    
    \filldraw[fill=white] (0:1.5) node [right] {} circle (2pt);
    \filldraw (60:0.75) node [right] {} circle (2pt);
    \filldraw[fill=white] (120:1.5) node [right] {} circle (2pt);
    \filldraw (180:0.75) node [right] {} circle (2pt);
    \filldraw[fill=white] (240:1.5) node [right] {} circle (2pt);
    \filldraw (300:0.75) node [right] {} circle (2pt);

     \filldraw[fill=white] (60:1.5) node [right] {} circle (2pt);
     \filldraw[fill=white] (180:1.5) node [right] {} circle (2pt);
     \filldraw[fill=white] (300:1.5) node [right] {} circle (2pt);


    \draw (0,-2.8) node {$PN(\Gamma)$};
\end{tikzpicture}\phantom{a}

\caption{An electrical network $e(\Gamma,\omega)$ and the correspondent Postnikov's network $PN(\Gamma)$.  }
    \label{fig:Postnikov network}
\end{figure}

In fact all maximal minors of an extended matrix of boundary measurements $M_B$ of a network  $PN(\Gamma, \omega', O_2)$ are nonnegative, which leads to the following theorem:
\begin{theorem} \cite[Corollary 5.4]{Pos} \label{m-pos}
    A network $PN(\Gamma, \omega', O_2)$ defines the point of $Gr_{\geq 0}(n-1, 2n).$ 
\end{theorem}

\end{document}